\documentclass[11pt,reqno]{amsart}
\date{1 Oct 2019}
\title{Cluster exchange groupoids and framed quadratic differentials}
\author{Alastair King}
\address{AK:
  Mathematical Sciences,
  University of Bath, Bath BA2 7AY, U.K.}
\email{a.d.king@bath.ac.uk}

\author{Yu Qiu}
\address{YQ:
Yau Mathematical Sciences Center,
Tsinghua University,
Beijing, China}
\email{yu.qiu@bath.edu}
\usepackage[a4paper]{geometry} 

\usepackage{float}

\usepackage[colorlinks, linkcolor=blue!50,anchorcolor=Periwinkle,
    citecolor=blue!72,urlcolor=cyan, bookmarksopen,bookmarksdepth=2]{hyperref}
\usepackage[usenames,dvipsnames]{xcolor}
\usepackage{enumitem}
\usepackage{subfigure}
\usepackage{bookmark}
\usepackage{url}
\usepackage{ifthen}

\usepackage{array}   
\newcolumntype{L}{>{$}l<{$}} 

\usepackage{tikz}\usetikzlibrary{matrix}

\usetikzlibrary{matrix, calc, arrows,
decorations.pathreplacing, decorations.markings, decorations.pathmorphing,
backgrounds,fit,positioning,shapes.symbols,chains,shadings,fadings}

\tikzset{->-/.style={decoration={  markings,  mark=at position #1 with
    {\arrow{>}}},postaction={decorate}}}
\tikzset{-<-/.style={decoration={  markings,  mark=at position #1 with
    {\arrow{<}}},postaction={decorate}}}

\usepackage{amsfonts}
\usepackage{amsmath,amssymb,amsthm,amsxtra}

\theoremstyle{plain}
\newtheorem{theorem}{Theorem}[section]
\newtheorem{lemma}[theorem]{Lemma}
\newtheorem{corollary}[theorem]{Corollary}
\newtheorem{proposition}[theorem]{Proposition}

\newtheorem{condition}[theorem]{Condition}

\theoremstyle{definition}
\newtheorem{definition}[theorem]{Definition}
\newtheorem{example}[theorem]{Example}
\newtheorem{remark}[theorem]{Remark}

\numberwithin{equation}{section}
\numberwithin{figure}{section}

\newcommand{\Note}[1]{\textcolor{black}{#1}}
\newcommand{\note}[1]{\textcolor{black}{#1}}
\newcommand{\AKedit}[1]{\textcolor{black}{#1}}
\newcommand{\danger}[1]{\textcolor{black}{#1}}

\newcommand\hua{\mathcal}

\newcommand\ZZ{\mathbb{Z}}

\newcommand\RR{\mathbb{R}}
\newcommand\CC{\mathbb{C}}
\newcommand{\mai}{\mathbf{i}} 
\newcommand\bfc{\mathbf{C}}
\newcommand\<{\langle}
\renewcommand\>{\rangle}

\newcommand\jiantou{edge[->,>=stealth]}

\newcommand{\textfrac}[2]{\textstyle{\frac{#1}{#2}}}
\renewcommand{\setminus}{\smallsetminus}
\renewcommand{\emptyset}{\varnothing}
\newcommand{\isom}{\cong}

\newcommand{\rank}{\operatorname{rank}}

\newcommand{\Imgy}{\operatorname{Im}} 

\newcommand\Sim{\operatorname{Sim}}
\newcommand\Hom{\operatorname{Hom}}
\newcommand\End{\operatorname{End}}
\newcommand\Ext{\operatorname{Ext}}

\newcommand\diff{\operatorname{d}} 
\newcommand\Br{\operatorname{Br}} 
\newcommand\Crel{\operatorname{Co}} 
\newcommand\Brel{\operatorname{Br}}  
\newcommand\CBr{\operatorname{CT}} 
\newcommand\Grot{\operatorname{K}} 

\newcommand{\h}{\operatorname{\hua{H}}} 
\newcommand{\C}{\operatorname{\hua{C}}} 
\newcommand{\D}{\operatorname{\hua{D}}} 
\newcommand{\per}{\operatorname{per}} 

\newcommand\Aut{\operatorname{Aut}} 
\newcommand\Autp{\Aut^\circ} 
\newcommand\Stab{\operatorname{Stab}} 
\newcommand\Stap{\Stab^\circ} 

\newcommand{\EG}{\operatorname{EG}} 
\newcommand{\EGp}{\EG^\circ}       
\newcommand{\ST}{\operatorname{ST}}  
\newcommand{\STp}{\ST^\circ}  

\newcommand{\EGT}{\EG^\T}
\newcommand{\uEG}{\underline{\EG}} 
\newcommand{\CEG}{\operatorname{CEG}} 
\newcommand{\uCEG}{\underline{\CEG}} 

\newcommand\egp{\operatorname{\hua{EG}}^{\circ}}
\newcommand\egt{\operatorname{\hua{EG}}^{\T}}
\newcommand\eg{\operatorname{\hua{EG}}}
\newcommand\ceg{\operatorname{\hua{CEG}}}
\newcommand\uceg{\underline{\operatorname{\hua{CEG}}}}

\renewcommand{\k}{\mathbf{k}}

\newcommand{\tilt}[3]{{#1}^{#2}_{#3}}
\newcommand{\Cone}{\operatorname{Cone}}
\newcommand\Sph{\operatorname{Sph}}

\newcommand{\numarc}{n}
\newcommand{\numtri}{\aleph}

\newcommand{\Tri}{\bigtriangleup}

\newcommand\surf{\mathbf{S}}  
\newcommand\surfo{{\mathbf{S}}_\Tri}  

\newcommand\Diff{\operatorname{Diff}} 
\newcommand{\MCG}{\operatorname{MCG}} 
\newcommand{\BT}{\operatorname{BT}}  

\newcommand\Bt[1]{\operatorname{B}_{#1}}

\newcommand\RT{{T}} 
\newcommand\T{\mathbb{T}} 
\newcommand\M{\mathbf{M}} 

\newcommand{\Quad}{\operatorname{Quad}}
\newcommand{\FQuad}[2]{\operatorname{FQuad}^{#1}(#2)}

\newcommand{\Zer}{\operatorname{Zero}}
\newcommand{\Pol}{\operatorname{Pol}}
\newcommand{\Crit}{\operatorname{Crit}}
\newcommand{\cA}{\operatorname{CA}}
\newcommand{\SBr}{\operatorname{SBr}}

\newcommand{\UHP}{\mathbf{H}} 
\newcommand{\skel}{\wp} 
\newcommand{\cub}{\operatorname{U}} 

\newcommand\xx{\mathbf{X}} 
\newcommand\surp{\xx^\circ}
\newcommand\dd{d} 

\newcommand\conj{\operatorname{ad}}

\newcommand{\pha}{\varphi} 

\newcommand{\slicing}{\hua{P}}
\newcommand{\tstruc}{\hua{P}}
\newcommand{\torsion}{\hua{T}}
\newcommand{\torfree}{\hua{F}}
\newcommand{\twi}{\varphi} 
\newcommand{\Qgrad}{\overline{Q}}

\newcommand\class{\mathfrak{Q}}

\begin{document}
\begin{abstract}
  We introduce the cluster exchange groupoid associated to a non-degenerate quiver with potential,
  as an enhancement of the cluster exchange graph.
  In the case that arises from an (unpunctured) marked surface,
  where the exchange graph is modelled on the graph of triangulations of the marked surface,
  we show that the universal cover of this groupoid can be constructed using the covering graph
  of triangulations of the surface with extra decorations.

  This covering graph is a skeleton for a space of suitably framed quadratic differentials on the surface,
  which in turn models the space of Bridgeland stability conditions for the 3-Calabi-Yau category
  associated to the marked surface.
  By showing that the relations in the covering groupoid are homotopically trivial
  when interpreted as loops in the space of stability conditions, we show that this space is simply connected.

\bigskip\noindent
\emph{Key words:} cluster exchange groupoid,
braid group, quiver with potential,
quadratic differential, stability condition
\end{abstract}
\maketitle
\tableofcontents\addtocontents{toc}{\setcounter{tocdepth}{1}}

\section{Introduction}
This paper is the last in a series on decorated marked surfaces (\cite{QQ,QQ2,QZ2,BQZ,QZ3}).
We construct a moduli space of framed quadratic differentials for a decorated marked surface,
that is isomorphic to the space of stability conditions on the 3-Calabi-Yau (3-CY)
category associated to the surface.
We introduce the cluster exchange groupoid as the main tool to show that
(the principal component of) this space is simply connected.
\subsection{Cluster exchange graphs and groupoids}
The cluster exchange graph $\uCEG(Q)$ for a quiver $Q$, without loops or 2-cycles, is initially defined with vertices given by the clusters in the cluster algebra defined from $Q$ and edges corresponding to mutation of clusters (or seeds) \cite{FZ}.
It is an unoriented $n$-regular graph, where $n$ is the number of vertices of $Q$.
Each cluster also carries with it a quiver, obtained from the initial one by quiver mutation, and any of these can be taken as the starting point for the construction.
Thus the graph is really associated to a mutation equivalence class of quivers.

For example, when $Q$ is a Dynkin quiver of type $A_3$ the graph is (the 1-skeleton of) the famous associahedron, whose faces are six pentagons and four squares.
Indeed, the appearance of squares or pentagons is a universal phenomenon, arising whenever a pair of vertices in $Q$ are joined by no arrow or one arrow, respectively.

Since the categorification of cluster combinatorics, the cluster exchange graph may also be realised with vertices corresponding to cluster tilting objects in the cluster category
$\C(\Gamma) = \per\Gamma /  D_{fd}(\Gamma)$ where $\Gamma$ is the Ginzburg dg algebra associated
to a quiver with potential $(Q,W)$ (see \S\ref{sec:C} for definitions).
The edges still correspond to mutation, now of cluster tilting objects, and the potential is required to be non-degenerate, in the sense that $(Q,W)$ repeatedly mutates (as in \cite{DWZ}) in the same way as $Q$, without the introduction of non-cancelling 2-cycles.

In this context, there is a second closely related exchange graph $\EGp(\Gamma)$
of (finite, reachable) hearts in $D_{fd}(\Gamma)$ or, equivalently, silting objects in $\per\Gamma$,
which naturally covers the cluster exchange graph of $\C(\Gamma)$, since the image, in the quotient $\C(\Gamma)$,
of a silting object is a cluster tilting object.

The mutation, or tilting, of hearts is a directed operation and so $\EGp(\Gamma)$ is an oriented $n,n$-regular graph and it is therefore more natural to consider an oriented version $\CEG(Q)$ of the cluster exchange graph in which each edge is replaced by a 2-cycle.
The natural map $\EGp(\Gamma)\to \CEG(Q)$ is then a covering of oriented $n,n$-regular graphs in the sense that it gives a bijection between the $2n$ edges emanating from any vertex and those from its image (cf. \cite{KQ}).

The first goal of this paper is to turn both these exchange graphs into groupoids by the addition of suitable relations,
with the expectation that the map of groupoids $\egp(\Gamma)\to \ceg(Q)$ actually becomes the universal cover.
In particular, $\egp(\Gamma)$  should be simply-connected.
The main surprise here is that, as well as adding expected relations corresponding to the squares and pentagons mentioned above, we have to add additional hexagonal relations.
One nice consequence of these is that the generators of the fundamental group of $\ceg(Q)$ become `local twists'
(see \S\ref{sec:ceg} and \S\ref{sec:CT} for details).

We will prove this expectation (see Theorem~\ref{cor:QQ} and Theorem~\ref{thm:45}) in the case that
arises from (unpunctured) marked surfaces, as studied by Fomin-Shapiro-Thurston \cite{FST}, when there is a topological model for cluster exchange graph with vertices given by triangulations and the edges by `flips', which replace one diagonal of a quadrilateral by the other.
\AKedit{More precisely, in this case,  $\Gamma=\Gamma_\RT$ is associated, in the manner of \cite{LF}, to a triangulation $\RT$ of a surface $\surf$ with marked points on its boundary (see \S\ref{sec:mark-surf} for a precise definition of ``triangulation'').}
We achieve the proof by showing that there is also a topological model $\EGT(\surfo)$ of $\EGp(\Gamma_\RT)$,
given by triangulations of the decorated marked surface $\surfo$ obtained by equipping $\surf$ with an extra set of interior points, or decorations, of which there must be one in each triangle.
\AKedit{ 
Note that the $\T$ in $\EGT(\surfo)$ is a decorated version of the $\RT$ in $\EGp(\Gamma_\RT)$
and picks out a connected component of a larger graph $\EG(\surfo)$ (see \S\ref{sec:DMS} for more details).}

A key step in the proof uses the presentation of the braid twist group of $\surfo$ obtained in the prequel \cite{QZ3}.
This approach of decorating marked surfaces is inspired by Krammer's analysis of
the classical braid group \cite{Kr}, which is the case when the marked surface is a polygon,
that is, a disc with a number of marked points on the boundary.

\subsection{Quadratic differentials and stability conditions}

\AKedit{One reason for being interested in the graph $\EGp(\Gamma)$ is that it is
a \danger{skeleton} of the principal component $\Stap(\Gamma)$ of the space of
Brideland stability conditions \cite{B1} on the category $D_{fd}(\Gamma)$.
In other words, there is an embedding $\skel_\Gamma\colon \EGp(\Gamma)\to \Stap(\Gamma)$
realising $\EGp(\Gamma)$ as a dual graph to the natural cell and wall structure on $\Stap(\Gamma)$.}
Thus there is also 
an induced map on fundamental groups
\[
 \skel_*\colon \pi_1 \EGp(\Gamma) \to \pi_1 \Stap(\Gamma)
\]
This provides a strategy for proving that $\Stap(\Gamma)$ is simply connected:
namely, prove (1) that $\skel_*$ is surjective, (2) that $\skel_*$ factors through the quotient
$\pi_1 \EGp(\Gamma)\to \pi_1 \egp(\Gamma)$ and (3) that $\egp(\Gamma)$ is simply connected.
We are able to prove (2) for general $\Gamma$ (Proposition~\ref{pp:hex}),
while (3) is proved in the first part of the paper
in just the marked surface case $\Gamma=\Gamma_\RT$, as already described.

To prove (1) in that case, we adapt the work of Bridgeland-Smith \cite{BS},
who showed that a geometric model for $\Stap(\Gamma_\RT)/\Aut$ is
given by the moduli space $\Quad(\surf)$ of quadratic differentials.
More precisely, we show that a connected component $\FQuad{\T}{\surfo}$ of a suitably defined space of
$\surfo$-framed quadratic differentials (Definition~\ref{def:SoFQuad}) provides a geometric model of
$\Stap(\Gamma_\RT)$ itself.
Note that, under the framing, the decorations are identified with the zeros of the quadratic differential.

Now, the association of a WKB triangulation to a saddle free quadratic differential gives rise to
a cell and wall structure on $\FQuad{\T}{\surfo}$ that is naturally dual to the graph $\EGT(\surfo)$.
Thus we actually have an embedding $\skel_S\colon \EGT(\surfo)\to \FQuad{\T}{\surfo}$,
which models $\skel_\Gamma$ and, crucially, for which one can show
(by an argument from \cite{BS}) that the induced map on $\pi_1$ is surjective.
This enables us to complete the above strategy and show that $\Stap(\Gamma_\RT)$ is simply-connected
(Theorem~\ref{thm:s.c.}).

\note{A consequence of this is that (each component of) the moduli space $\FQuad{}{\surfo}$ of $\surfo$-framed quadratic differentials is simply connected.
Furthermore, all components are isomorphic and the set of components can be parametrised explicitly
(Corollary~\ref{cor:s.c.}).
This result is similar to results in \cite{KZ, Bo}, where topological properties
of moduli spaces of abelian/quadratic differentials with given singularities are described.
However, the methods used here are quite different.}

Note also that the methods in this paper can't prove that the space of stability conditions is contractible,
in comparison with the results of \cite{Q2,QW,PSZ,AW} which do, but these depend crucially on the cluster type being finite,
which it rarely is in the surface case.

\subsection*{Acknowledgements}
QY would like to thank Aslak Buan and Yu Zhou for
collaborating on the prequels to this paper.
He is also grateful to Tom Bridgeland and Ivan Smith for inspiring discussions.
This work is supported by Hong Kong RGC 14300817 (from Chinese University of Hong Kong)
and Beijing Natural Science Foundation (Z180003).

\section{Cluster exchange groupoid}\label{sec:CBr}
\subsection{Quivers with potential and their associated categories}\label{sec:C}
(See \cite[\S7]{K10} for more details.)
A quiver $Q$ is a directed graph and a potential $W$ on $Q$
is a linear combination of cycles in $Q$.
For each vertex $i$ of a quiver with potential $(Q,W)$,
there is an involutory operation $\mu_i$, known as \emph{mutation} \cite[\S5]{DWZ}, that produces
a new quiver with potential $(Q',W')$.
We will always assume that a quiver with potential $(Q,W)$ is \emph{non-degenerate} \cite[\S7]{DWZ},
i.e. that it has no loops or 2-cycles and iterated mutation preserves this condition,
so that the underlying quivers undergo normal quiver mutation.
Denote by $\Gamma=\Gamma(Q,W)$ the \emph{Ginzburg dg algebra (of degree 3)} associated to
a quiver with potential $(Q,W)$,
which is constructed as follows (over any infinite field $\k$).
\begin{itemize}
\item \Note{Let $\Qgrad$ be the 3-Calabi-Yau double of $Q$,}
that is, the graded quiver whose vertex set is $Q_0$
and whose arrows are: the arrows in $Q$ in degree $0$;
an arrow $a^*\colon j\to i$ in degree $-1$ for each arrow $a\colon i\to j$ in $Q$;
a loop $e^*\colon i\to i$ in degree $-2$ for each vertex $e$ in $Q$.
\item The underlying graded algebra of $\Gamma$ is the completion of
the graded path algebra $\k\Qgrad$
with respect to the ideal generated by the arrows of $\Qgrad$.
\item  The differential of $\Gamma$ is the unique continuous linear endomorphism,
homogeneous of degree $1$, which satisfies the Leibniz rule and takes the following values
\[
  \diff \sum_{e\in Q_0} e^*  =  \sum_{a\in Q_1} \, [a,a^*],
  \qquad
  \diff \sum_{a\in Q_1} a^* =  \partial W,
\]
where $\partial W$ is the full cyclic derivative of $W$ with respect to the arrows.
\end{itemize}

A triangulated category $\D$ is called \emph{$N$-Calabi-Yau} ($N$-CY)
if, for any objects $L,M$ in $\D$, we have a natural isomorphism
\begin{gather}\label{eq:serre}
    \mathfrak{S}:\Hom_{\D}^{\bullet}(L,M)
        \xrightarrow{\sim}\Hom_{\D}^{\bullet}(M,L)^\vee[N],
\end{gather}
where the graded dual of a graded vector space $V=\bigoplus_{i\in\ZZ} V_i[i]$ is
$V^\vee=\bigoplus_{i\in\ZZ} V_i^*[-i]$.
Further, an object $S$ is \emph{$N$-spherical} if $\Hom^{\bullet}(S, S)=\k \oplus \k[-N]$.

For simplicity, we will assume that $\Gamma$ is Jacobi finite, that is, $H^0(\Gamma)$ is finite dimensional.
This holds in the case of most interest for this paper, by \cite[Thm~36]{LF}.
However, the Jacobi infinite case may also be handled using results in \cite{Pla}.
There are three triangulated categories associated to a Ginzburg dg algebra $\Gamma$, namely,
\begin{itemize}
\item
the perfect derived category $\per\Gamma$,
\item
the finite-dimensional derived category $\D_{fd}(\Gamma)$,
which is a full subcategory of $\per\Gamma$
and which is 3-Calabi-Yau,
\item
the cluster category $\C(\Gamma)$, which is defined by the following
short exact sequence of triangulated categories (due to Amiot \cite{A})
\begin{gather}\label{eq:ses}
    0\to D_{fd}(\Gamma)\to\per\Gamma\to \C(\Gamma)\to0
\end{gather}
and which is 2-Calabi-Yau.
\end{itemize}

If $\Gamma'=\Gamma(Q',W')$ is obtained by mutation,
then $D_{fd}(\Gamma')\isom D_{fd}(\Gamma)$ and $\per\Gamma'\isom\per\Gamma$,
so $\C(\Gamma')\isom\C(\Gamma)$ (see \cite[\S3-4]{KY}).

\subsection{Cluster tilting and cluster exchange graphs}
(See e.g. \cite{BIRS} for more detail.)
A \emph{cluster tilting set} (or \emph{cluster}, for short) $\bfc$
in a cluster category $\C(\Gamma)$ is
a maximal collection of indecomposable objects satisfying an $\Ext^1$ vanishing condition.
There is an involutory operation, known as \emph{mutation} on any cluster tilting set $\bfc$,
with respect to any one of its objects, cf \cite{IY}.
We denote by $Q_\bfc$ the Gabriel quiver of $\bfc$,
that is, of the algebra $\End(\bigoplus_{M\in\bfc} M)$.

The cluster category $\C(\Gamma)$ admits a canonical cluster $\bfc_\Gamma$,
which is the image of the set of indecomposable summands of $\Gamma$ in $\per\Gamma$
under the quotient map in \eqref{eq:ses}.
If $\Gamma=\Gamma(Q,W)$, then $Q_{\bfc_\Gamma}=Q$, by \cite[Thm~2.1]{A}.
If $\Gamma'$ is obtained from $\Gamma$ by mutation,
then the isomorphism $\C(\Gamma')\isom\C(\Gamma)$ identifies $\bfc_{\Gamma'}$ with a mutation of $\bfc_\Gamma$.
Consequently, for a non-degenerate $(Q,W)$, the quiver $Q_\bfc$ will have no loops or 2-cycles,
for every cluster $\bfc$ in $\C(\Gamma)$ obtained by iterated mutation from $\bfc_\Gamma$.
Indeed, when such clusters are related by cluster mutation, their Gabriel quivers are related by normal quiver mutation.

\begin{definition}\label{def:CEG}
The \emph{unoriented cluster exchange graph} $\uCEG(\Gamma)$
of the cluster category $\C(\Gamma)$ is
the connected graph whose vertices are the clusters in $\C(\Gamma)$ that are reachable from $\bfc_\Gamma$ by iterated mutation
and whose edges correspond to mutations.

The \emph{oriented cluster exchange graph} $\CEG(\Gamma)$ is
the graph obtained from $\uCEG(\Gamma)$ by replacing each unoriented edge
with a 2-cycle. We call the oriented edges \emph{forward mutations}.
See Figure~\ref{fig:s and p} for two elementary examples.

Note that, by the preceding discussion, these graphs only depend on the mutation class $\class$ of the (non-degenerate) quiver with potential $(Q,W)$ that determines $\Gamma$.
Hence we also write $\uCEG(\class)$ and $\CEG(\class)$ for these graphs.
\end{definition}

\begin{figure}[ht]\centering
\begin{tikzpicture}[scale=.7, rotate=45,
  arrow/.style={->,>=stealth}]
\foreach \j in {1,2,3,4}
   {\draw (90*\j:2cm) node (t\j) {$\bullet$};
   \draw (90*\j+45:2cm) node {$x$};\draw (90*\j+45:.8cm) node {$y$};}
\foreach \a/\b in {1/2,2/3,3/4,4/1}{
  \draw (t\a) edge[arrow,bend left=15]  (t\b);
  \draw (t\b) edge[arrow,bend left=15]  (t\a);}
\draw (t1) node[white] {$\bullet$} node{$\bfc$};
\end{tikzpicture}
\qquad\quad
\begin{tikzpicture}[scale=.7, rotate=90,
  arrow/.style={->,>=stealth}]
\foreach \j in {1,2,3,4,5}
   {\draw (72*\j:2cm) node (t\j) {$\bullet$};
   \draw (72*\j+36:2.1cm) node {$x$};\draw (72*\j+36:1cm) node {$y$};}
\foreach \a/\b in {1/2,2/3,3/4,4/5,5/1}{
  \draw (t\a) edge[arrow,bend left=15]  (t\b);
  \draw (t\b) edge[arrow,bend left=15]  (t\a);}
\draw (t1) node[white] {$\bullet$} node{$\bfc$};
\end{tikzpicture}
\caption{$\CEG(Q)$ for a quiver $Q$ of type $A_1\times A_1$ and $A_2$}
\label{fig:s and p}
\end{figure}
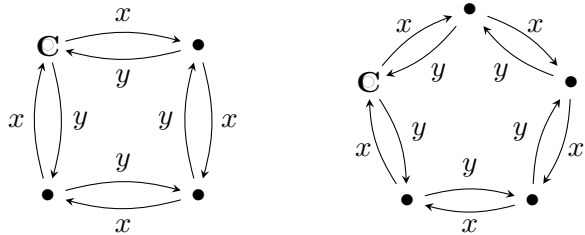

\subsection{Cluster exchange groupoids}\label{sec:ceg}
We now introduce the cluster exchange groupoid, which is an enhancement of the cluster exchange graph.
Note that the \emph{path groupoid} of an oriented graph $\mathrm{E}$ is the category
whose objects are the vertices of $\mathrm{E}$ and whose generating morphisms are
the edges of $\mathrm{E}$ and their formal inverses.

\Note{
\begin{definition}\label{def:local twist}
For each cluster $\bfc$ in $\CEG(\class)$ and each vertex $i$ of the associated quiver $Q_\bfc$,
the \emph{local twist} $t_i$ is the length two loop in the path groupoid of $\CEG(\class)$
obtained by composing the two forward mutations at $i$
corresponding to the edge $\mu_i$ in $\uCEG(\class)$.
\end{definition}
}

Let $i$ and $j$ be two vertices of $Q_\bfc$.
Then $\CEG(\class)$ contains the following full subgraph
\begin{gather}\label{eq:6}
\begin{tikzpicture}[scale=1, arrow/.style={->,>=stealth}, baseline=(bb.base)]
\path (5,0) node (bb) {}; 
\foreach \j in {1,2,3,4}
   {\draw (\j*2,0) node (t\j) {$\bullet$};}
\foreach \a/\b in {1/2,2/3,3/4}{
  \draw (t\a) edge[arrow,bend left=20]node[above]{$x$}  (t\b);
  \draw (t\b) edge[arrow,bend left=20]node[below]{$y$}  (t\a);}
\draw(4,0)node[below]{$\bfc$}(6,0)node[below]{$\bfc'$}
(3,0)node{\tiny $i$} (5,0)node{\tiny $j$} (7,0)node{\tiny $i$};
\end{tikzpicture}
\end{gather}
where the labels $i,j$ indicate the vertices being mutated;
for example, $\bfc'$ is the mutation of $\bfc$ at $j$.
The local twists at $\bfc$ are $t_i=yx$ and $t_j=xy$,
while the local twists at $\bfc'$ are $t_i'=xy$ and $t_j'=yx$,
where we compose arrows left to right.

As there are no 2-cycles in the quiver $Q_\bfc$ or any quiver obtained from it by mutation,
there can only be arrows between $i$ and $j$ in (at most) one direction.
If there are arrows in one direction, then $x$ and $y$ can be distinguished
as being forward mutation at either the tail or head of arrows.
For example, if there are arrows $j\to i$ in $Q_{\bfc}$, then there are arrows $i\to j$ in $Q_{\bfc'}$
and $x$ is mutation at the tail (or source) in each case.

\begin{definition}\label{def:ceg}
The \emph{cluster exchange groupoid} $\ceg(\class)$ is the quotient of
the path groupoid of $\CEG(\class)$ by the following three kinds of relations,
starting at each cluster $\bfc$ and for each $i,j$ in $Q_{\bfc}$
with no arrow from $i$ to $j$.
We use the notation in \eqref{eq:6}.
\begin{enumerate}
\item
The \emph{hexagonal dumbbell} relation $x^2y=yx^2$.
This can also written $xt'_i=t_ix$ (hence ``dumbbell''),
where $t_i$ and $t'_i$ are local twists at $\bfc$ and $\bfc'$, respectively, as above.
It can be drawn as a (flattened) hexagon:
\begin{gather}
\label{eq:hex}
\begin{tikzpicture}[arrow/.style={->,>=stealth}, baseline=(bb.base)]
\path (0,0) node (bb) {}; 
\draw (2,0.5) node (t1) {$\bfc$} (2,-0.5) node (t5) {$\bfc$}
   (4,0.5) node (t2) {$\bfc'$} (4,-0.5) node (t4) {$\bfc'$}
   (6,0) node (t3) {$\bullet$}  (0,0) node (t6) {$\bullet$};
\draw (2,0)node[rotate=90]{$=$} (4,0)node[rotate=90]{$=$};
\draw(t1)edge[arrow]node[above]{$x$} (t2);
\draw(t2)edge[arrow]node[above]{$x$} (t3);
\draw(t5)edge[arrow]node[below]{$x$} (t4);
\draw(t6)edge[arrow]node[below]{$x$} (t5);
\draw(t1)edge[arrow]node[above]{$y$} (t6);
\draw(t3)edge[arrow]node[below]{$y$} (t4);
\end{tikzpicture}\end{gather}
\item
The \emph{square relation} $x^2=y^2$,
when there is also no arrow from $j$ to $i$,
in which case \eqref{eq:6} extends to the subgraph of $\CEG(\class)$
on the left of Figure~\ref{fig:s and p}.
\item
The \emph{pentagon relation} $x^2=y^3$,
when there is exactly one arrow from $j$ to $i$,
in which case \eqref{eq:6} extends to the subgraph of $\CEG(\class)$
on the right of Figure~\ref{fig:s and p}.
\end{enumerate}
\end{definition}

\begin{remark}\label{rem:s and p}
When there is at most one arrow between two vertices $i$ and $j$ at $\bfc$,
the hexagonal dumbbell relation at $\bfc$ for $(i,j)$
is already implied by square or pentagon relations in $\ceg(\class)$.
More precisely, two square relations of the form $x^2=y^2$ imply $x^2y=y^3=yx^2$.
Note that, in this case, they also imply $y^2x=x^3=xy^2$,
which is another hexagonal dumbbell relation obtained by interchanging the roles of $i$ and $j$, and thus $x$ and $y$.

On the other hand, two pentagon relations of the form $x^2=y^3$ imply $x^2y=y^4=yx^2$.
\Note{
These calculations can be visualized as decomposing the hexagon into
two squares or two pentagons as shown in \eqref{eq:4 and 5}.
\begin{gather}
\label{eq:4 and 5}
\begin{tikzpicture}[arrow/.style={->,>=stealth}, baseline=(bb.base),yscale=.4, xscale=.2]
\path (0,0) node (bb) {}; 
\draw
   (0,3) node (t1) {$\bfc$} (0,-3) node (t5) {$\bfc$}
   (10,3) node (t2) {$\bfc'$} (10,-3) node (t4) {$\bfc'$}
   (10,0) node (t3) {$\bullet$}  (0,0) node (t6) {$\bullet$};
\draw(t1)edge[arrow]node[above]{$x$} (t2);
\draw(t2)edge[arrow]node[right]{$x$} (t3);
\draw(t5)edge[arrow]node[below]{$x$} (t4);
\draw(t6)edge[arrow]node[left]{$x$} (t5);
\draw(t1)edge[arrow]node[left]{$y$} (t6);
\draw(t3)edge[arrow]node[right]{$y$} (t4);
\draw(t6)edge[arrow]node[above]{$y$} (t3);
\end{tikzpicture}
\qquad
\begin{tikzpicture}[arrow/.style={->,>=stealth}, baseline=(bb.base),yscale=.4, xscale=.2]
\path (0,0) node (bb) {}; 
\draw
   (0,3) node (t1) {$\bfc$} (0,-3) node (t5) {$\bfc$}
   (10,3) node (t2) {$\bfc'$} (10,-3) node (t4) {$\bfc'$}
   (10,0) node (t3) {$\bullet$}  (0,0) node (t6) {$\bullet$};
\draw(t1)edge[arrow]node[above]{$x$} (t2);
\draw(t2)edge[arrow]node[right]{$x$} (t3);
\draw(t5)edge[arrow]node[below]{$x$} (t4);
\draw(t6)edge[arrow]node[left]{$x$} (t5);
\draw(t1)edge[arrow]node[left]{$y$} (t6);
\draw(t3)edge[arrow]node[right]{$y$} (t4);
\draw (5,0)node (t) {$\bullet$};
\draw(t6)edge[arrow]node[above]{$y$} (t);
\draw(t)edge[arrow]node[above]{$y$} (t3);
\end{tikzpicture}\end{gather}
When there is more than one arrow between two vertices,
the hexagon cannot be decomposed.
There are two (unconnected) infinite mutation sequences from/to the two middle clusters
as shown in \eqref{eq:hex}.
\begin{gather}
\label{eq:hex}
\begin{tikzpicture}[arrow/.style={->,>=stealth}, baseline=(bb.base),yscale=.4, xscale=.8]
\path (0,0) node (bb) {}; 
\foreach \j in {0,1,2,7-2,8-2,9-2} {\draw(\j+1,0) node[white](xx) {$\bullet$}; \draw(\j,0) node{$\bullet$} edge[arrow] (xx) ;}
\draw(3.1,0)node{$\cdots$}(7-2,0)node[white]{$\bullet$}node{$\cdots$};
\draw
   (0,3) node (t1) {$\bfc$} (0,-3) node (t5) {$\bfc$}
   (10-2,3) node (t2) {$\bfc'$} (10-2,-3) node (t4) {$\bfc'$}
   (10-2,0) node (t3) {$\bullet$}  (0,0) node (t6) {$\bullet$};
\draw(t1)edge[arrow]node[above]{$x$} (t2);
\draw(t2)edge[arrow]node[right]{$x$} (t3);
\draw(t5)edge[arrow]node[below]{$x$} (t4);
\draw(t6)edge[arrow]node[left]{$x$} (t5);
\draw(t1)edge[arrow]node[left]{$y$} (t6);
\draw(t3)edge[arrow]node[right]{$y$} (t4);
\end{tikzpicture}\end{gather}
}
\end{remark}

Consider a cluster $\bfc$ in $\uCEG(\class)$
with vertices $i$ and $j$ in its associated quiver $Q_\bfc$
such that there is at most one arrow between them.
The sub-graph of $\uCEG(\class)$ obtained by freezing all vertices of the quiver
$Q_\bfc$ except $i$ and $j$, is an ordinary square or pentagon,
i.e. the unoriented graphs that give the oriented graphs in Figure~\ref{fig:s and p} on replacing edges by 2-cycles.

We may also define the \emph{unoriented cluster exchange groupoid} $\uceg(\class)$
to be the quotient of the path groupoid of $\uCEG(\class)$
by appropriate square and pentagon relations.
More precisely, this path groupoid has a pair of mutually inverse generators for each unoriented edge in
$\uCEG(\class)$ and the cycle around any square or pentagon is made trivial.
Note that $\uceg(\class)$ can be obtained from $\ceg(\class)$ by setting
all local twists $t_i$ equal to 1 in Definition~\ref{def:local twist}.

The following condition is known to hold in many cases,
for example, when $\class$ is of Dynkin type (\cite[Prop~4.5]{Q2})
or when $\class$ is from a marked surface (\cite[Thm~3.10]{FST}).
Note, however, that it does not always hold (\cite[Rem~9.19]{FST}).

\begin{condition}\label{con:pi1triv}
The fundamental group of $\uceg(\class)$ is trivial.
In other words, any loop in the cluster exchange graph $\uCEG(\class)$
decomposes into ordinary (unoriented) squares and pentagons.
\end{condition}

\subsection{The cluster braid groups}\label{sec:CT}
Recall, from Definition~\ref{def:local twist}, that at each cluster $\bfc$ in $\CEG(\class)$,
there are local twists $t_i$ at $\bfc$, for each vertex $i$ of
the associated quiver $Q_\bfc$.

\begin{definition}
The \emph{cluster braid group} $\CBr(\bfc)$ of $\bfc$ is defined to be the subgroup of $\pi_1(\ceg(\class),\bfc)$ generated by $\{t_i\mid i\in Q_\bfc\}$.
\end{definition}

\begin{lemma}\label{lem:12}
If there is no arrow between $i$ and $j$, then the local twists $t_i$ and $t_j$ at $\bfc$ satisfy
$t_it_j=t_jt_i$. 
If there is exactly one arrow between $i$ and $j$, then the local twists $t_i$ and $t_j$ at $\bfc$ satisfy
$t_i t_j t_i=t_jt_i t_j$.
\end{lemma}

\begin{proof}
In the first case, we have the local full sub-graph shown on the left of
Figure~\ref{fig:s and p}. Then $x^2=y^2$ implies
\[
    t_i t_j=(yx)(xy)=y^4=x^4=(xy)(yx)=t_j t_i.
\]
In the second case, without loss of generality, suppose the arrow is from $j$ to $i$
so that we have the local full sub-graph shown on the right of
Figure~\ref{fig:s and p}. Then $x^2=y^3$ implies
\[\begin{array}{rll}
    t_i t_j t_i&=&(yx)(xy)(yx)=yx^2y^2x=y^6x=x^5\\
    &=&xy^6=xy^2x^2y=(xy)(yx)(xy)\\
    &=&t_j t_i t_j.
\end{array}\]

\Note{
These calculation also be also read from the pictures in Figure~\ref{fig:CoBr}, where
the bullet points denote the initial cluster $\bfc$, the green arrows are $x$ and the orange ones are $y$.}
\begin{figure}\centering
\begin{tikzpicture}[arrow/.style={->,>=stealth},scale=.5]
\matrix[blue!50] (m) [matrix of math nodes, row sep=1.5cm, column sep=1.5cm] at (0,0)
    {\bullet&\circ&\bullet\\
        \circ&\circ&\circ\\
        \bullet&\circ&\bullet\\};
\foreach \i/\j/\k/\l in {1/1/1/2,1/2/2/2,2/2/2/3,2/3/3/3,2/1/3/1,3/1/3/2}
{\path (m-\i-\j) edge[thick,->,Emerald,>=stealth] (m-\k-\l);
    \path (m-\j-\i) edge[thick,->,Orange,>=stealth] (m-\l-\k);}
\end{tikzpicture}
\qquad
\begin{tikzpicture}[arrow/.style={->,>=stealth},scale=.5]
\matrix[blue!50] (m) [matrix of math nodes, row sep=1.5cm, column sep=.7cm] at (0,0)
    {\bullet&&\circ&\bullet&\circ&&\bullet\\
    \circ&\circ&\circ&&\circ&\circ&\circ\\
    \bullet&&\circ&\bullet&\circ&&\bullet\\};
\foreach \i/\j/\k/\l in {1/1/1/3,1/3/2/3,2/3/2/5,2/5/3/5,3/5/3/7,2/1/3/1,3/1/3/3,1/5/1/7,1/7/2/7}
{\path (m-\i-\j) edge[thick,->,Emerald,>=stealth] (m-\k-\l);}
\foreach \i/\j/\k/\l in {1/1/2/1,2/1/2/2,2/2/2/3,2/3/3/3,3/3/3/4,3/4/3/5,1/3/1/4,1/4/1/5,1/5/2/5,2/5/2/6,2/6/2/7,2/7/3/7}
{\path (m-\i-\j) edge[thick,->,Orange,>=stealth] (m-\k-\l);}
\end{tikzpicture}
\caption{Braid relations of local twists comes from squares and pentagons}\label{fig:CoBr}
\end{figure}
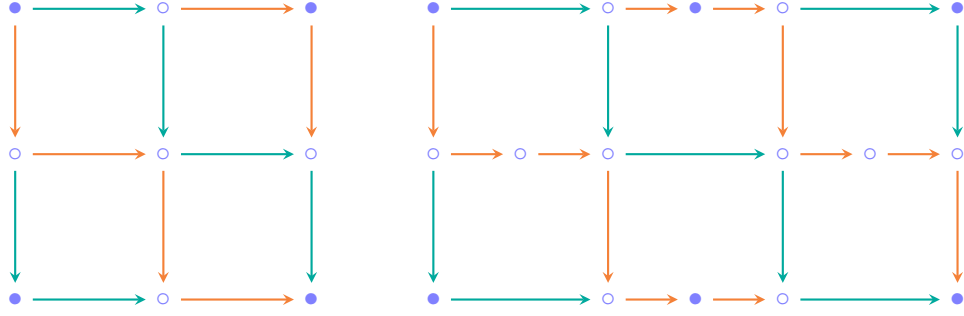
\end{proof}

\def\cyi{\bfc^\circ}
\def\cer{\bfc^\bullet}
\def\tyi{t^\circ}
\def\ter{t^\bullet}

Consider any forward mutation
$\cyi\xrightarrow{z }\cer$
in $\ceg(\class)$.
Conjugation by $z$ gives a isomorphism
\[\conj_z\colon \pi_1(\ceg(\class),\cyi)\to\pi_1(\ceg(\class),\cer)\colon
    t\mapsto z^{-1} t z.\]
Note that (locally) we can identify the vertex sets of $Q_{\cyi}$ and $Q_{\cer}$; denote this set by $Q_0$.
Denote by $\{\tyi_l \mid l\in Q_0\}$ the local twists generating $\CBr(\cyi)$
and by $\{\ter_l \mid l\in Q_0\}$ the local twists generating $\CBr(\cer)$.

\begin{proposition}[Conjugation formula]\label{pp:conj}
Suppose that $z$, as above, is forward mutation with respect to vertex $k\in Q_0$.
We have
\begin{gather*}
\begin{array}{ccl}
    \conj_z(\tyi_l)&=&\begin{cases}
    (\ter_k)^{-1} \ter_l \ter_k & \text{if there are arrows from $l$ to $k$ in $Q_0$,}\\
    \ter_l & \text{otherwise,}
    \end{cases}\end{array}
\end{gather*}
and hence $\conj_z$ induces an isomorphism
$\CBr(\cyi)\cong\CBr(\cer)$.
\end{proposition}
\begin{proof}
We only need to show the formula.
Consider the case that there are no arrows from $l$ to $k$ first.
Then we are in the situation of \eqref{eq:6}
such that $\cyi\xrightarrow{z}\cer$ is $\bfc\xrightarrow{x}\bfc'$ and $(k,l)=(j,i)$.
Hence the hexagon relation $x^2y=yx^2$ implies
\[
    \conj_x(\tyi_i)=x^{-1}(yx)x=x^{-1}(x^2y)=xy=\ter_i
\]
for $l\neq j$ and
\[
    \conj_x(\tyi_j)=x^{-1}(xy)x=yx=\ter_j
\]
as required.

Now, suppose there are arrows from $l$ to $k$.
Then we are in the situation of \eqref{eq:6}
such that $\cyi\xrightarrow{z}\cer$ is $\bfc'\xrightarrow{y}\bfc$ and $(k,l)=(i,j)$.
Hence the hexagon relations $x^2y=yx^2$ implies
\[
    \conj_y(\tyi_j)=y^{-1}(xy)y=(xy)^{-1}x^2y^2=(xy)^{-1}yx^2y
    =(xy)^{-1}(yx)(xy)=(\ter_i)^{-1} \ter_j \ter_i
\]
as required.
\end{proof}

\begin{proposition}\label{pp:ass}
Condition~\ref{con:pi1triv} is equivalent to the identity
$\CBr(\bfc)=\pi_1(\ceg(\class),\bfc)$ for any cluster $\bfc$.
\end{proposition}
\begin{proof}
First, we claim that for any $t\in \CBr(\cyi)$
and any path $p\colon \cyi\to\bfc$ in $\ceg(\class)$,
the loop $p^{-1}t p$ is in $\CBr(\bfc)$.
By induction, we only need to show the case when $p$ is of length one,
which follows from the conjugation formula in Proposition~\ref{pp:conj}.
Hence, $\{\CBr(\bfc) \mid \bfc\in\ceg(\class)\}$ is the normal subgroupoid
(cf. \cite[Chap. 12]{H})
of $\pi_1(\ceg(\class),\bfc)$, normally generated by the local twists.

Consider the canonical covering $\ceg(\class)\to\uceg(\class)$, which sets all local twists equal to 1.
We have a short exact sequence
\[
    1\to \CBr(\bfc) \to \pi_1(\ceg(\class),\bfc) \to
        \pi_1(\uceg(\class),\bfc)\to1
\]
which shows the required equivalence.
\end{proof}

\subsection{Exchange graphs of hearts and tilting}\label{sec:heart}

There is a close relationship between cluster exchange graphs and exchange graphs of hearts
in the corresponding 3-CY categories $\D_{fd}(\Gamma)$.
We will describe this relationship here and show that it extends to the same relationship between groupoids.

A \emph{t-structure} $\tstruc$
on a triangulated category $\D$ is
a full subcategory $\tstruc \subset \D$
with $\tstruc[1] \subset \tstruc$ and such that, if one defines
\[
  \tstruc^{\perp}=\{ G\in\D: \Hom_{\D}(F,G)=0,
  \forall F\in\tstruc  \},
\]
then, for every object $E\in\D$, there is
a unique triangle $F \to E \to G\to F[1]$ in $\D$
with $F\in\tstruc$ and $G\in\tstruc^{\perp}$.
It is \emph{bounded} if for any object $M$ in $\D$,
the shifts $M[k]$ are in $\tstruc$ for $k\gg0$ and in $\tstruc^{\perp}$ for $k\ll0$.
We will only consider bounded t-structures.
The \emph{heart} of a (bounded) t-structure $\tstruc$ is the full subcategory
$\h=\tstruc^\perp[1]\cap\tstruc$, which is an abelian category
and also uniquely determines $\tstruc$.

A \emph{torsion pair} in a heart $\h$ (or any abelian category) is a pair of
full subcategories $\<\torfree,\torsion\>$ of $\h$,
such that $\Hom(\torsion,\torfree)=0$ and, for every $E \in \h$, there is a short exact sequence
$0 \to E^{\torsion} \to E \to E^{\torfree} \to 0$,
for some $E^{\torsion} \in \torsion$ and $E^{\torfree} \in \torfree$.
In this situation, there is a new heart $\h^\sharp$ with torsion pair
$\<\torsion,\torfree[1]\>$, called the \emph{forward tilt} of $\h$, and
a heart $\h^\flat$ with torsion pair $\<\torsion[-1],\torfree\>$,
called the \emph{backward tilt} of $\h$.
We say that a forward tilt is \emph{simple}
if $\torfree=\<S\>$ for a rigid simple $S$
and write it as $\tilt{\h}{\sharp}{S}$.
Similarly, a backward tilt is \emph{simple}
if $\torsion=\<S\>$ for a rigid simple $S$,
and we write it as~$\tilt{\h}{\flat}{S}$.
See, for example, \cite[\S3]{KQ} for more detail.

Let $\Gamma=\Gamma(Q,W)$ be the Ginzburg dg algebra of some quiver with potential $(Q,W)$.
Then $\D_{fd}(\Gamma)$ admits a \emph{canonical heart} $\h_\Gamma$ generated
by simple $\Gamma$-modules $S_e$, for $e\in Q_0$, each of which is spherical
(see e.g. \cite[\S7.4]{K10}).
\note{
Recall that, here, an object $S$ is \emph{spherical} if
$\Hom^{\bullet}(S, S)=\k \oplus \k[-3]$
}
and that it then induces a \emph{twist functor} $\twi_S$, defined by
\begin{equation*}
    \twi_S(X)=\Cone\left(S\otimes\Hom^\bullet(S,X)\to X\right)
\end{equation*}
with inverse
\[
    \twi_S^{-1}(X)=\Cone\left(X\to S\otimes\Hom^\bullet(X,S)^\vee \right)[-1].
\]
Denote by $\ST(\Gamma)\leq \Aut\D_{fd}(\Gamma)$ the \emph{spherical twist group},
generated by $\{\twi_{S_e}\mid e\in Q_0\}$.

The \emph{(total) exchange graph} $\EG(\D)$ of a triangulated category $\D$
is the oriented graph whose vertices are all hearts in $\D$
and whose directed edges correspond to simple forward tiltings between them.
We will focus attention on the \emph{principal component} $\EGp(\Gamma)$ of the
exchange graph $\EG(\D_{fd}(\Gamma))$, consisting of those hearts that are
reachable by repeated tilting from the canonical heart $\h_\Gamma$.

Denote by $\Sph(\Gamma)$ the set of \emph{reachable} spherical objects
in $\D_{fd}(\Gamma)$, that is,
\begin{gather}\label{eq:sph=st}
    \Sph(\Gamma)=\ST(\Gamma)\cdot\Sim\h_\Gamma,
\end{gather}
where $\Sim\h$ denotes the set of simples of an abelian category $\h$.
Then $\Sph(\Gamma)$
in fact consists of all the simples of reachable hearts (see, e.g. \cite{QQ}).

We have the following result (cf. \cite[Thm~8.6]{KQ} for the acyclic case).
\begin{theorem} \label{thm:KN}
Let $\Gamma$ be the Ginzburg dg algebra of some non-degenerate quiver with potential $(Q,W)$.
There is a covering  of oriented graphs
\begin{gather}\label{eq:quo}
    \EGp(\Gamma) \to \CEG(\Gamma),
\end{gather}
with covering group \AKedit{$\STp(\Gamma) =\ST(\Gamma)/\ST_0$,}
where $\ST_0$ is the subgroup of $\ST(\Gamma)$
that acts trivially on $\EGp(\Gamma)$,
\AKedit{so that $\STp(\Gamma)$ acts faithfully.}
\end{theorem}

\begin{proof}
This is essentially a result of Keller-Nicol\'as \cite[Thm~5.6]{K6}.
The map is given by associating to a heart in $\D_{fd}(\Gamma)$ a silting set in $\per\Gamma$ and taking its image
in $\C(\Gamma)$ under the quotient map in \eqref{eq:ses}.
The map is a covering of oriented graphs because it is compatible with tilting/mutation.
The group $\ST(\Gamma)$ acts transitively on the fibres
and the stabiliser of a single heart must act trivially on adjacent hearts and thus, recursively, on the whole component $\EGp(\Gamma)$.
\end{proof}

\Note{
The compatibility of the covering \eqref{eq:quo} with mutation implies that
the Ext quiver of the simples in a heart $\h$ in $\EGp(\Gamma)$ coincides with
the 3-Calabi-Yau double (as in \S\ref{sec:C})
of the Gabriel quiver $Q_\bfc$ of the cluster $\bfc$ in $\CEG(\Gamma)$ to which it maps.
Note that the no-loop condition on $Q_\bfc$ becomes
the rigidity of simples in $\h$.
Recall that an object $S$ is \emph{rigid} if $\Ext^1(S,S)=0$.
One important feature of the covering \eqref{eq:quo} is that the squares, pentagons and hexagons
that we identified in $\CEG(\Gamma)$ in Section~\ref{sec:ceg} all lift to $\EGp(\Gamma)$.
}

\begin{remark}\label{rem:ST0triv}
\AKedit{
In fact, $\STp(\Gamma) =\ST(\Gamma)$, i.e. $\ST_0$ is trivial,
in both the Dynkin case (\cite{QW}) and the surface case (\cite[Theorem~B]{BQZ}).
In other words,  $\ST(\Gamma)$ acts faithfully on $\EGp(\Gamma)$ in these cases.
In Section~\ref{sec:KQ}, the surface case becomes the main example we study in the paper.
}
\end{remark}

\AKedit{In a similar way to $\CEG(\Gamma)$, the covering exchange graph $\EGp(\Gamma)$ contains hexagons, squares and pentagons and thus there is an exchange groupoid analogous to Definition~\ref{def:ceg}.
More precisely we have the following.
}

\begin{lemma}\label{lem:456}
Let $\h$ be a heart in $\EGp(\Gamma)$ with rigid simples $S_i,S_j$ satisfying $\Ext^1(S_i,S_j)=0$.
Then $S_i$ is a simple in $\h_j=\tilt{\h}{\sharp}{S_j}$
and there is a hexagon in $\EGp(\Gamma)$, as on the left of Figure~\ref{fig:hex},
where $\h_i=\tilt{\h}{\sharp}{S_i}$, $\h_{ji}=\tilt{ (\h_j) }{\sharp}{S_i}$ and $T_j=\twi_{S_i}^{-1}(S_j)$.
\end{lemma}

\begin{proof}
The assumption $\Ext^1(S_i,S_j)=0$ implies that $S_i$ is still a simple in $\h_j$ by \cite[Prop.~5.4]{KQ}.
By \cite[(8.3)]{KQ},
the spherical twist $\twi_{S_i}^{-1}$ of $\h$ can be obtained by two tilts,
with respect to $S_i$ and then $S_i[1]$:
\[
  \twi_{S_i}^{-1}(\h)=\tilt{\left( \tilt{\h}{\sharp}{S_i} \right)}{\sharp}{S_i[1]}.
\]
\note{
Similarly, we have
\[
    \twi_{S_i}^{-1}(\h_j)=\tilt{\left( \tilt{(\h_j)}{\sharp}{S_i} \right)}{\sharp}{S_i[1]}.
\]
Also note that, as $\twi_{S_i}^{-1}$ is an auto-equivalence,
the tilting $\h\xrightarrow{S_j}\h_j$ becomes $$\twi_{S_i}^{-1}(\h)\xrightarrow{T_j}\twi_{S_i}^{-1}(\h_j).$$
}
Thus we obtain the claimed hexagon.
\end{proof}

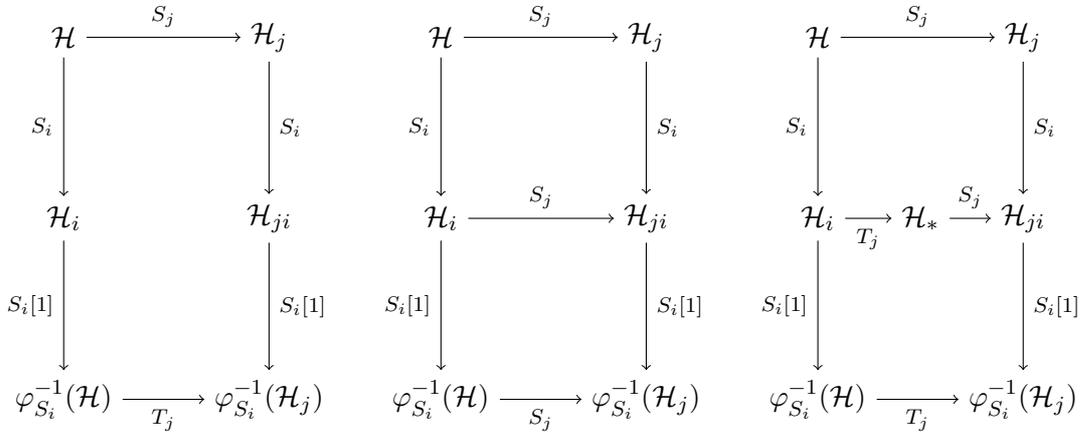
\begin{figure}[ht]
\begin{tikzpicture}[arrow/.style={->,>=stealth},rotate=-90,xscale=1.2,yscale=1.35]
\draw (0,0) node (t1) {$\h$}
    (0,2) node (t2) {$\h_j$}
    (2,2) node (t3) {$\h_{ji}$}
    (4,2) node (t4) {$\twi_{S_i}^{-1}(\h_j)$}
    (4,0) node (t5) {$\twi_{S_i}^{-1}(\h)$}
    (2,0) node (t6) {$\h_i$};
\draw [->, font=\scriptsize]
    (t1)edge node[above]{$S_j$} (t2)
    (t2)edge node[right]{$S_i$} (t3)
    (t5)edge node[below]{$T_j$} (t4)
    (t6)edge node[left]{$S_i[1]$} (t5)
    (t1)edge node[left]{$S_i$} (t6)
    (t3)edge node[right]{$S_i[1]$} (t4);
\end{tikzpicture}
\quad
\begin{tikzpicture}[arrow/.style={->,>=stealth},rotate=-90,xscale=1.2,yscale=1.35]
\draw (0,0) node (t1) {$\h$}
    (0,2) node (t2) {$\h_j$}
    (2,2) node (t3) {$\h_{ji}$}
    (4,2) node (t4) {$\twi_{S_i}^{-1}(\h_j)$}
    (4,0) node (t5) {$\twi_{S_i}^{-1}(\h)$}
    (2,0) node (t6) {$\h_i$};
\draw [->, font=\scriptsize]
    (t1)edge node[above]{$S_j$} (t2)
    (t2)edge node[right]{$S_i$} (t3)
    (t5)edge node[below]{$S_j$} (t4)
    (t6)edge node[left]{$S_i[1]$} (t5)
    (t1)edge node[left]{$S_i$} (t6)
    (t3)edge node[right]{$S_i[1]$} (t4)
    (t6)edge node[above]{$S_j$} (t3);
\end{tikzpicture}
\quad
\begin{tikzpicture}[arrow/.style={->,>=stealth},rotate=-90,xscale=1.2,yscale=1.35]
\draw (0,0) node (t1) {$\h$}
    (0,2) node (t2) {$\h_j$}
    (2,2) node (t3) {$\h_{ji}$}
    (4,2) node (t4) {$\twi_{S_i}^{-1}(\h_j)$}
    (4,0) node (t5) {$\twi_{S_i}^{-1}(\h)$}
    (2,0) node (t6) {$\h_i$}
    (2,1) node (t) {$\h_*$};
\draw [->, font=\scriptsize]
    (t1)edge node[above]{$S_j$} (t2)
    (t2)edge node[right]{$S_i$} (t3)
    (t5)edge node[below]{$T_j$} (t4)
    (t6)edge node[left]{$S_i[1]$} (t5)
    (t1)edge node[left]{$S_i$} (t6)
    (t3)edge node[right]{$S_i[1]$} (t4)
    (t6)edge node[below]{$T_j$}(t)
    (t)edge node[above]{$S_j$}(t3);
\end{tikzpicture}
\caption{A hexagon and its decompositions into squares/pentagons}
\label{fig:hex}
\end{figure}

\begin{remark}\label{rem:456}
In Lemma~\ref{lem:456}, we may also describe $\h_{ji}$ as the forward tilt of $\h$
with respect to the torsion pair with $\torfree=\<S_i,S_j\>$.
Furthermore, by \cite[Prop.~5.4]{KQ}, $T_j$ is the simple in $\h_i$ that replaces $S_j$ after the tilt from $\h$.

If, in addition, $\Ext^1(S_j, S_i)=0$, then $T_j=S_j$ and $\h_{ji}$ is also
$\tilt{ (\h_i) }{\sharp}{S_j}$.
Thus there is a square in $\EGp(\Gamma)$,
as in the upper part of the middle diagram in Figure~\ref{fig:hex},
and the hexagon in Lemma~\ref{lem:456} decomposes into two similar squares.

If $\Ext^1(S_j, S_i)=\k$, then $\h_{ji}=\tilt{  (\h_*)  }{\sharp}{S_j}$,
where $\h_*=\tilt{(\h_i)}{\sharp}{T_j}$.
This follows because $\torfree=\<S_i,S_j\>$ has just three indecomposables $S_i, T_j, S_j$
(cf. proof of Lemma~6.1 in \cite{QW}).
Thus there is a pentagon in $\EGp(\Gamma)$,
as in the upper part of the right diagram in Figure~\ref{fig:hex},
and the hexagon in Lemma~\ref{lem:456} decomposes into two similar pentagons.
\end{remark}

\begin{definition}\label{def:gpd.h}
The \emph{exchange groupoid} $\egp(\Gamma)$
is the quotient of the path groupoid of $\EGp(\Gamma)$ by
the commutation relations, starting at each heart $\h$,
corresponding to the squares, pentagons and hexagons
in Figure~\ref{fig:hex}.
\end{definition}

With this definition, Theorem~\ref{thm:KN} can be upgraded as follows,
noting that the squares/pentagons/hexagons in Figure~\ref{fig:hex} cover
the ones in Definition~\ref{def:ceg}.

\begin{proposition}\label{pp:cover ST}
There is a covering of groupoids
\begin{gather}\label{eq:cov}
    \egp(\Gamma)\to \ceg(\Gamma),
\end{gather}
with covering group $\STp(\Gamma)$.
Hence there is an induced short exact sequence
\begin{gather}\label{eq:ses ST}
    1\to \pi_1(\egp(\Gamma), \h_\Gamma) \to
        \pi_1(\ceg(\Gamma),\bfc_\Gamma) \to \STp(\Gamma) \to 1.
\end{gather}
\end{proposition}

For the mutation class of a Dynkin quiver, we can explicitly identify these fundamental groups.

\begin{theorem}\label{thm:QW}
If $\Gamma=\Gamma(Q,0)$ for a Dynkin quiver $Q$, then
$\CBr(\bfc_\Gamma) = \pi_1(\ceg(\Gamma),\bfc_\Gamma)$ and is isomorphic to
the usual braid group $\Br(Q)$.
\end{theorem}

\begin{proof}
\Note{
As explained in Remark~\ref{rem:456},
there are squares and pentagons in the exchange graph $\EGp(\Gamma)$.
}
By \cite[Remark~6.2]{QW},
$\pi_1(\EGp(\Gamma), \h_\Gamma)$ is generated by these squares and pentagons
and hence $\pi_1(\egp(\Gamma), \h_\Gamma)$ is trivial.
Thus \eqref{eq:ses ST} gives an isomorphism
$\pi_1(\ceg(\Gamma),\bfc_\Gamma) \cong \STp(\Gamma)$.
By Remark~\ref{rem:ST0triv} and \cite[Cor.~6.12]{QW},
\[
 \STp(\Gamma)=\ST(\Gamma)\isom \Br(Q).
\]
In this case, Condition~\ref{con:pi1triv} holds by \cite[Prop~4.5]{Q2},
so $\CBr(\bfc_\Gamma)=\pi_1(\ceg(\Gamma),\bfc_\Gamma)$ by Proposition~\ref{pp:ass}.
\end{proof}

Note that the same conclusion holds for $\Gamma=\Gamma(Q,W)$, where $(Q,W)$ is mutation equivalent to a Dynkin quiver, simply by choosing a mutation sequence from it to $(Q,W)$.
However, the isomorphism $\CBr(\bfc_\Gamma)\cong\Br(Q)$ only sends the local twist generators to
standard generators in the case of the quiver itself.

When the mutation class is of type $A_1\times A_1$ or $A_2$,
the universal cover of $\ceg(\Gamma)$ is described in \cite[Example~3.7 and 3.8]{Qs}.

\note{
\begin{remark}
In Section~\ref{sec:stab},
we will see that the exchange graph $\EGp(\Gamma)$ is a \danger{skeleton}
(in the sense explained there)
for the space $\Stap(\Gamma)$ of Bridgeland stability conditions on $\D_{fd}(\Gamma)$.
In Section~\ref{sec:simply} (Proposition~\ref{pp:hex}),
we will show that the squares, pentagons and hexagons in $\EGp(\Gamma)$ are contractible loops in $\Stap(\Gamma)$.
This fact also justifies these relations in the exchange groupoids $\egp(\Gamma)$ and $\ceg(\Gamma)$.
\end{remark}
}

\section{Decorated marked surfaces} \label{sec:KQ}
\subsection{Marked surfaces} \label{sec:mark-surf}

Following \cite{FST}, an \emph{(unpunctured) marked surface} $\surf$
is a compact connected oriented smooth surface
with a finite set $\M$ of marked points on its boundary~$\partial\surf$,
satisfying the following conditions:
\begin{itemize}
\item  $\surf$ is not closed, i.e. $\partial\surf\neq\emptyset$,
\item  each connected component of $\partial\surf$ contains at least one marked point.
\end{itemize}
Up to diffeomorphism, $\surf$ is determined by the following data
\begin{itemize}
\item the genus $g$,
\item the number $b$ of boundary components,
\item the integer partition of $m=\#\M$ into $b$ parts giving the number of marked points
on each boundary component.
\end{itemize}


In this paper, we use the following terminology:
\begin{itemize}
\item an \emph{open arc} is (the isotopy class of) a curve on $\surf$ with endpoints in $\M$
but otherwise in $\surf \setminus \partial\surf$,
which is \emph{simple}, that is, does not intersect itself (except maybe at its endpoints)
and \emph{essential}, that is, not homotopic to a constant arc or a boundary arc
(i.e. an arc in $\partial\surf$),
\item a \emph{curve} is a closed curve in $\surf \setminus \M$,
\item two arcs are \emph{compatible} if they do not intersect (except maybe at their endpoints),
\item an \emph{ideal triangulation} $\RT$ of $\surf$
is a maximal collection of compatible open arcs,
considered up to isotopy.
\end{itemize}
Recall first the following elementary result
(cf. \cite[Prop.~2.10]{FST}).

\begin{proposition}\label{pp:FST}
Any ideal triangulation $\RT$ of $\surf$, consists of
\begin{equation}\label{eq:n}
\numarc=6g-6+3b+m
\end{equation}
open arcs and divides $\surf$ into $\numtri=({2\numarc+m})/{3}$ triangles.
\end{proposition}

Note that the sides of a triangle in a triangulation can be either open arcs or boundary arcs.
As in \cite[\S2]{FST}, when $\surf$ is a disc, i.e. $g=0$ and $b=1$, we will require
$m\geq 4$, so that $\numarc\geq1$ in \eqref{eq:n} and $\numtri\geq 2$.

The \emph{unoriented exchange graph} $\uEG(\surf)$ has vertices
corresponding to ideal triangulations of $\surf$ and edges corresponding to \emph{flips},
as illustrated in Figure~\ref{fig:ord-flip}.
Every open arc in an ideal triangulation is the diagonal of
the quadrilateral formed by the two triangles on either side of it.
Hence every open arc can be flipped and so the exchange graph is $n$-regular.

\begin{figure}[ht]\centering

\begin{tikzpicture}[scale=.6, rotate=180]
\draw (0,0)node {$\bullet$} rectangle (2,2)node {$\bullet$};
\draw (2,0)node {$\bullet$} to (0,2)node {$\bullet$};

\draw(3,1) node [rotate=0,ultra thick] {\Huge{$-$}};

\draw (3+1,0)node {$\bullet$} to (5+1,2)node {$\bullet$};
\draw (3+1,2)node {$\bullet$} rectangle (5+1,0)node {$\bullet$};
\end{tikzpicture}
\caption{An ordinary (unoriented) flip}
\label{fig:ord-flip}
\end{figure}
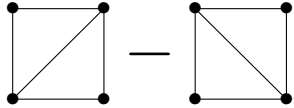

Given a pair of open arcs in an ideal triangulation, one of three things can happen.
\begin{itemize}
\item The two arcs border four distinct triangles, which form two quadrilaterals in the surface and the two flips are independent, giving rise to a `square' in the exchange graph, as on the left of Figure~\ref{fig:Pent}.
\item The two arcs share one triangle, so border three distinct triangles, which form a local pentagon in the surface and repeated flips give rise to a `pentagon' in the exchange graph, as on the right of Figure~\ref{fig:Pent}.
\item The two arcs share two triangles, which form a local annulus in the surface and
repeated flips give rise to a infinite line in the exchange graph,
part of which is shown in Figure~\ref{fig:flat-hex},
where the thick green lines are glued to make the annulus.
\end{itemize}

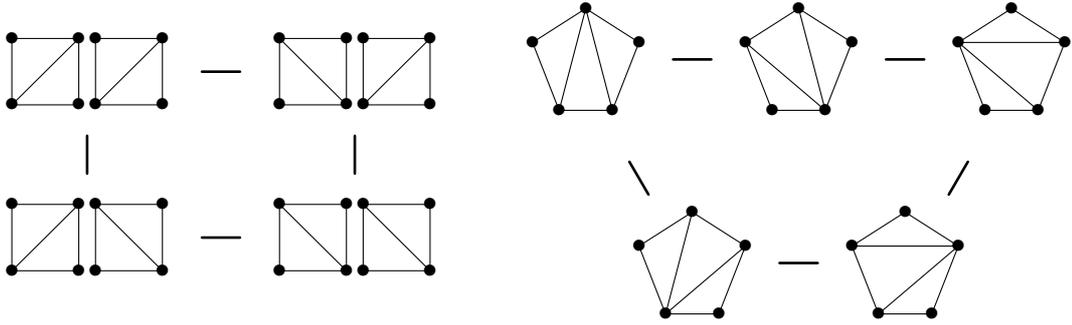
\begin{figure}[ht]\centering
\begin{tikzpicture}[scale=.44, rotate=180]
  \foreach \j in {0,8}
  \foreach \k in {0,5}
  {
    \draw (0+\j,\k+0)node {$\bullet$} rectangle (2+\j,\k+2)node {$\bullet$};
    \draw (2+\j,\k+0)node {$\bullet$}  (0+\j,\k+2)node {$\bullet$};
    \draw (0+2.5+\j,\k+0)node {$\bullet$} rectangle (2+2.5+\j,\k+2)node {$\bullet$};
    \draw (2+2.5+\j,\k+0)node {$\bullet$}  (0+2.5+\j,\k+2)node {$\bullet$};
  }
\draw (0,0)to(2,2);\draw (2+2.5,0)to(0+2.5,2);
\draw (8+0+2.5,0)to(8+2+2.5,2);\draw (8+2,2)to(8+0,0);
\draw (8+0+2.5,5+0)to(8+2+2.5,5+2);\draw (8+2,5+0)to(8+0,5+2);
\draw (0+2.5,5+2)to(2+2.5,5);\draw (2,5+0)to(0,5+2);

\draw(2.25+4,1)node[rotate=0,ultra thick]{\Huge{$-$}};
\draw(2.25+4,1+5)node[rotate=0,ultra thick]{\Huge{$-$}};
\draw(2.25,1+2.5)node[rotate=90,ultra thick]{\Huge{$-$}};
\draw(2.25+8,1+2.5)node[rotate=90,ultra thick]{\Huge{$-$}};
\draw(0,8.5)node{};
\end{tikzpicture}
\qquad
\begin{tikzpicture}[xscale=.35,yscale=.45]
  \foreach \j in {0, 8, 16}
  {
    \draw(0+\j,0)node{$\bullet$}to(2+\j,0)node{$\bullet$}to(3+\j,2)node{$\bullet$}
    to(1+\j,3)node{$\bullet$}to(-1+\j,2)node{$\bullet$}to(0+\j,0);
  }
  \foreach \j in {4, 12}
  {
    \draw(0+\j,0-6)node{$\bullet$}to(2+\j,0-6)node{$\bullet$}to(3+\j,2-6)node{$\bullet$}
    to(1+\j,3-6)node{$\bullet$}to(-1+\j,2-6)node{$\bullet$}to(0+\j,0-6);
  }
\draw(0,0)to(1,3)to(2,0);
\draw(-1+8,2)to(2+8,0)to(1+8,3) (0,4)node{};
\draw(2+8+8,0)to(-1+8+8,2)to(3+8+8,2);
\draw(1+4,3-6)to(0+4,0-6)to(3+4,2-6);
\draw(0+12,0-6)to(3+12,2-6)to(-1+12,2-6);

\draw (1+4,1.5) node[rotate=0,ultra thick]{\Huge{$-$}};
\draw (1+12,1.5) node[rotate=0,ultra thick]{\Huge{$-$}};
\draw (1+2,1-3) node[rotate=-60,ultra thick]{\Huge{$-$}};
\draw (1+6+8,1-3) node[rotate=60,ultra thick]{\Huge{$-$}};
\draw (1+8,1.5-6) node[rotate=0,ultra thick]{\Huge{$-$}};
\end{tikzpicture}
\caption{A square and a pentagon in the unoriented exchange graph}
\label{fig:Pent}
\end{figure}

\tikzset{gluedge/.style={line width=4pt,green!50,>=stealth}}

\begin{figure}[hb]\centering
\begin{tikzpicture}[scale=1, rotate=0]
  \foreach \j/\k in {0/2, -3/2,  3/2, 6/2}
  {
    \draw[gluedge] (\j,\k+0) to (\j,\k+2);
    \draw[gluedge] (\j+2,\k+2) to (\j+2,\k);
    \draw[thick] (\j+0,\k+0)to(\j+2,\k)  (\j+2,\k+2)to(\j,\k+2);
  }
  \foreach \j/\k in {0/2, 3/2}
  {\draw[thick] (\j+0,\k+0)to(\j,\k+2)  (\j+2,\k+0)to(\j+2,\k+2);}

  \foreach \j/\k in {0/2, -3/2, 0/2}
  {\draw[thick] (\j+0,\k+0)to(\j+2,\k+2);}

  \foreach \j/\k in {3/2, 6/2}
  {\draw[thick] (\j+2,\k+0)to(\j+0,\k+2);}

  \foreach \j/\k in {-3/2}
  {\draw[thick] (\j+0,\k+0)to(\j+2,\k+1) (\j+0,\k+1)to(\j+2,\k+2);}

  \foreach \j/\k in {6/2}
  {\draw[thick] (\j+0,\k+2)to(\j+2,\k+1) (\j+0,\k+1)to(\j+2,\k+0);}
\foreach \j/\k in {0/2, -3/2,  3/2, 6/2}
  {\draw(\j,\k)node {$\bullet$}(\j,\k+2)node {$\bullet$}
        (\j+2,\k+0)node {$\bullet$}(\j+2,\k+2)node {$\bullet$};}
\foreach \j in {-3.5, -0.5, 2.5, 5.5,8.5}
{\draw(\j,3)node[rotate=0,ultra thick]{\Huge{$-$}};}
\foreach \j in {-4.25, 9.25}
{\draw(\j,3)node[rotate=0,ultra thick]{\large{$\cdots$}};}

\end{tikzpicture}
\caption{A length three interval of an infinite line in the unoriented exchange graph}
\label{fig:flat-hex}
\end{figure}
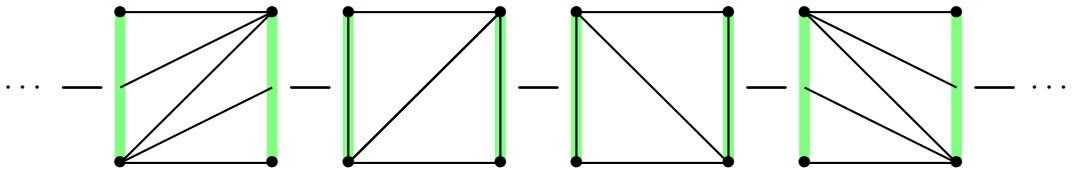

\subsection{The cluster category of a surface} \label{sec:cc-surf}

Let $\surf$ be a marked surface with a triangulation $\RT$.
Then there is an associated quiver $Q_\RT$ (without loops or 2-cycles) with a potential $W_\RT$,
constructed as follows
(cf. \cite[\S4]{FST} and \cite[\S3]{LF}, which also deal with the more general punctured case):
\begin{itemize}
\item the vertices of $Q_\RT$ correspond to the open arcs in $\RT$;
\item for each angle in $\RT$, i.e. a pair of open arcs in the same triangle, there is an arrow
    between the corresponding vertices pointing towards the arc that is a positive (anti-clockwise) rotation of the other;
\item if three open arcs form a triangle, then the corresponding arrows form a 3-cycle in $Q_\RT$
(as in Figure~\ref{fig:type I}) and $W_\RT$ is the sum of all such 3-cycles.
\end{itemize}
If two triangulations are related by a flip, then both the corresponding quivers and quivers with potential are related by mutation (\cite[Prop~4.8]{FST} and \cite[Thm~30]{LF}).
In particular, since $Q_T$ never has loops or 2-cycles, $(Q_\RT,W_\RT)$ is always non-degenerate.

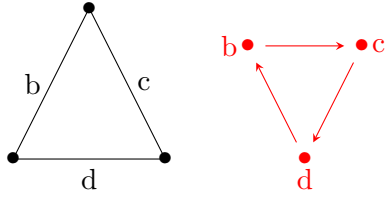
\begin{figure}[ht]\centering
\begin{tikzpicture}[scale=1]
\draw(0,0)node {$\bullet$} to (1,2)node {$\bullet$} to (2,0)node {$\bullet$}
    --cycle;\draw (1,0)node[below]{d}(.5,1)node[left]{b}(1.5,1)node[right]{c};
\end{tikzpicture}\quad
\begin{tikzpicture}[scale=1.5]
\draw[red] (1,0)node[red](d){$\bullet$}node[below]{d}
    (.5,1)node[red](b){$\bullet$}node[left]{b}
    (1.5,1)node[red](c){$\bullet$}node[right]{c};
\draw[red](d)\jiantou(b) (b)\jiantou(c) (c)\jiantou(d);
\end{tikzpicture}
\caption{Part of the quiver with potential associated to a triangle on a marked surface}
\label{fig:type I}
\end{figure}

As proved in \cite{BZ}, this correspondence identifies $\uEG(\surf)$ with the cluster exchange graph of the associated mutation class of quivers with potential.

\begin{theorem}\label{thm:BZ}
Let $\Gamma_\RT$ be the Ginzburg dg algebra $\Gamma(Q_\RT,W_\RT)$, for a triangulation $\RT$ of
a marked surface $\surf$.
There is a bijection $\gamma\mapsto M_\gamma$ between the set of open arcs on $\surf$ and
the set of rigid indecomposables in $\C(\Gamma_\RT)$.
This induces an isomorphism $\uEG(\surf)\cong\uCEG(\Gamma_\RT)$,
sending any triangulation $\{\gamma\}$ to a cluster tilting set $\{M_\gamma\}$.
In particular, it sends the initial triangulation $\RT$ to the canonical cluster tilting set
$\bfc_\RT:=\bfc_{\Gamma_\RT}$.
\end{theorem}

\begin{proof}
See \cite[Corollary~1.6]{BZ} for the isomorphism on vertices
and \cite[Theorem~4.4]{BZ} for the isomorphism on edges.
\end{proof}

\subsection{The exchange groupoid of a surface} \label{sec:EGS}

Just as the exchange graph $\uCEG(\Gamma)$ can be enhanced
to the exchange groupoid $\ceg(\Gamma)$, as in Definition~\ref{def:ceg},
we may enhance $\uEG(\surf)$, which we can think of as a special case,
following Theorem~\ref{thm:BZ}.

In this case, doubling the edges to form the \emph{oriented exchange graph} $\EG(\surf)$
has a geometric interpretation, whose significance will become apparent shortly when we decorate the surface.
More precisely, we consider a \emph{forward flip} to move the endpoints of an open arc~$\gamma$ anti-clockwise
along two sides of the quadrilateral with diagonal~$\gamma$.
Thus each ordinary flip $T - T'$ can be realised by either of two forward flips
$\RT\to\RT'$ or $\RT'\to\RT$, as shown in Figure~\ref{fig:F2-1}.
These are not inverse to each other;
the inverse of a forward flip is a \emph{backward flip}, which moves the endpoints clockwise.

\begin{figure}[ht]\centering
\begin{tikzpicture}[scale=.8, rotate=180]
\draw (0,0)node {$\bullet$} rectangle (2,2)node {$\bullet$};
\draw (2,0)node {$\bullet$} to (0,2)node {$\bullet$};
\draw[blue,->,>=stealth](3-.6,1.3)to(3+.6,1.3);
\draw[blue](3-.25,1.5+.5)rectangle(3+.25,1.5);
\draw[blue,->,>=stealth](3-.25,1.5+.5)to(3-.25,1.5+.15);
\draw[blue,->,>=stealth](3+.25,1.5)to(3+.25,1.5+.35);

\draw[blue,<-,>=stealth](3-.6,.7)to(3+.6,.7);
\draw[blue](3-.25,.5-.5)rectangle(3+.25,.5);
\draw[blue,->,>=stealth](3-.25,.5-.5)to(3+.1,.5-.5);
\draw[blue,->,>=stealth](3+.25,.5)to(3-.1,.5);

\draw (3+1,0)node {$\bullet$} to (5+1,2)node {$\bullet$};
\draw (3+1,2)node {$\bullet$} rectangle (5+1,0)node {$\bullet$};
\end{tikzpicture}
\caption{An ordinary flip becomes two forward flips}
\label{fig:F2-1}
\end{figure}
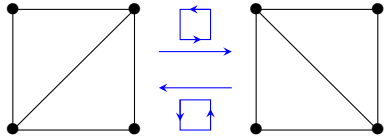

\begin{definition} \label{def:gpd}
The \emph{exchange groupoid} $\eg(\surf)$ of an (unpunctured) marked surface $\surf$
is the quotient of the path groupoid of $\EG(\surf)$ by the following relations:
\begin{itemize}
\item Any square in $\uEG(\surf)$, as on the left of Figure~\ref{fig:Pent},
induces four oriented squares in $\EG(\surf)$,
one starting at each triangulation,
as on the left of Figure~\ref{fig:Pent.rel}.
The square relation is the commutation relation between the two paths from source to sink in this figure.
\item Any pentagon in $\uEG(\surf)$, as on the right of Figure~\ref{fig:Pent},
induces five oriented pentagons in $\EG(\surf)$,
one starting at each triangulation,
as on on the right of Figure~\ref{fig:Pent.rel}.
The pentagon relation is the commutation relation between the two paths from source to sink in this figure.
\item Any length three interval in $\uEG(\surf)$, as in Figure~\ref{fig:flat-hex},
induces a flat hexagon in $\EG(\surf)$, as in Figure~\ref{fig:hexagon}
(cf. \eqref{eq:hex}).
The hexagonal dumbbell relation is the commutation relation
between the two paths from source to sink in this figure.
As before, the thick green edges are glued to form an annulus.
\end{itemize}
\end{definition}

\begin{figure}[ht]\centering
\begin{tikzpicture}[scale=.44,rotate=180]
  \foreach \j in {0,8}
  \foreach \k in {0,5}
  {
    \draw (0+\j,\k+0)node {$\bullet$} rectangle (2+\j,\k+2)node {$\bullet$};
    \draw (2+\j,\k+0)node {$\bullet$}  (0+\j,\k+2)node {$\bullet$};
    \draw (0+2.5+\j,\k+0)node {$\bullet$} rectangle (2+2.5+\j,\k+2)node {$\bullet$};
    \draw (2+2.5+\j,\k+0)node {$\bullet$}  (0+2.5+\j,\k+2)node {$\bullet$};
  }
\draw (0+2.5,0)to(2+2.5,2);\draw (2,0)to(0,2);
\draw (8+0+2.5,0)to(8+2+2.5,2);\draw (8+2,2)to(8+0,0);
\draw (8+0,5+0)to(8+2,5+2);\draw (8+2+2.5,5+0)to(8+0+2.5,5+2);
\draw (0+2.5,5+2)to(2+2.5,5);\draw (2,5+0)to(0,5+2);

\draw[blue,<-,>=stealth](2.25+4-1,1)to(2.25+4+1,1);
\draw[blue](2.25+4-.6,1-.3-.2)rectangle(2.25+4-.1,1-.8-.2);
\draw[blue](2.25+4+.6,1-.3-.2)rectangle(2.25+4+.1,1-.8-.2);
\draw[blue](2.25+4+.6,1-.3-.2)to(2.25+4+.1,1-.8-.2);
\draw[blue,-<-=.7,>=stealth](2.25+4-.6,1-.3-.2)to(2.25+4-.1,1-.3-.2);
\draw[blue,->-=.7,>=stealth](2.25+4-.6,1-.8-.2)to(2.25+4-.1,1-.8-.2);

\draw[blue](2.25+4-.6,1-.3+6.2)rectangle(2.25+4-.1,1-.8+6.2);
\draw[blue](2.25+4+.6,1-.3+6.2)rectangle(2.25+4+.1,1-.8+6.2);
\draw[blue](2.25+4+.6,1-.8+6.2)to(2.25+4+.1,1-.3+6.2);
\draw[blue,-<-=.7,>=stealth](2.25+4-.6,1-.3+6.2)to(2.25+4-.1,1-.3+6.2);
\draw[blue,->-=.7,>=stealth](2.25+4-.6,1-.8+6.2)to(2.25+4-.1,1-.8+6.2);

\draw[blue](2.25+4-.6+5.0,1-.25+2.5)rectangle(2.25+4-.1+5.0,1+.25+2.5);
\draw[blue](2.25+4+.6+5.0,1-.25+2.5)rectangle(2.25+4+.1+5.0,1+.25+2.5);
\draw[blue](2.25+4-.6+5.0,1-.25+2.5)to(2.25+4-.1+5.0,1+.25+2.5);
\draw[blue,-<-=.7,>=stealth](2.25+4+.6+5.0,1-.25+2.5)to(2.25+4+.1+5.0,1-.25+2.5);
\draw[blue,->-=.7,>=stealth](2.25+4+.6+5.0,1+.25+2.5)to(2.25+4+.1+5.0,1+.25+2.5);

\draw[blue](2.25+4-.6-5.0,1-.25+2.5)rectangle(2.25+4-.1-5.0,1+.25+2.5);
\draw[blue](2.25+4+.6-5.0,1-.25+2.5)rectangle(2.25+4+.1-5.0,1+.25+2.5);
\draw[blue](2.25+4-.6-5.0,1+.25+2.5)to(2.25+4-.1-5.0,1-.25+2.5);
\draw[blue,-<-=.7,>=stealth](2.25+4+.6-5.0,1-.25+2.5)to(2.25+4+.1-5.0,1-.25+2.5);
\draw[blue,->-=.7,>=stealth](2.25+4+.6-5.0,1+.25+2.5)to(2.25+4+.1-5.0,1+.25+2.5);

\draw[blue,<-,>=stealth](2.25+4-1,1+5)to(2.25+4+1,1+5);
\draw[blue,<-,>=stealth](2.25,1+2.5+1)to(2.25,1+2.5-1);
\draw[blue,->,>=stealth](2.25+8,1+2.5-1)to(2.25+8,1+2.5+1);
\end{tikzpicture}
\quad
\begin{tikzpicture}[xscale=.35,yscale=.45]
  \foreach \j in {0, 8, 16}
  {
    \draw(0+\j,0)node{$\bullet$}to(2+\j,0)node{$\bullet$}to(3+\j,2)node{$\bullet$}
    to(1+\j,3)node{$\bullet$}to(-1+\j,2)node{$\bullet$}to(0+\j,0);
  }
  \foreach \j in {4, 12}
  {
    \draw(0+\j,0-6)node{$\bullet$}to(2+\j,0-6)node{$\bullet$}to(3+\j,2-6)node{$\bullet$}
    to(1+\j,3-6)node{$\bullet$}to(-1+\j,2-6)node{$\bullet$}to(0+\j,0-6);
  }
\draw(0,0)to(1,3)to(2,0);
\draw(-1+8,2)to(2+8,0)to(1+8,3);
\draw(2+8+8,0)to(-1+8+8,2)to(3+8+8,2);
\draw(1+4,3-6)to(0+4,0-6)to(3+4,2-6);
\draw(0+12,0-6)to(3+12,2-6)to(-1+12,2-6);

\draw[blue,->,>=stealth] (1+4-1.5,1.5)to(1+4+1.5,1.5);
\draw[blue,->,>=stealth] (1+12-1.5,1.5)to(1+12+1.5,1.5);
\draw[blue,->,>=stealth] (1+2-1,1-3+1.5)to(1+2+.5,1-3-1);
\draw[blue,->,>=stealth] (1+6+8-.5,1-3-1)to(1+6+8+1,1-3+1.5);
\draw[blue,->,>=stealth] (1+8-1.5,1.5-6)to(1+8+1.5,1.5-6);

    \path (1+4,1.5+1) coordinate (v);

\draw[xshift=4.5cm,yshift=2cm,blue,>=stealth]
    (0,0)to(2*.4,0)to(3*.4,2*.4)to(1*.4,3*.4)to(-1*.4,2*.4)to(0,0)
    (.8,0)to(.4,1.2)
    (0,0)edge[->-=.7](.8,0) (.4,1.2)edge[->-=.7](-.4,.8);
\draw[xshift=12.5cm,yshift=2cm,blue,>=stealth]
    (0,0)to(2*.4,0)to(3*.4,2*.4)to(1*.4,3*.4)to(-1*.4,2*.4)to(0,0)
    (.8,0)to(-.4,.8)
    (1.2,.8)edge[-<-=.7](.8,0) (.4,1.2)edge[->-=.7](-.4,.8);
\draw[xshift=1cm,yshift=-2.5cm,blue,>=stealth]
    (0,0)to(2*.4,0)to(3*.4,2*.4)to(1*.4,3*.4)to(-1*.4,2*.4)to(0,0)
    (1.2,.8)edge[-<-=.7](.8,0) (.4,1.2)edge[->-=.7](0,0);
\draw[xshift=16.2cm,yshift=-2.5cm,blue,>=stealth]
    (0,0)to(2*.4,0)to(3*.4,2*.4)to(1*.4,3*.4)to(-1*.4,2*.4)to(0,0)
    (0,0)edge[->-=.7](.8,0) (-.4,.8)edge[-<-=.5](1.2,.8);
\draw[xshift=8.5cm,yshift=-6cm,blue,>=stealth]
    (0,0)to(2*.4,0)to(3*.4,2*.4)to(1*.4,3*.4)to(-1*.4,2*.4)to(0,0)
    (1.2,.8)edge[-<-=.7](0,0) (.4,1.2)edge[->-=.7](-.4,.8);
\end{tikzpicture}
\caption{The square and pentagon relation for $\eg(\surf)$}
\label{fig:Pent.rel}
\end{figure}
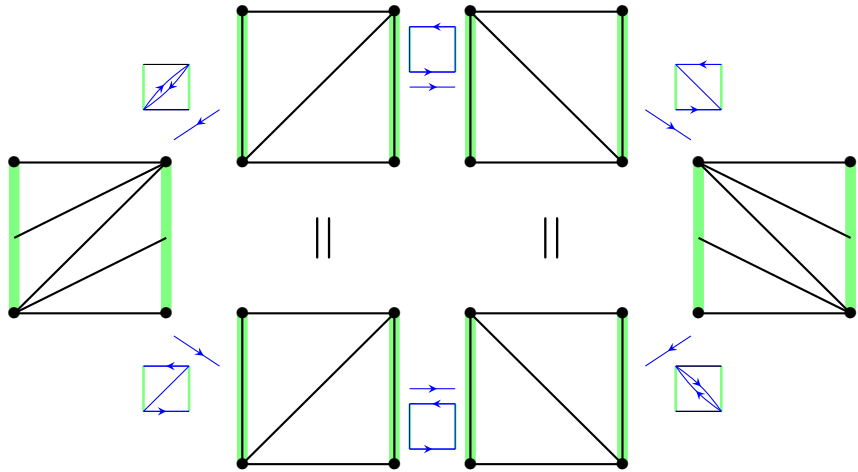
\begin{figure}[hb]\centering
\begin{tikzpicture}[scale=1, rotate=0]

  \foreach \j/\k in {0/0, -3/2, 0/4, 3/0, 6/2, 3/4}
  { \draw[gluedge] (\j,\k+0) to (\j,\k+2);
    \draw[gluedge] (\j+2,\k+2) to (\j+2,\k);
    \draw[thick] (\j+0,\k+0)to(\j+2,\k)  (\j+2,\k+2)to(\j,\k+2);}

  \foreach \j/\k in {0/0, 0/4, 3/0, 3/4}
  {\draw[thick] (\j+0,\k+0)to(\j,\k+2)  (\j+2,\k+0)to(\j+2,\k+2);}

  \foreach \j/\k in {0/0, -3/2, 0/4}
  {\draw[thick] (\j+0,\k+0)to(\j+2,\k+2);}

  \foreach \j/\k in {3/0, 6/2, 3/4}
  {\draw[thick] (\j+2,\k+0)to(\j+0,\k+2);}

  \foreach \j/\k in {-3/2}
  {\draw[thick] (\j+0,\k+0)to(\j+2,\k+1) (\j+0,\k+1)to(\j+2,\k+2);}

  \foreach \j/\k in {6/2}
  {\draw[thick] (\j+0,\k+2)to(\j+2,\k+1) (\j+0,\k+1)to(\j+2,\k+0);}

\foreach \j/\k in {0/0, -3/2, 0/4, 3/0, 6/2, 3/4}
  {\draw(\j,\k)node {$\bullet$}(\j,\k+2)node {$\bullet$}
        (\j+2,\k+0)node {$\bullet$}(\j+2,\k+2)node {$\bullet$};}

\draw[blue,->-=.6,>=stealth] (2.5-.3,1)to(2.5+.3,1);
\draw[blue,->-=.6,>=stealth] (2.5-.3,1+4)to(2.5+.3,1+4);
\draw[blue,-<-=.5,>=stealth] (-0.6+.3,1.5-.2)to(-0.6-.3,1.5+.2);
\draw[blue,->-=.5,>=stealth] (-0.6+.3,4.5+.2)to(-0.6-.3,4.5-.2);
\draw[blue,-<-=.5,>=stealth] (5.6+.3,4.5-.2)to(5.6-.3,4.5+.2);
\draw[blue,->-=.5,>=stealth] (5.6+.3,1.5+.2)to(5.6-.3,1.5-.2);

  \foreach \j/\k in {2.5/0.5, 2.5/5.5, -1/1, -1/5, 6/1, 6/5}
  {
    \draw[line width=1pt,green!50,>=stealth] (\j-.3,\k-.3)to(\j-.3,\k+.3);
    \draw[line width=1pt,green!50,>=stealth] (\j+.3,\k-.3)to(\j+.3,\k+.3);
    \draw[blue] (\j-.3,\k-.3)to(\j+.3,\k-.3)
        (\j+.3,\k+.3)to(\j-.3,\k+.3);
  }
  \foreach \j/\k in {2.5/0.5, 2.5/5.5}
    {\draw[blue,->-=.5,>=stealth](\j+.3,\k+.3)to(\j-.3,\k+.3);
    \draw[blue,->-=.5,>=stealth](\j-.3,\k-.3)to(\j+.3,\k-.3);
    \draw[blue](\j-.3,\k-.3)to(\j-.3,\k+.3)
        (\j+.3,\k-.3)to(\j+.3,\k+.3);}

  \foreach \j/\k in {-1/1}
    {\draw[blue,->-=.5,>=stealth](\j+.3,\k+.3)to(\j-.3,\k+.3);
    \draw[blue,->-=.5,>=stealth](\j-.3,\k-.3)to(\j+.3,\k-.3);
    \draw[blue](\j-.3,\k-.3)to(\j+.3,\k+.3);}

  \foreach \j/\k in {-1/5}
    {\draw(\j+.3,\k+.3)to(\j-.3,\k+.3)(\j-.3,\k-.3)to(\j+.3,\k-.3);
    \draw[blue,->-=.5,>=stealth](\j-.3,\k-.3)to[bend left=10](\j+.3,\k+.3);
    \draw[blue,->-=.5,>=stealth](\j+.3,\k+.3)to[bend left=10](\j-.3,\k-.3);}

  \foreach \j/\k in {6/5}
    {\draw[blue,->-=.5,>=stealth](\j+.3,\k+.3)to(\j-.3,\k+.3);
    \draw[blue,->-=.5,>=stealth](\j-.3,\k-.3)to(\j+.3,\k-.3);
    \draw[blue](\j+.3,\k-.3)to(\j-.3,\k+.3);}

  \foreach \j/\k in {6/1}
    {\draw(\j+.3,\k+.3)to(\j-.3,\k+.3)(\j-.3,\k-.3)to(\j+.3,\k-.3);
    \draw[blue,->-=.5,>=stealth](\j-.3,\k+.3)to[bend left=10](\j+.3,\k-.3);
    \draw[blue,->-=.5,>=stealth](\j+.3,\k-.3)to[bend left=10](\j-.3,\k+.3);}

\draw(1,3)node[rotate=-90,ultra thick]{\Huge{$=$}};
\draw(4,3)node[rotate=-90,ultra thick]{\Huge{$=$}};

\end{tikzpicture}
\caption{The hexagonal dumbbell relation (in a local annulus) for $\eg(\surf)$}
\label{fig:hexagon}
\end{figure}

\begin{remark}\label{rem:compat}
The three cases in Definition~\ref{def:gpd} correspond precisely
to the three cases for a pair of open arcs described at the end of \S\ref{sec:mark-surf}.
Hence, for any triangulation, each pair of open arcs in it will determine one
of the three exchange groupoid relations starting at that triangulation.
However, just as in Remark~\ref{rem:s and p}, if a square or pentagon relation holds,
then a further hexagonal dumbbell relation will be implied.

Indeed, if we associate the quiver $Q_\RT$ to a triangulation $\RT$,
then the square, pentagon and hexagonal dumbbell relations of Definition~\ref{def:gpd}, starting at $\RT$,
coincide with the corresponding relations of Definition~\ref{def:ceg}, starting at $Q_\RT$.
Hence the isomorphism $\EG(\surf)\cong\CEG(\Gamma_\RT)$ implied by Theorem~\ref{thm:BZ}
induces an isomorphism $\eg(\surf)\cong\ceg(\Gamma_\RT)$,
so that $\eg(\surf)$ is a special case of a cluster exchange groupoid.
\end{remark}

To conclude this subsection, we recall the following classical topological result
(see \cite[Thm~3.10]{FST} for the statement and attributions).

\begin{theorem}\label{thm:FST45}
The fundamental group of the exchange graph $\uEG(\surf)$ of (ideal) triangulations
is generated by squares and pentagons.
\end{theorem}

In other words, Condition~\ref{con:pi1triv} holds and so, by Proposition~\ref{pp:ass}, we have
\begin{gather}\label{eq:ct=pi}
    \CBr(\bfc_\RT)=\pi_1(\ceg(\Gamma_\RT),\bfc_\RT)\cong\pi_1(\eg(\surf),\RT).
\end{gather}

\subsection{Decorated marked surfaces} \label{sec:DMS}
Following Krammer \cite{Kr}, we can construct a covering of $\eg(\surf)$
by decorating the marked surface $\surf$.
Recall that any triangulation of $\surf$ consists of $\numtri=(2\numarc+m)/3$ triangles.
\begin{definition}\cite{QQ}\label{def:arcs}
A \emph{decorated marked surface} $\surfo$ is a marked surface $\surf$ together with
a fixed set $\Tri$ of $\numtri$ `decorating' points in the interior of $\surf$.
Moreover, \begin{itemize}
\item An \emph{open arc} in $\surfo$ is (the isotopy class of) a simple essential curve in
$\surfo\setminus\Tri$
that connects two marked points in $\M$.
\item A \emph{triangulation} $\T$ of $\surfo$ is an (isotopy class of)
maximal collection of compatible open arcs (i.e. no intersection in $\surf\smallsetminus\M$) dividing $\surfo$ into $\numtri$ triangles, each of which contains exactly one point in $\Tri$.
\item Forgeting the decorating points gives a \emph{forgetful map}
$F\colon\surfo\to\surf$,
which induces a map from the set of open arcs in $\surfo$
to the set of open arcs in $\surf$,
because isotopy in $\surfo$ implies isotopy in $\surf$.
\item The \emph{forward flip} of a triangulation $\T$ of $\surfo$ is defined, as in \S\ref{sec:EGS},
by moving the endpoints of an open arc $\gamma$ anticlockwise along the quadrilateral
to obtain a new open arc $\gamma^\sharp$, as shown in Figure~\ref{fig:flip.o}.
\item The \emph{exchange graph} $\EG(\surfo)$ is the oriented graph whose vertices
are triangulations of $\surfo$ and whose edges correspond to forward flips between them.
\end{itemize}
\end{definition}
\begin{figure}[ht]\centering
\begin{tikzpicture}[scale=.4]
    \path (-135:4) coordinate (v1)
          (-45:4) coordinate (v2)
          (45:4) coordinate (v3);
\draw[very thick](v1)to(v2)node{$\bullet$}to(v3);
    \path (-135:4) coordinate (v1)
          (45:4) coordinate (v2)
          (135:4) coordinate (v3);
\draw[Emerald, thick](v1)to(v2);
\draw[very thick](v2)node{$\bullet$}to(v3)node{$\bullet$}to
    (v1)node{$\bullet$}(45:1)node[above]{$\gamma$};
\draw[red,thick](135:1.333)node{\tiny{$\circ$}}(-45:1.333)node{\tiny{$\circ$}};
\end{tikzpicture}
\begin{tikzpicture}[scale=1.2, rotate=180]
\draw[blue,<-,>=stealth](3-.6,.7)to(3+.6,.7);
\draw(3,.7)node[below,black]{\footnotesize{in $\surfo$}};
\draw[blue](3-.25,.5-.5)rectangle(3+.25,.5);\draw(3,1.5)node{};
\draw[blue,->,>=stealth](3-.25,.5-.5)to(3+.1,.5-.5);
\draw[blue,->,>=stealth](3+.25,.5)to(3-.1,.5);
\end{tikzpicture}
\begin{tikzpicture}[scale=.4];
    \path (-135:4) coordinate (v1)
          (-45:4) coordinate (v2)
          (45:4) coordinate (v3);
\draw[,very thick](v1)to(v2)node{$\bullet$}to(v3)
(45:1)node[above right]{$\gamma^\sharp$};
    \path (-135:4) coordinate (v1)
          (45:4) coordinate (v2)
          (135:4) coordinate (v3);
\draw[Emerald,,thick](135:4).. controls +(-10:2) and +(45:3) ..(0,0)
                             .. controls +(-135:3) and +(170:2) ..(-45:4);
\draw[,very thick](v2)node{$\bullet$}to(v3)node{$\bullet$}to(v1)node{$\bullet$};
\draw[red,thick](135:1.333)node{\tiny{$\circ$}}(-45:1.333)node{\tiny{$\circ$}};
\end{tikzpicture}

\begin{tikzpicture}[scale=.4]
\draw[thick,>=stealth,->](0,5)to(0,3.6);\draw(0,4.3)node[left]{$^F$};
    \path (-135:4) coordinate (v1)
          (-45:4) coordinate (v2)
          (45:4) coordinate (v3);
\draw[,very thick](v1)to(v2)node{$\bullet$}to(v3);
    \path (-135:4) coordinate (v1)
          (45:4) coordinate (v2)
          (135:4) coordinate (v3);
\draw[,very thick](v2)node{$\bullet$}to(v3)node{$\bullet$}to(v1)node{$\bullet$};
\draw[>=stealth,,thick](-135:4)to(45:4) (45:1)node[above]{$\gamma$};
\end{tikzpicture}
\begin{tikzpicture}[scale=1.2, rotate=180]
\draw(3,1.5)node{}(3,.5)node[below]{\footnotesize{in $\surf$}};
\draw[blue,<-,>=stealth](3-.6,.5)to(3+.6,.5);;
\end{tikzpicture}
\begin{tikzpicture}[scale=.4]
\draw[thick,>=stealth,->](0,5)to(0,3.6);\draw(0,4.3)node[right]{$^F$};
    \path (-135:4) coordinate (v1)
          (-45:4) coordinate (v2)
          (45:4) coordinate (v3);
\draw[,very thick](v1)to(v2)node{$\bullet$}to(v3);
    \path (-135:4) coordinate (v1)
          (45:4) coordinate (v2)
          (135:4) coordinate (v3);
\draw[,very thick](v2)node{$\bullet$}to(v3)node{$\bullet$}to(v1)node{$\bullet$};
\draw[>=stealth,,thick](135:4)to(-45:4) (130:1)node[above]{$\gamma^\sharp$};;
\end{tikzpicture}
\caption{The forward flip in $\EG(\surfo)$ and $\EG(\surf)$}
\label{fig:flip.o}
\end{figure}

\newcommand{\FM}{F_{\mathrm{M}}}

The mapping class group $\MCG(\surfo)$ is the group of isotopy classes of diffeomorphisms of $\surfo$,
where all diffeomorphisms and isotopies fix $\M$ and $\Tri$ setwise.
On the other hand, the mapping class group $\MCG(\surf)$  fixes just $\M$ setwise.
Thus there is a forgetful group homomorphism
\begin{equation}\label{eq:FM}
\FM\colon\MCG(\surfo)\to\MCG(\surf).
\end{equation}
See Lemma~\ref{lem:SBr} for more about this map and its kernel.

\begin{definition}\label{def:gpd.b}
The \emph{exchange groupoid} $\eg(\surfo)$
is the quotient of the path groupoid of $\EG(\surfo)$ by the following
relations, starting at any triangulation $\T$ in $\EG(\surfo)$:
\begin{itemize}
\item if two open arcs are not adjacent
in any triangle of $\T$, then
the forward flips with respect to them form a square in $\EG(\surfo)$, as in Figure~\ref{fig:4+},
and we impose the commuting square relation;
\item if two open arcs are adjacent
in one triangle of $\T$, then they induce an oriented pentagon in $\EG(\surfo)$, as in Figure~\ref{fig:5+},
and we impose the corresponding commuting pentagon relation;
\item if two open arcs are adjacent
in two triangles of $\T$,
then they induce an oriented hexagon in $\EG(\surfo)$, as in Figure~\ref{fig:6+},
and we impose the corresponding commuting hexagon relation.
\end{itemize}
\end{definition}

\begin{figure}[ht]\centering
\begin{tikzpicture}[scale=.6,rotate=180]
\draw[Emerald,thick] (0+2.5,0)to(2+2.5,2)
    (2,0).. controls +(170:4) and +(-10:4) ..(0,2)
    (8+0+2.5,0)to(8+2+2.5,2)
    (8+2,2)to(8+0,0)
    (8+0,5+0)to(8+2,5+2)
    (8+2+2.5,5+0).. controls +(170:4) and +(-10:4) ..(8+0+2.5,5+2)
    (0+2.5,5+2).. controls +(-10:4) and +(170:4) ..(2+2.5,5)
    (2,5+0).. controls +(170:4) and +(-10:4) ..(0,5+2);
\foreach \j in {0,8}
  \foreach \k in {0,5}
  { \draw (0+\j,\k+0)node {$\bullet$} rectangle (2+\j,\k+2)node {$\bullet$};
    \draw (2+\j,\k+0)node {$\bullet$}  (0+\j,\k+2)node {$\bullet$};
    \draw (0+2.5+\j,\k+0)node {$\bullet$} rectangle (2+2.5+\j,\k+2)node {$\bullet$};
    \draw (2+2.5+\j,\k+0)node {$\bullet$}  (0+2.5+\j,\k+2)node {$\bullet$};
    \draw[red] (0.667+\j,\k+1.333)node {$_\circ$}(1.333+\j,\k+0.667)node {$_\circ$}
               (0.667+2.5+\j,\k+1.333)node {$_\circ$}(1.333+2.5+\j,\k+0.667)node {$_\circ$}; }
\draw[blue,<-,>=stealth](2.25+4-1,1)to(2.25+4+1,1);
\draw[blue](2.25+4-.6,1-.3-.2)rectangle(2.25+4-.1,1-.8-.2);
\draw[blue](2.25+4+.6,1-.3-.2)rectangle(2.25+4+.1,1-.8-.2);
\draw[blue](2.25+4+.6,1-.3-.2)to(2.25+4+.1,1-.8-.2);
\draw[blue,-<-=.7,>=stealth](2.25+4-.6,1-.3-.2)to(2.25+4-.1,1-.3-.2);
\draw[blue,->-=.7,>=stealth](2.25+4-.6,1-.8-.2)to(2.25+4-.1,1-.8-.2);

\draw[blue](2.25+4-.6,1-.3+6.2)rectangle(2.25+4-.1,1-.8+6.2);
\draw[blue](2.25+4+.6,1-.3+6.2)rectangle(2.25+4+.1,1-.8+6.2);
\draw[blue](2.25+4+.6,1-.8+6.2)to(2.25+4+.1,1-.3+6.2);
\draw[blue,-<-=.7,>=stealth](2.25+4-.6,1-.3+6.2)to(2.25+4-.1,1-.3+6.2);
\draw[blue,->-=.7,>=stealth](2.25+4-.6,1-.8+6.2)to(2.25+4-.1,1-.8+6.2);

\draw[blue](2.25+4-.6+5.0,1-.25+2.5)rectangle(2.25+4-.1+5.0,1+.25+2.5);
\draw[blue](2.25+4+.6+5.0,1-.25+2.5)rectangle(2.25+4+.1+5.0,1+.25+2.5);
\draw[blue](2.25+4-.6+5.0,1-.25+2.5)to(2.25+4-.1+5.0,1+.25+2.5);
\draw[blue,-<-=.7,>=stealth](2.25+4+.6+5.0,1-.25+2.5)to(2.25+4+.1+5.0,1-.25+2.5);
\draw[blue,->-=.7,>=stealth](2.25+4+.6+5.0,1+.25+2.5)to(2.25+4+.1+5.0,1+.25+2.5);

\draw[blue](2.25+4-.6-5.0,1-.25+2.5)rectangle(2.25+4-.1-5.0,1+.25+2.5);
\draw[blue](2.25+4+.6-5.0,1-.25+2.5)rectangle(2.25+4+.1-5.0,1+.25+2.5);
\draw[blue](2.25+4-.6-5.0,1+.25+2.5)to(2.25+4-.1-5.0,1-.25+2.5);
\draw[blue,-<-=.7,>=stealth](2.25+4+.6-5.0,1-.25+2.5)to(2.25+4+.1-5.0,1-.25+2.5);
\draw[blue,->-=.7,>=stealth](2.25+4+.6-5.0,1+.25+2.5)to(2.25+4+.1-5.0,1+.25+2.5);

\draw[blue,<-,>=stealth](2.25+4-1,1+5)to(2.25+4+1,1+5);
\draw[blue,<-,>=stealth](2.25,1+2.5+1)to(2.25,1+2.5-1);
\draw[blue,->,>=stealth](2.25+8,1+2.5-1)to(2.25+8,1+2.5+1);
\end{tikzpicture}
\caption{The square relation for $\eg(\surfo)$}
\label{fig:4+}
\end{figure}

\begin{figure}[ht]\centering
\begin{tikzpicture}[xscale=.5,yscale=.61]
\draw[Emerald,thick](0,0)to(1,3)to(2,0)
    (-1+8,2).. controls +(5:3) and +(170:3) ..(2+8,0)to(1+8,3)
    (2+8+8,0).. controls +(170:3) and +(5:3) ..(-1+8+8,2).. controls +(15:3.5) and +(-135:2.5) ..(3+8+8,2)
    (1+4,3-6)to(0+4,0-6).. controls +(65:3.6) and +(-120:3) ..(3+4,2-6)
    (0+12,0-6).. controls +(65:3.6) and +(-120:3) ..(3+12,2-6)
    .. controls +(-135:2.5) and +(15:3.5) ..(-1+12,2-6);
  \foreach \j in {0, 8, 16}
  { \draw[thick](0+\j,0)node[]{$\bullet$}to(2+\j,0)node[]{$\bullet$}to
        (3+\j,2)node[]{$\bullet$}
    to(1+\j,3)node[]{$\bullet$}to(-1+\j,2)node[]{$\bullet$}to(0+\j,0);
    \draw(1+\j,1)node[red]{$\circ$}(\j,1.666)node[red]{$\circ$}(2+\j,1.666)node[red]{$\circ$}; }
  \foreach \j in {4, 12}
  { \draw[thick](0+\j,0-6)node[]{$\bullet$}to
        (2+\j,0-6)node[]{$\bullet$}to(3+\j,2-6)node[]{$\bullet$}
    to(1+\j,3-6)node[]{$\bullet$}to(-1+\j,2-6)node[]{$\bullet$}to(0+\j,0-6);
    \draw(1+\j,1-6)node[red]{$\circ$}(\j,1.666-6)node[red]{$\circ$}(2+\j,1.666-6)node[red]{$\circ$}; }

\draw[blue,->,>=stealth] (1+4-1.5,1.5)to(1+4+1.5,1.5);
\draw[blue,->,>=stealth] (1+12-1.5,1.5)to(1+12+1.5,1.5);
\draw[blue,->,>=stealth] (1+2-1,1-3+1.5)to(1+2+.5,1-3-1);
\draw[blue,->,>=stealth] (1+6+8-.5,1-3-1)to(1+6+8+1,1-3+1.5);
\draw[blue,->,>=stealth] (1+8-1.5,1.5-6)to(1+8+1.5,1.5-6);

    \path (1+4,1.5+1) coordinate (v);

\draw[xshift=4.5cm,yshift=2cm,blue,>=stealth]
    (0,0)to(2*.4,0)to(3*.4,2*.4)to(1*.4,3*.4)to(-1*.4,2*.4)to(0,0)
    (.8,0)to(.4,1.2)
    (0,0)edge[->-=.7](.8,0) (.4,1.2)edge[->-=.7](-.4,.8);
\draw[xshift=12.5cm,yshift=2cm,blue,>=stealth]
    (0,0)to(2*.4,0)to(3*.4,2*.4)to(1*.4,3*.4)to(-1*.4,2*.4)to(0,0)
    (.8,0)to(-.4,.8)
    (1.2,.8)edge[-<-=.7](.8,0) (.4,1.2)edge[->-=.7](-.4,.8);
\draw[xshift=1cm,yshift=-2.5cm,blue,>=stealth]
    (0,0)to(2*.4,0)to(3*.4,2*.4)to(1*.4,3*.4)to(-1*.4,2*.4)to(0,0)
    (1.2,.8)edge[-<-=.7](.8,0) (.4,1.2)edge[->-=.7](0,0);
\draw[xshift=16.2cm,yshift=-2.5cm,blue,>=stealth]
    (0,0)to(2*.4,0)to(3*.4,2*.4)to(1*.4,3*.4)to(-1*.4,2*.4)to(0,0)
    (0,0)edge[->-=.7](.8,0) (-.4,.8)edge[-<-=.5](1.2,.8);
\draw[xshift=8.5cm,yshift=-6cm,blue,>=stealth]
    (0,0)to(2*.4,0)to(3*.4,2*.4)to(1*.4,3*.4)to(-1*.4,2*.4)to(0,0)
    (1.2,.8)edge[-<-=.7](0,0) (.4,1.2)edge[->-=.7](-.4,.8);
\end{tikzpicture}
\caption{The pentagon relation for $\eg(\surfo)$}
\label{fig:5+}
\end{figure}

\begin{figure}\centering
\begin{tikzpicture}[scale=1.1]

  \foreach \j/\k in {0/0, -3/2, 0/4, 3/0, 6/2, 3/4}
  { \draw[gluedge] (\j,\k+0) to (\j,\k+2);
    \draw[gluedge] (\j+2,\k+2) to (\j+2,\k);
    \draw[thick] (\j+0,\k+0)to(\j+2,\k)  (\j+2,\k+2)to(\j,\k+2); }

\draw[Emerald,thick]
    plot [smooth,tension=1] coordinates {(2,0)(.7,.5)(2,1)}
    plot [smooth,tension=1] coordinates {(0,2)(1.3,1.5)(0,1)}
    (0,0) to (2,2)

    plot [smooth,tension=1] coordinates {(2+3,0)(.7+3,.5)(2+3,1)}
    plot [smooth,tension=1] coordinates {(0+3,2)(1.3+3,1.5)(0+3,1)}
    plot [smooth,tension=1] coordinates {(0+3,2)(1.5+3,1.5)(.5+3,.5)(2+3,0)}

    plot [smooth,tension=.8] coordinates {(6,4)(7.3,3.5)(6.7,2.5)(8,2)}
    plot [smooth,tension=.8] coordinates {(6,4)(7.5,3.7)(8,3)}
    plot [smooth,tension=.8] coordinates {(8,2)(6.5,2.3)(6,3)}

    (-3,2) to [bend left=12](-1,3) (-3,3) to[bend right=12] (-1,4) to (-3,2)

    (0,6) to (0,4) to (2,6) to (2,4)

    plot [smooth,tension=.8] coordinates {(0+3,2+4)(1.3+3,1.5+4)(.7+3,.5+4)(2+3,0+4)}
    (3,6) to (3,4) (5,6) to (5,4);

\foreach \j/\k in {0/0, -3/2, 0/4, 3/0, 6/2, 3/4}
  {\draw(\j,\k)node {$\bullet$}(\j,\k+2)node {$\bullet$}
        (\j+2,\k+0)node {$\bullet$}(\j+2,\k+2)node {$\bullet$}
  (\j+1,\k+1.5)node[red]{$\circ$}(\j+1,\k+.5)node[red]{$\circ$};}

\draw[blue,->-=.6,>=stealth] (2.5-.3,1)to(2.5+.3,1);
\draw[blue,->-=.6,>=stealth] (2.5-.3,1+4)to(2.5+.3,1+4);
\draw[blue,-<-=.5,>=stealth] (-0.6+.3,1.5-.2)to(-0.6-.3,1.5+.2);
\draw[blue,->-=.5,>=stealth] (-0.6+.3,4.5+.2)to(-0.6-.3,4.5-.2);
\draw[blue,-<-=.5,>=stealth] (5.6+.3,4.5-.2)to(5.6-.3,4.5+.2);
\draw[blue,->-=.5,>=stealth] (5.6+.3,1.5+.2)to(5.6-.3,1.5-.2);

  \foreach \j/\k in {2.5/0.5, 2.5/5.5, -1/1, -1/5, 6/1, 6/5}
  { \draw[line width=1pt,green!50,>=stealth] (\j-.3,\k-.3)to(\j-.3,\k+.3);
    \draw[line width=1pt,green!50,>=stealth] (\j+.3,\k-.3)to(\j+.3,\k+.3);
    \draw[blue] (\j-.3,\k-.3)to(\j+.3,\k-.3)
        (\j+.3,\k+.3)to(\j-.3,\k+.3); }

  \foreach \j/\k in {2.5/0.5, 2.5/5.5}
    {\draw[blue,->-=.5,>=stealth](\j+.3,\k+.3)to(\j-.3,\k+.3);
    \draw[blue,->-=.5,>=stealth](\j-.3,\k-.3)to(\j+.3,\k-.3);
    \draw[blue](\j-.3,\k-.3)to(\j-.3,\k+.3)
        (\j+.3,\k-.3)to(\j+.3,\k+.3);}

  \foreach \j/\k in {-1/1}
    {\draw[blue,->-=.5,>=stealth](\j+.3,\k+.3)to(\j-.3,\k+.3);
    \draw[blue,->-=.5,>=stealth](\j-.3,\k-.3)to(\j+.3,\k-.3);
    \draw[blue](\j-.3,\k-.3)to(\j+.3,\k+.3);}

  \foreach \j/\k in {-1/5}
    {\draw(\j+.3,\k+.3)to(\j-.3,\k+.3)(\j-.3,\k-.3)to(\j+.3,\k-.3);
    \draw[blue,->-=.5,>=stealth](\j-.3,\k-.3)to[bend left=10](\j+.3,\k+.3);
    \draw[blue,->-=.5,>=stealth](\j+.3,\k+.3)to[bend left=10](\j-.3,\k-.3);}

  \foreach \j/\k in {6/5}
    {\draw[blue,->-=.5,>=stealth](\j+.3,\k+.3)to(\j-.3,\k+.3);
    \draw[blue,->-=.5,>=stealth](\j-.3,\k-.3)to(\j+.3,\k-.3);
    \draw[blue](\j+.3,\k-.3)to(\j-.3,\k+.3);}

  \foreach \j/\k in {6/1}
    {\draw(\j+.3,\k+.3)to(\j-.3,\k+.3)(\j-.3,\k-.3)to(\j+.3,\k-.3);
    \draw[blue,->-=.5,>=stealth](\j-.3,\k+.3)to[bend left=10](\j+.3,\k-.3);
    \draw[blue,->-=.5,>=stealth](\j+.3,\k-.3)to[bend left=10](\j-.3,\k+.3);}

\end{tikzpicture}\caption{The hexagon relation for $\eg(\surfo)$}
\label{fig:6+}
\end{figure}

\begin{lemma}\label{lem:Krammer}
The forgetful map $F$ induces a covering functor $F_*\colon\eg(\surfo)\to\eg(\surf)$,
which is a Galois covering with group $\ker \FM$.
\end{lemma}

\begin{proof}
Since $F$ induces a map on open arcs that respects compatibility, it also induces a map $F_*$
on triangulations, which is easily seen to be surjective.
Figure~\ref{fig:flip.o} shows that $F_*$ maps generating morphisms to generating morphisms.
Comparing Figures~\ref{fig:Pent.rel} and Figure~\ref{fig:hexagon}
with Figures~\ref{fig:4+}, ~\ref{fig:5+} and ~\ref{fig:6+},
we see that $F_*$ also respects the relations.

To show that $F_*$ is a covering, we need to check that it has the unique lifting property,
i.e. for any morphism $\alpha\colon k\to l$ in $\eg(\surf)$ and any object $\widehat{k}$ in $\eg(\surfo)$
with $F_*(\widehat{k})=k$, there is a unique morphism $\widehat\alpha\colon \widehat{k}\to \widehat{l}$
in $\eg(\surfo)$ such that $F_*(\widehat{\alpha})=\alpha$.

This property clearly holds for the similar map $F_*\colon\EG(\surfo)\to\EG(\surf)$,
because each graph is $(n,n)$ regular
\Note{(i.e. there are $n$ arrows going out and $n$ arrows coming in at each vertex)}
and $F_*$ locally matches edges.
In addition, each square, pentagon and hexagon
relation in $\EG(\surf)$ lifts to
a corresponding relation in $\EG(\surfo)$ with an arbitrary source lifting.
Thus paths in $\EG(\surf)$ that are equal in $\eg(\surf)$ lift to paths in $\EG(\surfo)$
that are equal in $\eg(\surfo)$.

By the Alexander method (cf. \cite[\S2.3]{FM}),
the mapping class group $\MCG(\surfo)$ acts freely on $\EG(\surfo)$.
On the other hand, any two decorated triangulations of $\surfo$ mapping by $F$ to the same triangulation of $\surf$ are related by the action of an element of $\MCG(\surfo)$ in $\ker \FM$.
Since $\ker \FM$ preserves the fibres, we see that it is the Galois group of the covering.
\end{proof}

Note that $\EG(\surfo)$, and hence $\eg(\surfo)$, are usually not connected,
but all connected components are isomorphic by Lemma~\ref{lem:Krammer}.
Consequently, for any initial triangulation $\T$ of $\surfo$,
we define $\EGT(\surfo)$ and $\egt(\surfo)$ to be
the connected components of $\EG(\surfo)$ and $\eg(\surfo)$, respectively, that contain $\T$.

\subsection{The braid twist group}

\begin{definition}
A \emph{closed arc} in $\surfo$ is (the isotopy class of) a simple curve in
the interior of $\surfo$ that connects two decorating points in $\Tri$.
Denote by $\cA(\surfo)$ the set of closed arcs on $\surfo$.

Let $\T$ be a triangulation of $\surfo$ consisting of $\numarc$ open arcs.
The \emph{dual graph} $\T^*$ of $\T$ is the collection of $\numarc$ closed arcs
in $\surfo$ with the property that each closed arc intersects just one different open arc in $\T$
once (see Figure~\ref{fig:ex0} for an example).
More precisely, for $\gamma$ in $\T$, the corresponding closed arc $\eta$ in $\T^*$
is the one contained in the quadrilateral $A$ with diagonal $\gamma$,
connecting the two decorating points in $A$ and intersecting $\gamma$ only once.
We say that $\eta$ and $\gamma$ are dual to each other with respect to $\T$.
\end{definition}

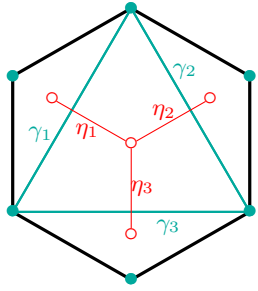
\begin{figure}[t]\centering
\begin{tikzpicture}[scale=.30,rotate=-120]
\foreach \j in {1,...,6}{\draw[very thick](60*\j+30:6)to(60*\j-30:6);}
\foreach \j in {1,...,6}{\draw[Emerald,thick](120*\j-30:6)node{$\bullet$}to(120*\j+90:6)
    (120*\j-90:6)node{$\bullet$};}
\foreach \j in {1,...,3}{  \draw[red](30+120*\j:4)to(0,0);}
\foreach \j in {1,...,3}{  \draw(30-120*\j:4)node[white]{$\bullet$}node[red]{$\circ$};
    \draw(30-120*\j+14:2)node[red]{\text{\footnotesize{$\eta_\j$}}};
    \draw(30-120*\j+24:4)node[Emerald]{\text{\footnotesize{$\gamma_\j$}}};}
\draw(0,0)node[white]{$\bullet$}node[red]{$\circ$};
\end{tikzpicture}
  \caption{The dual graph of a triangulation}
  \label{fig:ex0}
\end{figure}

\begin{definition}
For any closed arc $\eta\in\cA(\surfo)$, there is a (positive) \emph{braid twist}
$\Bt{\eta}\in\MCG(\surfo)$ along $\eta$, as shown in Figure~\ref{fig:Braid twist}.

The \emph{braid twist group} $\BT(\surfo)$ of the decorated marked surface $\surfo$
is the subgroup of $\MCG(\surfo)$ generated by the braid twists $\Bt{\eta}$
for all $\eta\in\cA(\surfo)$.
\end{definition}

\begin{figure}[ht]\centering
\begin{tikzpicture}[scale=.225]
  \draw[very thick,NavyBlue](0,0)circle(6)node[above,black]{$_\eta$};
  \draw(-120:5)node{+};
  \draw(-2,0)edge[red, very thick](2,0)  edge[cyan,very thick, dashed](-6,0);
  \draw(2,0)edge[cyan,very thick,dashed](6,0);
  \draw(-2,0)node[white] {$\bullet$} node[red] {$\circ$};
  \draw(2,0)node[white] {$\bullet$} node[red] {$\circ$};
  \draw(0:7.5)edge[very thick,->,>=stealth](0:11);
\end{tikzpicture}\;
\begin{tikzpicture}[scale=.225]
  \draw[very thick, NavyBlue](0,0)circle(6)node[below,black]{$_\eta$};
  \draw[red, very thick](-2,0)to(2,0);
  \draw[cyan,very thick, dashed](2,0).. controls +(0:2) and +(0:2) ..(0,-2.5)
    .. controls +(180:1.5) and +(0:1.5) ..(-6,0);
  \draw[cyan,very thick,dashed](-2,0).. controls +(180:2) and +(180:2) ..(0,2.5)
    .. controls +(0:1.5) and +(180:1.5) ..(6,0);
  \draw(-2,0)node[white] {$\bullet$} node[red] {$\circ$};
  \draw(2,0)node[white] {$\bullet$} node[red] {$\circ$};
\end{tikzpicture}
\caption{The braid twist $\Bt{\eta}$}
\label{fig:Braid twist}
\end{figure}
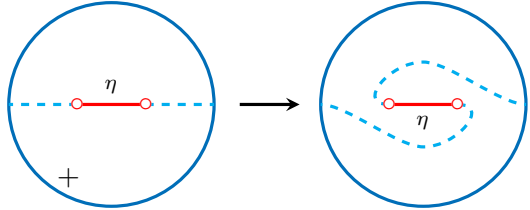

Given a triangulation $\T$ of $\surfo$,
the group $\BT(\surfo)$ has a natural set of generators
$\{\Bt{\eta}\mid \eta\in\T^*\}$ (\cite[Lemma~4.2]{QQ})
and we will also write the group as $\BT(\T)$
to indicate this choice of generators.

Observe that the composition of a pair of forward/backward flips in $\EG(\surfo)$ is
the change of triangulation of $\surfo$ induced by the negative/positive braid twist of
the dual closed arc (cf. Figure~\ref{fig:flips}).
Hence $\BT(\surfo)$ acts on each connected component $\EGT(\surfo)$.
In fact, we have the following.

\begin{lemma}\label{lem:covering}
The restricted covering $F_*\colon \egt(\surfo)\to \eg(\surf)$ (cf. Lemma~\ref{lem:Krammer})
is a Galois covering with group $\BT(\surfo)$.
\end{lemma}

\begin{proof}
The group $\BT(\surfo)$ does not alter the underlying triangulation of $\surf$,
so acts on the fibres of the restricted covering.
Given Lemma~\ref{lem:Krammer}, we just need to show that it acts transitively on these fibres.

Take any path $\hat p\colon \T_1\to \T_2$ in $\egt(\surfo)$ with end points in the same fibre,
mapping down to $p\colon T \to T$ in $\eg(\surf)$.
By \eqref{eq:ct=pi}, $p$ can be expressed (up to homotopy, i.e. groupoid relations) as a product of local twists at $T$.
Lifted to $\egt(\surfo)$, these local twists become braid twists.
In other words, there is an element of $\BT(\surfo)$ mapping $\T_1$ to $\T_2$, as required.
\end{proof}

\begin{figure}[ht]\centering
\begin{tikzpicture}[xscale=-.4,yscale=.4,rotate=90]
    \path (-135:4) coordinate (v1)
          (-45:4) coordinate (v2)
          (45:4) coordinate (v3);
\draw[Emerald,very thick](v1)to(v2)node{$\bullet$}to(v3);
    \path (-135:4) coordinate (v1)
          (45:4) coordinate (v2)
          (135:4) coordinate (v3);
\draw[Emerald,very thick](v2)node{$\bullet$}to(v3)node{$\bullet$}to(v1)node{$\bullet$};
\draw[>=stealth,Emerald,thick](135:4).. controls +(-10:2) and +(45:3) ..(0,0)
                             .. controls +(-135:3) and +(170:2) ..(-45:4);
\draw[red,thick](135:1.333)node{\tiny{$\circ$}}(-45:1.333)node{\tiny{$\circ$}};
\end{tikzpicture}
\begin{tikzpicture}[scale=1.2, rotate=180]
\draw[blue,<-,>=stealth](3-.6,2.2)to(3+.6,2.2);
\draw[blue](3-.25,1.5+.5)rectangle(3+.25,1.5);\draw(3,3)node{};
\draw[blue,->,>=stealth](3-.25,1.5+.5)to(3-.25,1.5+.15);
\draw[blue,->,>=stealth](3+.25,1.5)to(3+.25,1.5+.35);
\end{tikzpicture}
\begin{tikzpicture}[scale=.4]
    \path (-135:4) coordinate (v1)
          (-45:4) coordinate (v2)
          (45:4) coordinate (v3);
\draw[Emerald,very thick](v1)to(v2)node{$\bullet$}to(v3);
    \path (-135:4) coordinate (v1)
          (45:4) coordinate (v2)
          (135:4) coordinate (v3);
\draw[Emerald,very thick](v2)node{$\bullet$}to(v3)node{$\bullet$}to(v1)node{$\bullet$};
\draw[>=stealth,Emerald,thick](-135:4)to(45:4);
\draw[red,thick](135:1.333)node{\tiny{$\circ$}}(-45:1.333)node{\tiny{$\circ$}};
\end{tikzpicture}
\begin{tikzpicture}[scale=1.2, rotate=180]
\draw[blue,<-,>=stealth](3-.6,.7)to(3+.6,.7);
\draw[blue](3-.25,.5-.5)rectangle(3+.25,.5);\draw(3,1.5)node{};
\draw[blue,->,>=stealth](3-.25,.5-.5)to(3+.1,.5-.5);
\draw[blue,->,>=stealth](3+.25,.5)to(3-.1,.5);
\end{tikzpicture}
\begin{tikzpicture}[scale=.4]
    \path (-135:4) coordinate (v1)
          (-45:4) coordinate (v2)
          (45:4) coordinate (v3);
\draw[Emerald,very thick](v1)to(v2)node{$\bullet$}to(v3);
    \path (-135:4) coordinate (v1)
          (45:4) coordinate (v2)
          (135:4) coordinate (v3);
\draw[Emerald,very thick](v2)node{$\bullet$}to(v3)node{$\bullet$}to(v1)node{$\bullet$};
\draw[>=stealth,Emerald,thick](135:4).. controls +(-10:2) and +(45:3) ..(0,0)
                             .. controls +(-135:3) and +(170:2) ..(-45:4);
\draw[red,thick](135:1.333)node{\tiny{$\circ$}}(-45:1.333)node{\tiny{$\circ$}};
\end{tikzpicture}
\caption{The composition of forward flips}
\label{fig:flips}
\end{figure}
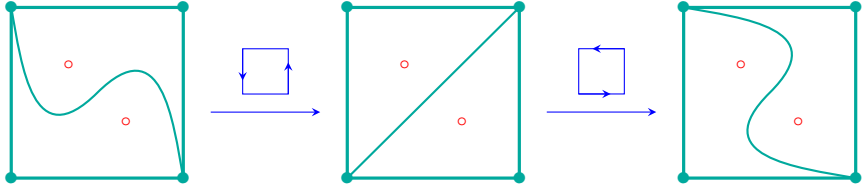

 \AKedit{Combining and enhancing results from the prequels to this paper,
we record the following closely related results, for later use.}

\begin{theorem}\label{thm:QQ}
There is an isomorphism
\begin{gather}\label{eq:iota_T}
    \iota_\T\colon \BT(\T)\to\ST(\Gamma_\RT)
\end{gather}
sending generating braid twists to the generating spherical twists
and $\ST(\Gamma_\RT)$ acts faithfully on $\EGp(\Gamma_\RT)$.
\end{theorem}

\begin{proof}
In \cite[Thm~1]{QQ} the isomorphism is proved replacing $\ST(\Gamma_\RT)$ by its quotient
$\STp(\Gamma_\RT)$, as defined in Theorem~\ref{thm:KN},
which acts faithfully on $\Sph(\Gamma_\RT)$ or equivalently on $\EGp(\Gamma_\RT)$.
\AKedit{
However we now know that $\STp(\Gamma_\RT)=\ST(\Gamma_\RT)$ by Remark~\ref{rem:ST0triv}.
}
\end{proof}


\begin{theorem}\label{cor:QQ}
The isomorphism $\uEG(\surf)\cong\uCEG(\Gamma_\RT)$ of Theorem~\ref{thm:BZ} lifts to
an isomorphism
\begin{gather}\label{eq:EGp=EGT}
    \EGT(\surfo) \cong \EGp(\Gamma_\RT),
\end{gather}
mapping the decorated triangulation $\T$ with underlying ordinary triangulation $\RT$ to the canonical heart $\h_{\Gamma_\RT}$.
Furthermore, this induces an isomorphism
\begin{gather}\label{eq:egp=egt}
    \egt(\surfo) \cong \egp(\Gamma_\RT).
\end{gather}
\end{theorem}

\begin{proof}
The first isomorphism \eqref{eq:EGp=EGT} is proved
(just using \cite[Thm~1]{QQ}) in \cite[Prop~3.2]{QQ2} for the silting exchange graph,
but the proof applies (even more naturally) to $\EGp(\Gamma_\RT)$.
The compatibility of squares, pentagons and hexagons in $\EG(\surf)$ and $\CEG(\Gamma_\RT)$
discussed in Remark~\ref{rem:compat}, lifts to
$\EGT(\surfo)$  and $\EGp(\Gamma_\RT)$,
so \eqref{eq:EGp=EGT} induces \eqref{eq:egp=egt}, as required.
\end{proof}

\subsection{Fundamental groups}

We now proceed to prove a theorem analogous to Theorem~\ref{thm:FST45}, but for the cover $\EGT(\surfo)$.
In this case, we also have to include hexagons.
We need to import a result from a prequel concerning various types of relations in $\BT(\T)$.

\begin{definition}\label{def:app}\cite[Definition~5.1]{QZ3}
Let $(Q_\T,W_\T)$ be the quiver with potential
associated to a triangulation $\T$ of $\surfo$.
Define the braid group $\Br(Q_\T,W_\T)$ as follows:
the generators are identified with the vertices of $Q_\T$, and thus with
the open arcs in $\T$, and the relations are
\begin{enumerate}
  \item $\Crel(a,b)$, that is, $ab=ba$, if there is no arrow between them.
  \item $\Brel(a,b)$, that is, $aba=bab$, if there is exactly one arrow between them.
  \item $\Crel(a^b,c)$, where $$a^b\colon =b^{-1}ab,$$ if the full subquiver between $a,b,c$ is
  the first quiver in Figure~\ref{eq:mut} and this 3-cycle contributes a term in $W_\T$.
  \item $\Brel(a^b,c)$ if the full subquiver between $a,b,c$ is
  the second quiver in Figure~\ref{eq:mut} and (exactly) one 3-cycle between them
  contributes a term in $W_\T$.
  \item $\Crel(c^{ae},b)$ if the full subquiver between $a,e,b,c$ is
  the third quiver in Figure~\ref{eq:mut} and one 3-cycle between $\{a,b,c\}$ and one 3-cycle between $\{e,b,c\}$
  contribute terms in $W_\T$, in such a way that the arrows $b\to c$ in these two 3-cycles are different.
  \item $\Brel(c^{ae},b)$ and $\Brel(c^{ea},b)$
  if the full subquiver between $a,b,c,e$ is
  the forth quiver in Figure~\ref{eq:mut} and
  the 3-cycles between $\{a,b,c\}$ and $\{e,b,c\}$ contribute terms in $W_\T$ in the same way as
  the previous case.
  \item if, in the previous case, there is also a unique 3-cycle between $\{a,e,f\}$, which
  also contributes a term in $W_\T$, then there is an additional relation $\Crel(e,f^{abc})$.
\end{enumerate}
\end{definition}

\begin{figure}[ht]
\begin{tikzpicture}[scale=0.5,
  arrow/.style={->,>=stealth},
  equalto/.style={double,double distance=2pt},
  mapto/.style={|->}]
\node (x1) at (0,2){};
\node (x2) at (2,-1){};
\node (x3) at (-2,-1){};
\draw[white, fill=gray!11] (2,-1)to(-2,-1)to(0,2)to(2,-1);
  \node at (x1){$a$};
  \node at (x2){$b$};
  \node at (x3){$c$};
  \foreach \n/\m in {1/2,2/3,3/1}
    {\draw[arrow]
     (x\n) to (x\m);}
\end{tikzpicture}
\quad
\begin{tikzpicture}[scale=0.5,
  arrow/.style={->,>=stealth},
  equalto/.style={double,double distance=2pt},
  mapto/.style={|->}]
\node (x2) at (0,2){};
\node (x1) at (2,-1){};
\node (x3) at (-2,-1){};
\draw[white, fill=gray!11] (2,-.8)to(-2,-.8)to(0,2)--cycle;
  \node at (x1){$b$};
  \node at (x2){$a$};
  \node at (x3){$c$};
  \foreach \n/\m in {3/2,2/1}
    {\draw[arrow]
     (x\n) to (x\m);}
\draw[arrow] (x1.150) to (x3.30);
\draw[arrow] (x1.-150)to (x3.-30);
\end{tikzpicture}
\quad
\begin{tikzpicture}[xscale=0.5,yscale=.5,
  arrow/.style={->,>=stealth},
  equalto/.style={double,double distance=2pt},
  mapto/.style={|->}]
\node (x4) at (2,2){};
\node (x1) at (2,-1){};
\node (x3) at (-2,-1){};
\node (x2) at (-2,2){};
\draw[white, fill=gray!11] (2,-.8)to(-2,-.8)to(2,2)--cycle;
\draw[white, fill=gray!11] (2,-1.2)to(-2,-1.2)to(-2,2)--cycle;
  \node at (x1){$b$};
  \node at (x2){$a$};
  \node at (x3){$c$};
  \node at (x4){$e$};
  \foreach \n/\m in {3/2,2/1,3/4,4/1}
    {\draw[arrow]
     (x\n) to (x\m);}
\draw[arrow] (x1.150) to (x3.30);
\draw[arrow] (x1.-150)to (x3.-30);
\end{tikzpicture}
\quad
\begin{tikzpicture}[xscale=0.5,yscale=.5,
  arrow/.style={->,>=stealth},
  equalto/.style={double,double distance=2pt},
  mapto/.style={|->}]
\node (x4) at (2,2){};
\node (x1) at (2,-1){};
\node (x3) at (-2,-1){};
\node (x2) at (-2,2){};
\draw[white, fill=gray!11] (2,-1)to(-2,-1)to(2,2)--cycle;
\draw[white, fill=gray!11] (2,-1)to(-2,-1)to(-2,2)--cycle;
  \node at (x1){$b$};
  \node at (x2){$a$};
  \node at (x3){$c$};
  \node at (x4){$e$};
  \foreach \n/\m in {3/2,2/1,3/4,4/1,2/4}
    {\draw[arrow]
     (x\n) to (x\m);}
\draw[arrow] (x1.150) to (x3.30);
\draw[arrow] (x1.-150)to (x3.-30);
\end{tikzpicture}
\quad
\begin{tikzpicture}[xscale=0.5,yscale=.5,
  arrow/.style={->,>=stealth},
  equalto/.style={double,double distance=2pt},
  mapto/.style={|->}]
\node (x4) at (2,2){};
\node (x1) at (2,-1){};
\node (x3) at (-2,-1){};
\node (x2) at (-2,2){};
\node (x5) at (0,3){};
\draw[white, fill=gray!11] (2,2)to(-2,2)to(0,3)--cycle;
\draw[white, fill=gray!11] (2,-1)to(-2,-1)to(2,2)--cycle;
\draw[white, fill=gray!11] (2,-1)to(-2,-1)to(-2,2)--cycle;
  \node at (x1){$b$};
  \node at (x2){$a$};
  \node at (x3){$c$};
  \node at (x4){$e$};
  \node at (x5){$f$};
  \foreach \n/\m in {3/2,2/1,3/4,4/1,5/2,4/5,2/4}
    {\draw[arrow]
     (x\n) to (x\m);}
\draw[arrow] (x1.150) to (x3.30);
\draw[arrow] (x1.-150)to (x3.-30);
\end{tikzpicture}
\caption{Five cases of (sub-)quivers with potential}\label{eq:mut}
\end{figure}
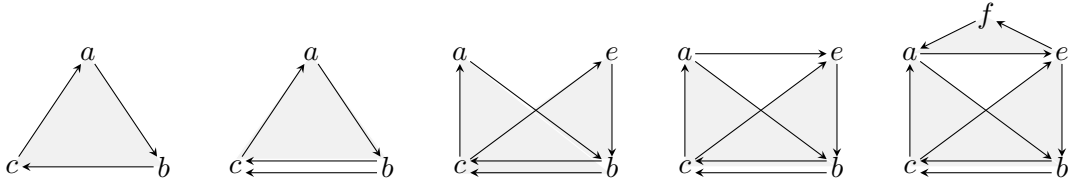
Note that all the terms in the potentials mentioned above are necessary
\Note{(and locally they are sufficient)}
to ensure that that quivers with potential appearing are non-degenerate,
i.e. that any 2-cycles created by mutation will be cancelled by a corresponding quadratic term in the mutated potential.
Thus the corresponding relations are in effect determined just by the full sub-quivers appearing.

\begin{theorem}\label{thm:QZ}\cite[Theorem~5.3]{QZ3}
The braid twist group $\BT(\T)$ is isomorphic
to the braid group $\Br(Q_\T,W_\T)$ above,
where the braid twist $\Bt{\eta}$ corresponds to the dual open arc of $\eta$ in $\T$,
for any $\eta\in\T^*$.
Thus we obtain an explicit finite presentation of $\BT(\T)$.
\end{theorem}

\begin{proposition}\label{cor:relations}
If the associated quiver of a cluster $\bfc$ in $\CEG(\Gamma)$ has a full sub-quiver in
one of the cases of Definition~\ref{def:app},
then the corresponding relation there holds in the cluster braid group $\CBr(\bfc)$.
\end{proposition}

\begin{proof}
The first two relations are Lemma~\ref{lem:12}.
The remaining relations follow from the conjugation formula in Proposition~\ref{pp:conj},
after applying one or more mutations and using the first two relations.
The precise arguments are the same as for \cite[Prop~5.3]{QZ3}.
For instance, in the third case (the leftmost quiver in Figure~\ref{eq:mut}), after mutating at $b$,
$a$ becomes $a^b$, $c$ does not change and there is no arrow between
the corresponding vertices
in the mutated quiver.
Hence, the commutation relation $\Crel(a^b,c)$ holds.
\end{proof}

\begin{theorem}\label{thm:45}
The fundamental group of the exchange graph $\EGT(\surfo)$ of decorated triangulations
is generated by squares, pentagons and hexagons.
Equivalently, $\egt(\surfo)$ is simply connected or
\[ \pi_1(\eg(\surf),\RT)=\BT(\T). \]
\end{theorem}

\begin{proof}
Consider the Galois covering $F_*\colon \egt(\surfo)\to \eg(\surf)$ with covering group $\BT(\surfo)$,
as in Lemma~\ref{lem:covering}.
Then there is a covering sequence:
\begin{equation}\label{eq:ses2}
  1\to \pi_1(\egt(\surfo),\T) \to \pi_1(\eg(\surf),\RT) \xrightarrow{i_\T} \BT(\T) \to 1.
\end{equation}
Denote by $\{\gamma_i\}$ the open arcs in $\T$ and their dual closed arcs by $\{\eta_i\}$.
Since $\pi_1(\eg(\surf),\RT)$ is generated by local twists, by \eqref{eq:ct=pi},
the map $i_\T$ is determined by sending the local twist with respect to an open arc $F(\gamma_i)$ in $\RT$
to the braid twist $\Bt{\eta_i}^{-1}$
(cf. Figure~\ref{fig:flips}).

A finite presentation of $\BT(\T)$ is given in
Theorem~\ref{thm:QZ} with respect to the standard generators $\{\Bt{\eta}\mid \eta\in\T^*\}$.
The (generating) relations are precisely those given in Definition~\ref{def:app}
and we know, by Proposition~\ref{cor:relations}, that these relations are satisfied by the local twists
in $\CBr(\bfc_\RT)$.
Hence $i_\T$ has a well-defined inverse, also taking generators to generators, and so
$i_\T$ is an isomorphism and $\pi_1(\egt(\surfo),\T)=1$, as required.
\end{proof}

One consequence of Theorem~\ref{thm:45} is that the covering $F_*\colon \egt(\surfo)\to \eg(\surf)$
is the universal cover of the groupoid $\eg(\surf)$.
\note{
Note that the short exact sequence \eqref{eq:ses2} can in fact be identified with \eqref{eq:ses ST}.
}
\goodbreak
\section{Stability conditions via quadratic differentials}
\subsection{Quadratic differentials}
We recall the relationship between marked surfaces and quadratic differential, following \cite{BS}.
Let $\xx$ be a compact Riemann surface and $\omega_\xx$
be its holomorphic cotangent bundle.
A \emph{meromorphic quadratic differential} $\phi$ on $\xx$ is a meromorphic section
of the line bundle $\omega_{\xx}^{2}$.
In terms of a local coordinate $z$ on $\xx$,
such a $\phi$  can be written as $\phi(z)=g(z)\, \dd z^2$,
where $g(z)$ is a meromorphic function.

We will only consider \emph{GMN differentials} $\phi$ on $\xx$,
which, in our case (i.e. an unpunctured marked surface),
are meromorphic quadratic differential such that
\begin{itemize}
\item all zeroes of $\phi$ are simple.
\item every pole of $\phi$ has order at least three.
\end{itemize}
Denote by $\Zer(\phi)$ the set of zeroes of $\phi$,
$\Pol_j(\phi)$ the set of poles of $\phi$ with order $j$
and $\Crit(\phi)=\Zer(\phi)\cup\Pol(\phi)$.

At a point of $\xx^\circ=\xx \setminus \Crit(\phi)$,
there is a  distinguished local coordinate $\omega$,
uniquely defined up to transformations of the form $\omega \mapsto \pm\, \omega+\operatorname{const}$,
with respect to which
$\phi(\omega)=\dd \omega \otimes \dd \omega$.
In terms of a local coordinate $z$, we have $\omega=\int \sqrt{g(z)}\dd z$.
A GMN differential $\phi$ on $\xx$ determines the $\phi$-metric on $\surp$, which is defined locally
by pulling back the  Euclidean metric on $\CC$ using a distinguished coordinate $\omega$.
Thus, there are geodesics on $\surp$ and each geodesics have a constant phase with respect to $\omega$.

A (horizontal) \emph{trajectory} of a GMN differential $\phi$ on $\surp$
is a maximal horizontal geodesic $\gamma\colon(0,1)\to\surp$,
with respect to the $\phi$ metric.
When $\lim\gamma(t)$ exists in $\xx$ as $t\to0$ (resp. $t\to1$),
the limits are called the left (resp. right) endpoint of $\gamma$.
The trajectories of a meromeorphic quadratic differential $\phi$
provide the \emph{horizontal foliation} on $\xx$.
There are several cases (cf. \cite[\S~3]{BS}.
\begin{itemize}
\item For a simple zero of $\phi$ on $\xx$, the local trajectory structure is shown in
Figure~\ref{fig:simple zp}.
\begin{figure}[ht]\centering
\begin{tikzpicture}[scale=.4]
\foreach \k in {1,2,0}
{    \path (120*\k+30+210:4) coordinate (v2)
          (120*\k+120+30+210:4) coordinate (v1)
          (120*\k+60+30+210:2) coordinate (v3)
          (0,0) coordinate (v4);
  \draw[thick](v2)to(v4)to(v1);
  \foreach \j in {.36,.54,.72}
    {
      \path (v4)--(v3) coordinate[pos=\j] (m0);
      \draw[Emerald] plot [smooth,tension=.5] coordinates {(v1)(m0)(v2)};
    }
}
\foreach \k in {1,2,0}{\draw[white,fill=white](120*\k+30+210:4)circle(.6);}
\draw[red,fill=white](0,0)circle(.1);
\end{tikzpicture}
\caption{Local trajectories at a simple zero}\label{fig:simple zp}
\end{figure}
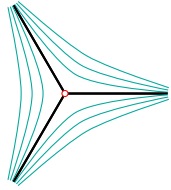
\item For poles of $\phi$ with order $3,4,5$ on $\xx$,
the local trajectory structures are shown in the pictures of Figure~\ref{fig:poles}.
In general, there is a neighbourhood $U$ of each pole $P$ (with order $m$) in $\xx$
and a collection of $m-2$ distinguished tangent directions $\{v_j\}$ at $P$,
such that any trajectory entering $U$ will eventually tends to $P$ and
becomes asymptotic to one of the $v_j$.
\end{itemize}

\begin{figure}[ht]\centering
\begin{tikzpicture}[scale=.32,rotate=0]
\draw[dashed, thin](0,0)circle(4);
\draw[very thick](0,0)to(0:4)(-90:4.5);
\draw[Emerald] (0,0)
    .. controls +(0:2) and +(195:1) ..(15:4);
\draw[Emerald] (0,0)
    .. controls +(0:2) and +(165:1) ..(-15:4);
\draw[Emerald] (0,0)
    .. controls +(0:3) and +(225:.5) ..(45:4);
\draw[Emerald] (0,0)
    .. controls +(0:3) and +(125:.5) ..(-45:4);

\draw[Emerald] (0,0)
    .. controls +(0:3) and +(0:1.5) ..(90:2.8)
    .. controls +(180:1.5) and +(90:2) ..(180:3.2)
    .. controls +(-90:2) and +(180:1.5) ..(-90:2.8)
    .. controls +(0:1.5) and +(0:3) .. (0,0);
\draw[Emerald] (0,0)
    .. controls +(0:2) and +(0:1) ..(90:1.8)
    .. controls +(180:1) and +(90:1) ..(180:2.2)
    .. controls +(-90:1) and +(180:1) ..(-90:1.8)
    .. controls +(0:1) and +(0:2) .. (0,0);
\draw[Emerald] (0,0)
    .. controls +(0:1) and +(0:1) ..(90:1)
    .. controls +(180:1) and +(90:.3) ..(180:1.2)
    .. controls +(-90:.3) and +(180:1) ..(-90:1)
    .. controls +(0:1) and +(0:1) .. (0,0);
\draw[NavyBlue](0,0)node{$\bullet$};
\end{tikzpicture}\quad
\begin{tikzpicture}[scale=.32,rotate=0]
\draw[dashed, thin](0,0)circle(4);
\draw[very thick](180:4)to(0:4)(-90:4.5);
\draw[Emerald] (0,0)
    .. controls +(0:2) and +(195:1) ..(15:4);
\draw[Emerald] (0,0)
    .. controls +(0:2) and +(165:1) ..(-15:4);
\draw[Emerald] (0,0)
    .. controls +(180:2) and +(15:1) ..(195:4);
\draw[Emerald] (0,0)
    .. controls +(180:2) and +(-15:1) ..(165:4);

\draw[Emerald] (0,0)
    .. controls +(0:3) and +(0:3) ..(90:3.5)
    .. controls +(180:3) and +(180:3) .. (0,0);
\draw[Emerald] (0,0)
    .. controls +(0:2) and +(0:2) ..(90:2.5)
    .. controls +(180:2) and +(180:2) .. (0,0);
\draw[Emerald] (0,0)
    .. controls +(0:1) and +(0:1) ..(90:1.5)
    .. controls +(180:1) and +(180:1) .. (0,0);
\draw[Emerald] (0,0)
    .. controls +(0:3) and +(0:3) ..(-90:3.5)
    .. controls +(180:3) and +(180:3) .. (0,0);
\draw[Emerald] (0,0)
    .. controls +(0:2) and +(0:2) ..(-90:2.5)
    .. controls +(180:2) and +(180:2) .. (0,0);
\draw[Emerald] (0,0)
    .. controls +(0:1) and +(0:1) ..(-90:1.5)
    .. controls +(180:1) and +(180:1) .. (0,0);

\draw[NavyBlue](0,0)node{$\bullet$};
\end{tikzpicture}\quad
\begin{tikzpicture}[scale=.32,rotate=0]
\draw[dashed, thin](0,0)circle(4)(-90:4.5);
\foreach \j in {0,1,2}
{
  \draw[very thick](120*\j+0:4)to(0,0);
\draw[Emerald] (0,0)
    .. controls +(120*\j+0:2) and +(120*\j+195:1) ..(120*\j+15:4);
\draw[Emerald] (0,0)
    .. controls +(120*\j+0:2) and +(120*\j+165:1) ..(120*\j+-15:4);

\draw[Emerald] (0,0)
    .. controls +(0+120*\j:1) and +(-30+120*\j:3) ..(60+120*\j:3)
    .. controls +(150+120*\j:3) and +(120+120*\j:1) .. (0,0);
\draw[Emerald] (0,0)
    .. controls +(0+120*\j:.3) and +(-30+120*\j:2) ..(60+120*\j:2)
    .. controls +(150+120*\j:2) and +(120+120*\j:.3) .. (0,0);
}
\draw[NavyBlue](0,0)node{$\bullet$};
\end{tikzpicture}
\caption{Local trajectories at poles of order $3,4,5$}\label{fig:poles}
\end{figure}
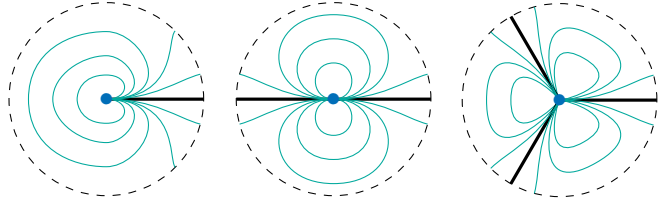

The real (oriented) blow-up of $(\xx,\phi)$ is a differentiable surface $\xx^\phi$,
which is obtained from the underlying differentiable surface by replacing a pole $P\in\Pol(\phi)$
with order at least 3 by a boundary $\partial_P$,
where the points on the boundary correspond to the real tangent directions at $P$.
Furthermore, we will mark the points on $\partial_P$ that correspond to the distinguished tangent directions,
so there are $\operatorname{ord}_\phi(P)-2$ marked points on $\partial_P$.
Thus $\xx^\phi$ is a marked surface, diffeomorphic to some $(\surf, M)$ as encountered already.

Denote by $\Diff(\surf)$ the group of diffeomorphism of $\surf$
that preserves the set $\M$ of marked points setwise
and $\Diff_0(\surf)$ the connected component of $\Diff(\surf)$ containing the identity.
The mapping class group is defined as $\MCG(\surf)=\Diff(\surf)/\Diff_0(\surf)$.

\begin{definition}\label{def:SFquad}
An \emph{$\surf$-framed quadratic differential} $(\xx,\phi,\psi)$
is a Riemann surface $\xx$ with GMN differential $\phi$,
equipped with a diffeomorphism $\psi\colon\surf \to\xx^\phi$,
preserving the marked points.

Two $\surf$-framed quadratic differentials $(\xx_1,\phi_1,\psi_1)$ and $(\xx_2,\phi_2,\psi_2)$
are equivalent, if there exists a biholomorphism $f\colon\xx_1\to\xx_2$
such that $f^*(\phi_2)=\phi_1$ and furthermore $\psi_2^{-1}\circ f_*\circ\psi_1\in\Diff_0(\surf)$,
where $f_*\colon\xx_1^{\phi_1}\to\xx_2^{\phi_2}$ is the induced diffeomorphism.
\end{definition}

We denote by $\FQuad{}{\surf}$ the moduli space of $\surf$-framed quadratic differentials.

\begin{remark}
As $\Diff(\surf)$ acts on $\FQuad{}{\surf}$ by pre-composition with $\psi$
and $\Diff_0(\surf)$ acts trivially,
we deduce that $\MCG(\surf)$ acts on $\FQuad{}{\surf}$ in a natural way.
\AKedit{Without the framing $\psi$, the group $\MCG(\surf)$ does not act on the data $(\xx, \phi)$.}

The analogous description of the unframed moduli space $\Quad(\surf)$ of quadratic differentials on $\surf$,
which appears in \cite{BS}, places no restriction on $\psi_2^{-1}\circ f_*\circ\psi_1$
and hence, in effect, the framing $\psi$ carries no information.
As a consequence, we have
\begin{equation}\label{eq:SFquad}
\Quad(\surf)=\FQuad{}{\surf}/\MCG(\surf).
\end{equation}
\end{remark}

\begin{definition}\label{def:SoFQuad}
The decorated real blow-up $\xx^\phi_\Tri$ of $(\xx,\phi)$ is
the decorated marked surface obtained from $\xx^\phi$ by adding the set $\Zer(\phi)$ as decorations.

Given any decorated marked surface $\surfo$,
an \emph{$\surfo$-framed quadratic differential} $(\xx,\phi,\psi)$
is a Riemann surface $\xx$ with GMN differential $\phi$,
equipped with a diffeomorphism $\psi\colon\surfo \to\xx^\phi_\Tri$,
preserving the marked points and decorations.

Equivalence is defined in the same way as in Definition~\ref{def:SFquad},
but with the restriction $\psi_2^{-1}\circ f_*\circ\psi_1\in\Diff_0(\surfo)$,
the identity component of the group $\Diff(\surfo)$ of diffeomorphisms preserving marked points and decorations (each setwise).
\end{definition}

We denote by $\FQuad{}{\surfo}$ the moduli space of $\surfo$-framed quadratic differentials.
As before, $\MCG(\surfo)=\Diff(\surfo)/\Diff_0(\surfo)$ acts on
$\FQuad{}{\surfo}$ and we have
\begin{equation}\label{eq:SoFquad}
\Quad(\surf)=\FQuad{}{\surfo}/\MCG(\surfo).
\end{equation}

\begin{lemma}\label{lem:SBr}
The kernel of the forgetful map $\MCG(\surfo)\to\MCG(\surf)$ is
the surface braid group $\SBr(\surfo)$,
that is, the fundamental group of the configuration space of $|\Tri|$ points in (the interior of) $\surf$,
based at the set $\Tri$.
\end{lemma}

\begin{proof}
Since $\MCG(\surfo)$ and $\MCG(\surf)$ act in the same way on the marked points $\M$,
we can restrict to the subgroups which fix $\M$ pointwise.
However, any such (orientation preserving) diffeomorphism is isotopic to one that fixes the boundary
$\partial\surf$ pointwise.
Hence the required kernel is the same as the kernel of the forgetful map
\[ \MCG(\surfo,\partial\surf)\to\MCG(\surf,\partial\surf)
\]
between the mapping class groups that fix the boundary.
This kernel is well-known to be the surface braid group $\SBr(\surfo)$,
e.g. by \cite[\S~2.4 (5)]{GJP}.
\end{proof}

By Lemma~\ref{lem:SBr},  we have the following commutative diagram of quotients.

\begin{equation}\label{eq:quad0}
\begin{tikzpicture}[xscale=1.6,yscale=0.8,baseline=(bb.base)]
\path (0,1) node (bb) {}; 
\draw (0,2) node (s0) {$\FQuad{}{\surfo}$}
 (0,0) node (s1) {$\FQuad{}{\surf}$}
 (2.5,1) node (s2) {$\Quad(\surf)$};
\draw [->, font=\scriptsize]
 (s0) edge node [left] {$\SBr(\surfo)$} (s1)
 (s0) edge node [above] {$\MCG(\surfo)$} (s2)
 (s1) edge node [below] {$\MCG(\surf)$} (s2);
\end{tikzpicture}
\end{equation}

We will see that $\FQuad{}{\surfo}$ and $\FQuad{}{\surf}$ are both manifolds
and the vertical quotient is a Galois covering.
However, $\FQuad{}{\surfo}$ is not connected,
so we must see how to understand its connected components.

We do this by embedding the exchange graphs $\EG(\surf)$ and $\EG(\surfo)$
in $\FQuad{}{\surf}$ and $\FQuad{}{\surfo}$, respectively.

\subsection{WKB triangulations}

There are the following types of trajectories of a GMN differential $\phi$ in our case:
\begin{itemize}
\item \emph{saddle trajectories} whose both ends are in $\Zer(\phi)$;
\item \emph{separating trajectories} with one end in $\Zer(\phi)$ and the other
in $\Pol(\phi)$;
\item \emph{generic trajectories} whose both ends are in $\Pol(\phi)$;
\end{itemize}
By removing all separating trajectories (which are finitely many) from $\surp$,
the remaining open surface splits as a disjoint union of connected components.
Each component is one of the following types \Note{(cf. \cite[\S3]{BS})}:
\begin{itemize}
\item a \emph{half-plane}, i.e. is isomorphic to
$\{z\in \CC \mid \Imgy(z)>0\}$
equipped with the differential $\dd z^{2}$.
It is swept out by generic trajectories which connect a fixed pole to itself.
\item
a \emph{horizontal strip}, i.e. is isomorphic to
$\{z\in \CC\mid a<\Imgy(z)<b\}$
equipped with the differential $\dd z^{2}$ for some $a<b \in \RR$.
It is swept out by generic trajectories connecting two (not necessarily distinct) poles.
\end{itemize}
We call this union the \emph{horizontal strip decomposition} of $\xx$ with respect to $\phi$.
A GMN differential $\phi$ on $\xx$ is \emph{saddle-free}, if it has no saddle trajectory.
Similarly, a framed quadratic differential on $\surf$ (or $\surfo$) is saddle-free
if the corresponding GMN differential is saddle-free.
Note the following:
\begin{itemize}
\item In each horizontal strip, the trajectories are \AKedit{isotopic} to each other.
\item the boundary of any \AKedit{component} consists of separating trajectories.
\item In each horizontal strip, there is \AKedit{a} unique geodesic,
the \emph{saddle connection}, connecting the two zeroes on its boundary.
\item For a saddle-free GMN differential $\phi$ on $\xx$,
we have $\Pol(\phi)\neq\emptyset$. Then
$\phi$ has no closed or recurrent trajectories by \cite[Lemma~3.1]{BS}.
Thus, in the horizontal strip decomposition of $\xx$ with respect to $\phi$,
there is only half-planes and horizontal strips.
\end{itemize}

By construction, the generic trajectories on $\xx$ (with respect to $\phi$)
are inherited by $\surf$, for any $\psi\colon\surf\to\xx^\phi$, and
all trajectories on $\xx$ (with respect to $\phi$)
are inherited by $\surfo$, for any $\Psi\colon\surfo\to\xx^\phi$.
\note{
For instance, the generic trajectories become open arcs on $\surf$ (as well as on $\surfo$)
and saddle trajectories becomes closed arcs on $\surfo$.
}

\begin{definition}\cite{BS}
Let $\psi\colon\surf\to\xx^\phi$ be an $\surf$-framed quadratic differential,
which is saddle-free.
Then there is a \emph{WKB triangulation} $\RT_\psi$ on $\surf$ induced from $\psi$,
where the arcs are (isotopy classes of inherited) generic trajectories.
\note{Moreover, each triangle in $\RT_\psi$ contains exactly one zero,
so $\RT_\psi$ becomes a decorated triangulation $\T_\psi$ of $\surfo$,
with dual graph consisting of saddle trajectories.}
\end{definition}

\begin{example}
Figure~\ref{fig:logo} shows several local horizontal strip decompositions on $\xx$ or $\surf$ corresponding to the arcs in a square and its flip, including two `wall crossings',
when there is a saddle trajectory. 
Note that, in the pictures
\begin{itemize}
\item the blue vertices are poles or marked points,
\item the red vertices are simple zeroes,
\item the green arcs are geodesics,
\item the black arcs are separating trajectories,
\item the red dotted arcs are the saddle connections in the horizontal strips;
the red solid arcs are saddle trajectories.
\end{itemize}
Moreover, Figure~\ref{fig:logo} actually demonstrates
a loop in the moduli space $\FQuad{}{\surf}$
which corresponds to the braid twist of the saddle connection
in the center of each picture in the figure.
\end{example}

\begin{figure}[ht]\centering
\begin{tikzpicture}[xscale=-.25,yscale=.25]\draw[Emerald!60,thin](-45:5)--(-225:5);
    \path (45:5) coordinate (v2)
          (135:5) coordinate (v1)
          (22.5:1.666) coordinate (v3)
          (0,5) coordinate (v4);
  \foreach \j in {.1, .18, .26, .34, .42, .5,.58, .66, .74, .82, .9}
    {
      \path (v3)--(v4) coordinate[pos=\j] (m0);
      \draw[Emerald!60, thin] plot [smooth,tension=.3] coordinates {(v1)(m0)(v2)};
    }
\draw[](v4)to(v1)to(v3)to(v2)to(v4);
    \path (-45:5) coordinate (v1)
          (-135:5) coordinate (v2)
          (202.5:1.666) coordinate (v3)
          (0,-5) coordinate (v4);
  \foreach \j in {.1, .18, .26, .34, .42, .5,.58, .66, .74, .82, .9}
    {
      \path (v3)--(v4) coordinate[pos=\j] (m0);
      \draw[Emerald!60, thin] plot [smooth,tension=.3] coordinates {(v1)(m0)(v2)};
    }
\draw[](v4)to(v1)to(v3)to(v2)to(v4);
    \path (-45:5) coordinate (v2)
          (135:5) coordinate (v1)
          (22.5:1.666) coordinate (v3)
          (202.5:1.666) coordinate (v4);
  \foreach \j in {.13,.26,.39,.87,.74,.61}
    {
      \path (v3)--(v4) coordinate[pos=\j] (m0);
      \draw[Emerald!60, thin] plot [smooth,tension=.3] coordinates {(v1)(m0)(v2)};
    }
\draw[](v4)to(v1)to(v3)to(v2)to(v4);
    \path (45:5) coordinate (v1)
          (-45:5) coordinate (v2)
          (22.5:1.666) coordinate (v3)
          (5,0) coordinate (v4);
  \foreach \j in {.1,.2,.3,.4,.5,.6,.7,.8,.9}
    {
      \path (v3)--(v4) coordinate[pos=\j] (m0);
      \draw[Emerald!60, thin] plot [smooth,tension=.3] coordinates {(v1)(m0)(v2)};
    }
\draw[](v4)to(v1)to(v3)to(v2)to(v4);
    \path (-135:5) coordinate (v2)
          (135:5) coordinate (v1)
          (202.5:1.666) coordinate (v3)
          (-5,0) coordinate (v4);
  \foreach \j in {.1,.2,.3,.4,.5,.6,.7,.8,.9}
    {
      \path (v3)--(v4) coordinate[pos=\j] (m0);
      \draw[Emerald!60, thin] plot [smooth,tension=.3] coordinates {(v1)(m0)(v2)};
    }
\draw[](v4)to(v1)to(v3)to(v2)to(v4);
\draw[NavyBlue,thin, dashed]
    (-45:5)node{$\bullet$}(45:5)node{$\bullet$}(135:5)node{$\bullet$}(-135:5)node{$\bullet$}(-45:5)(135:5);
\draw[red,dashed]
  (-5,0)node{$\bullet$}node[white]{\tiny{$\bullet$}}node{\tiny{$\circ$}}to
  (202.5:1.666)node{$\bullet$}node[white]{\tiny{$\bullet$}}node{\tiny{$\circ$}}to
  (22.5:1.666)node{$\bullet$}node[white]{\tiny{$\bullet$}}node{\tiny{$\circ$}}to
  (5,0)node{$\bullet$}node[white]{\tiny{$\bullet$}}node{\tiny{$\circ$}}
  (0,5)node{$\bullet$}node[white]{\tiny{$\bullet$}}node{\tiny{$\circ$}}to(22.5:1.666)
  (0,-5)node{$\bullet$}node[white]{\tiny{$\bullet$}}node{\tiny{$\circ$}}to(202.5:1.666);
\end{tikzpicture}
\begin{tikzpicture}[scale=0.250, rotate=0]
\draw[->=stealth](-5,-7)to(5,-7);
\draw(0,-10)node{};
\draw[](0,-8)node{wall crossing};
    \path (45:5) coordinate (v2)
          (135:5) coordinate (v1)
          (0:1.666) coordinate (v3)(-180:1.666) coordinate (w3)
          (0,5) coordinate (v4);\draw[red](v3)to(w3);
  \foreach \j in {.1, .18, .26, .34, .42, .5,.58, .66, .74}
    {
      \path (v3)--(v4) coordinate[pos=\j] (m0);
      \path (w3)--(v4) coordinate[pos=\j] (m1);
      \draw[Emerald!60, thin] plot [thick, smooth,tension=.3] coordinates {(v1)(m1)(m0)(v2)};
    }
  \foreach \j in { .82, .9}
    {
      \path (v3)--(v4) coordinate[pos=\j] (m0);
      \path (w3)--(v4) coordinate[pos=\j] (m1);
      \path (m1)--(m0) coordinate[pos=.5] (m1);
      \draw[Emerald!60, thin] plot [thick, smooth,tension=.3] coordinates {(v1)(m1)(v2)};
    }
\draw[](v4)to(v1)to(w3)to(v3)to(v2)to(v4);
    \path (-45:5) coordinate (v2)
          (-135:5) coordinate (v1)
          (0:1.666) coordinate (v3)(-180:1.666) coordinate (w3)
          (0,-5) coordinate (v4);
  \foreach \j in {.1, .18, .26, .34, .42, .5,.58, .66, .74}
    {
      \path (v3)--(v4) coordinate[pos=\j] (m0);
      \path (w3)--(v4) coordinate[pos=\j] (m1);
      \draw[Emerald!60, thin] plot [thick, smooth,tension=.3] coordinates {(v1)(m1)(m0)(v2)};
    }
  \foreach \j in { .82, .9}
    {
      \path (v3)--(v4) coordinate[pos=\j] (m0);
      \path (w3)--(v4) coordinate[pos=\j] (m1);
      \path (m1)--(m0) coordinate[pos=.5] (m1);
      \draw[Emerald!60, thin] plot [thick, smooth,tension=.3] coordinates {(v1)(m1)(v2)};
    }
\draw[](v4)to(v1)to(w3)to(v3)to(v2)to(v4);

    \path (45:5) coordinate (v1)
          (-45:5) coordinate (v2)
          (0:1.666) coordinate (v3)
          (5,0) coordinate (v4);
  \foreach \j in {.1,.2,.3,.4,.5,.6,.7,.8,.9}
    {
      \path (v3)--(v4) coordinate[pos=\j] (m0);
      \draw[Emerald!60, thin] plot [smooth,tension=.3] coordinates {(v1)(m0)(v2)};
    }
\draw[](v4)to(v1)to(v3)to(v2)to(v4);
    \path (-135:5) coordinate (v2)
          (135:5) coordinate (v1)
          (-180:1.666) coordinate (v3)
          (-5,0) coordinate (v4);
  \foreach \j in {.1,.2,.3,.4,.5,.6,.7,.8,.9}
    {
      \path (v3)--(v4) coordinate[pos=\j] (m0);
      \draw[Emerald!60, thin] plot [smooth,tension=.3] coordinates {(v1)(m0)(v2)};
    }
\draw[](v4)to(v1)to(v3)to(v2)to(v4);
\draw[NavyBlue,thin, dashed]
    (-45:5)node{$\bullet$}(45:5)node{$\bullet$}(135:5)node{$\bullet$}(-135:5)node{$\bullet$}(-45:5)(135:5);
\draw[white,ultra thick](-180:1.666)to(0:1.666);
\draw[red,thick](-2,0)to(2,0);
\draw[red,dashed](-180:1.666)to(0,5)to(0:1.666);
\draw[red,dashed](-180:1.666)to(0,-5)to(0:1.666);
\draw[red,dashed]
  (-5,0)node{$\bullet$}node[white]{\tiny{$\bullet$}}node{\tiny{$\circ$}}to
  (-180:1.666)node{$\bullet$}node[white]{\tiny{$\bullet$}}node{\tiny{$\circ$}}
  (0:1.666)node{$\bullet$}node[white]{\tiny{$\bullet$}}node{\tiny{$\circ$}}to
  (5,0)node{$\bullet$}node[white]{\tiny{$\bullet$}}node{\tiny{$\circ$}}
  (0,5)node{$\bullet$}node[white]{\tiny{$\bullet$}}node{\tiny{$\circ$}}
  (0,-5)node{$\bullet$}node[white]{\tiny{$\bullet$}}node{\tiny{$\circ$}};
\draw[red,dashed](-180:1.666)node{$\bullet$}node[white]{\tiny{$\bullet$}}node{\tiny{$\circ$}}
  (0:1.666)node{$\bullet$}node[white]{\tiny{$\bullet$}}node{\tiny{$\circ$}};
\end{tikzpicture}
\begin{tikzpicture}[xscale=.25,yscale=.25]\draw[Emerald!60,thin](-45:5)--(-225:5);
    \path (45:5) coordinate (v2)
          (135:5) coordinate (v1)
          (22.5:1.666) coordinate (v3)
          (0,5) coordinate (v4);
  \foreach \j in {.1, .18, .26, .34, .42, .5,.58, .66, .74, .82, .9}
    {
      \path (v3)--(v4) coordinate[pos=\j] (m0);
      \draw[Emerald!60, thin] plot [smooth,tension=.3] coordinates {(v1)(m0)(v2)};
    }
\draw[](v4)to(v1)to(v3)to(v2)to(v4);
    \path (-45:5) coordinate (v1)
          (-135:5) coordinate (v2)
          (202.5:1.666) coordinate (v3)
          (0,-5) coordinate (v4);
  \foreach \j in {.1, .18, .26, .34, .42, .5,.58, .66, .74, .82, .9}
    {
      \path (v3)--(v4) coordinate[pos=\j] (m0);
      \draw[Emerald!60, thin] plot [smooth,tension=.3] coordinates {(v1)(m0)(v2)};
    }
\draw[](v4)to(v1)to(v3)to(v2)to(v4);
    \path (-45:5) coordinate (v2)
          (135:5) coordinate (v1)
          (22.5:1.666) coordinate (v3)
          (202.5:1.666) coordinate (v4);
  \foreach \j in {.13,.26,.39,.87,.74,.61}
    {
      \path (v3)--(v4) coordinate[pos=\j] (m0);
      \draw[Emerald!60, thin] plot [smooth,tension=.3] coordinates {(v1)(m0)(v2)};
    }
\draw[](v4)to(v1)to(v3)to(v2)to(v4);
    \path (45:5) coordinate (v1)
          (-45:5) coordinate (v2)
          (22.5:1.666) coordinate (v3)
          (5,0) coordinate (v4);
  \foreach \j in {.1,.2,.3,.4,.5,.6,.7,.8,.9}
    {
      \path (v3)--(v4) coordinate[pos=\j] (m0);
      \draw[Emerald!60, thin] plot [smooth,tension=.3] coordinates {(v1)(m0)(v2)};
    }
\draw[](v4)to(v1)to(v3)to(v2)to(v4);
    \path (-135:5) coordinate (v2)
          (135:5) coordinate (v1)
          (202.5:1.666) coordinate (v3)
          (-5,0) coordinate (v4);
  \foreach \j in {.1,.2,.3,.4,.5,.6,.7,.8,.9}
    {
      \path (v3)--(v4) coordinate[pos=\j] (m0);
      \draw[Emerald!60, thin] plot [smooth,tension=.3] coordinates {(v1)(m0)(v2)};
    }
\draw[](v4)to(v1)to(v3)to(v2)to(v4);
\draw[NavyBlue,thin, dashed]
    (-45:5)node{$\bullet$}(45:5)node{$\bullet$}(135:5)node{$\bullet$}(-135:5)node{$\bullet$}(-45:5)(135:5);
\draw[red,dashed]
  (-5,0)node{$\bullet$}node[white]{\tiny{$\bullet$}}node{\tiny{$\circ$}}to
  (202.5:1.666)node{$\bullet$}node[white]{\tiny{$\bullet$}}node{\tiny{$\circ$}}to
  (22.5:1.666)node{$\bullet$}node[white]{\tiny{$\bullet$}}node{\tiny{$\circ$}}to
  (5,0)node{$\bullet$}node[white]{\tiny{$\bullet$}}node{\tiny{$\circ$}}
  (0,5)node{$\bullet$}node[white]{\tiny{$\bullet$}}node{\tiny{$\circ$}}to(22.5:1.666)
  (0,-5)node{$\bullet$}node[white]{\tiny{$\bullet$}}node{\tiny{$\circ$}}to(202.5:1.666);
\end{tikzpicture}

\begin{tikzpicture}[xscale=-.25,yscale=.25]\draw[Emerald!60,thin](-45:5)--(-225:5);
    \path (45:5) coordinate (v2)
          (135:5) coordinate (v1)
          (45:5/3) coordinate (v3)
          (0,5) coordinate (v4);
  \foreach \j in {.1, .18, .26, .34, .42, .5,.58, .66, .74, .82, .9}
    {
      \path (v3)--(v4) coordinate[pos=\j] (m0);
      \draw[Emerald!60, thin] plot [smooth,tension=.3] coordinates {(v1)(m0)(v2)};
    }
\draw[](v4)to(v1)to(v3)to(v2)to(v4);
    \path (-45:5) coordinate (v1)
          (-135:5) coordinate (v2)
          (-135:5/3) coordinate (v3)
          (0,-5) coordinate (v4);
  \foreach \j in {.1, .18, .26, .34, .42, .5,.58, .66, .74, .82, .9}
    {
      \path (v3)--(v4) coordinate[pos=\j] (m0);
      \draw[Emerald!60, thin] plot [smooth,tension=.3] coordinates {(v1)(m0)(v2)};
    }
\draw[](v4)to(v1)to(v3)to(v2)to(v4);
    \path (-45:5) coordinate (v2)
          (135:5) coordinate (v1)
          (45:5/3) coordinate (v3)
          (-135:5/3) coordinate (v4);
  \foreach \j in {.13,.26,.39,.87,.74,.61}
    {
      \path (v3)--(v4) coordinate[pos=\j] (m0);
      \draw[Emerald!60, thin] plot [smooth,tension=.3] coordinates {(v1)(m0)(v2)};
    }
\draw[](v4)to(v1)to(v3)to(v2)to(v4);
    \path (45:5) coordinate (v1)
          (-45:5) coordinate (v2)
          (45:5/3) coordinate (v3)
          (5,0) coordinate (v4);
  \foreach \j in {.1,.2,.3,.4,.5,.6,.7,.8,.9}
    {
      \path (v3)--(v4) coordinate[pos=\j] (m0);
      \draw[Emerald!60, thin] plot [smooth,tension=.3] coordinates {(v1)(m0)(v2)};
    }
\draw[](v4)to(v1)to(v3)to(v2)to(v4);
    \path (-135:5) coordinate (v2)
          (135:5) coordinate (v1)
          (-135:5/3) coordinate (v3)
          (-5,0) coordinate (v4);
  \foreach \j in {.1,.2,.3,.4,.5,.6,.7,.8,.9}
    {
      \path (v3)--(v4) coordinate[pos=\j] (m0);
      \draw[Emerald!60, thin] plot [smooth,tension=.3] coordinates {(v1)(m0)(v2)};
    }
\draw[](v4)to(v1)to(v3)to(v2)to(v4);
\draw[NavyBlue,thin, dashed]
    (-45:5)node{$\bullet$}(45:5)node{$\bullet$}(135:5)node{$\bullet$}(-135:5)node{$\bullet$}(-45:5)(135:5);
\draw[red,dashed]
  (-5,0)node{$\bullet$}node[white]{\tiny{$\bullet$}}node{\tiny{$\circ$}}to
  (-135:5/3)node{$\bullet$}node[white]{\tiny{$\bullet$}}node{\tiny{$\circ$}}to
  (45:5/3)node{$\bullet$}node[white]{\tiny{$\bullet$}}node{\tiny{$\circ$}}to
  (5,0)node{$\bullet$}node[white]{\tiny{$\bullet$}}node{\tiny{$\circ$}}
  (0,5)node{$\bullet$}node[white]{\tiny{$\bullet$}}node{\tiny{$\circ$}}to(45:5/3)
  (0,-5)node{$\bullet$}node[white]{\tiny{$\bullet$}}node{\tiny{$\circ$}}to(-135:5/3);
\end{tikzpicture}
\begin{tikzpicture}[scale=.67]
\draw[->=stealth,white](-4,-1)to(-4,1);\draw(0,-2)node{};
\draw[<-=stealth,white](4,-1)to(4,1);
\draw[Emerald!60,thick](2.2,-.6)rectangle(1,.6)to(2.2,-.6);
\draw[Emerald!60,thick](-2.2,-.6)rectangle(-1,.6)to(-2.2,-.6);
\draw[->=stealth](-.7,.2)to(.7,.2);
\draw[->=stealth](.7,-.2)to(-.7,-.2);
\draw[NavyBlue](1,.6)node{$\bullet$}(2.2,-.6)node{$\bullet$}
               (2.2,.6)node{$\bullet$}(1,-.6)node{$\bullet$}
               (-1,.6)node{$\bullet$}(-2.2,-.6)node{$\bullet$}
               (-2.2,.6)node{$\bullet$}(-1,-.6)node{$\bullet$};
\end{tikzpicture}
\begin{tikzpicture}[xscale=.25,yscale=.25]\draw[Emerald!60,thin](-45:5)--(-225:5);
    \path (45:5) coordinate (v2)
          (135:5) coordinate (v1)
          (45:5/3) coordinate (v3)
          (0,5) coordinate (v4);
  \foreach \j in {.1, .18, .26, .34, .42, .5,.58, .66, .74, .82, .9}
    {
      \path (v3)--(v4) coordinate[pos=\j] (m0);
      \draw[Emerald!60, thin] plot [smooth,tension=.3] coordinates {(v1)(m0)(v2)};
    }
\draw[](v4)to(v1)to(v3)to(v2)to(v4);
    \path (-45:5) coordinate (v1)
          (-135:5) coordinate (v2)
          (-135:5/3) coordinate (v3)
          (0,-5) coordinate (v4);
  \foreach \j in {.1, .18, .26, .34, .42, .5,.58, .66, .74, .82, .9}
    {
      \path (v3)--(v4) coordinate[pos=\j] (m0);
      \draw[Emerald!60, thin] plot [smooth,tension=.3] coordinates {(v1)(m0)(v2)};
    }
\draw[](v4)to(v1)to(v3)to(v2)to(v4);
    \path (-45:5) coordinate (v2)
          (135:5) coordinate (v1)
          (45:5/3) coordinate (v3)
          (-135:5/3) coordinate (v4);
  \foreach \j in {.13,.26,.39,.87,.74,.61}
    {
      \path (v3)--(v4) coordinate[pos=\j] (m0);
      \draw[Emerald!60, thin] plot [smooth,tension=.3] coordinates {(v1)(m0)(v2)};
    }
\draw[](v4)to(v1)to(v3)to(v2)to(v4);
    \path (45:5) coordinate (v1)
          (-45:5) coordinate (v2)
          (45:5/3) coordinate (v3)
          (5,0) coordinate (v4);
  \foreach \j in {.1,.2,.3,.4,.5,.6,.7,.8,.9}
    {
      \path (v3)--(v4) coordinate[pos=\j] (m0);
      \draw[Emerald!60, thin] plot [smooth,tension=.3] coordinates {(v1)(m0)(v2)};
    }
\draw[](v4)to(v1)to(v3)to(v2)to(v4);
    \path (-135:5) coordinate (v2)
          (135:5) coordinate (v1)
          (-135:5/3) coordinate (v3)
          (-5,0) coordinate (v4);
  \foreach \j in {.1,.2,.3,.4,.5,.6,.7,.8,.9}
    {
      \path (v3)--(v4) coordinate[pos=\j] (m0);
      \draw[Emerald!60, thin] plot [smooth,tension=.3] coordinates {(v1)(m0)(v2)};
    }
\draw[](v4)to(v1)to(v3)to(v2)to(v4);
\draw[NavyBlue,thin, dashed]
    (-45:5)node{$\bullet$}(45:5)node{$\bullet$}(135:5)node{$\bullet$}(-135:5)node{$\bullet$}(-45:5)(135:5);
\draw[red,dashed]
  (-5,0)node{$\bullet$}node[white]{\tiny{$\bullet$}}node{\tiny{$\circ$}}to
  (-135:5/3)node{$\bullet$}node[white]{\tiny{$\bullet$}}node{\tiny{$\circ$}}to
  (45:5/3)node{$\bullet$}node[white]{\tiny{$\bullet$}}node{\tiny{$\circ$}}to
  (5,0)node{$\bullet$}node[white]{\tiny{$\bullet$}}node{\tiny{$\circ$}}
  (0,5)node{$\bullet$}node[white]{\tiny{$\bullet$}}node{\tiny{$\circ$}}to(45:5/3)
  (0,-5)node{$\bullet$}node[white]{\tiny{$\bullet$}}node{\tiny{$\circ$}}to(-135:5/3);
\end{tikzpicture}

\begin{tikzpicture}[xscale=.25,yscale=.25,rotate=90]\draw[Emerald!60,thin](-45:5)--(-225:5);
    \path (45:5) coordinate (v2)
          (135:5) coordinate (v1)
          (22.5:1.666) coordinate (v3)
          (0,5) coordinate (v4);
  \foreach \j in {.1, .18, .26, .34, .42, .5,.58, .66, .74, .82, .9}
    {
      \path (v3)--(v4) coordinate[pos=\j] (m0);
      \draw[Emerald!60, thin] plot [smooth,tension=.3] coordinates {(v1)(m0)(v2)};
    }
\draw[](v4)to(v1)to(v3)to(v2)to(v4);
    \path (-45:5) coordinate (v1)
          (-135:5) coordinate (v2)
          (202.5:1.666) coordinate (v3)
          (0,-5) coordinate (v4);
  \foreach \j in {.1, .18, .26, .34, .42, .5,.58, .66, .74, .82, .9}
    {
      \path (v3)--(v4) coordinate[pos=\j] (m0);
      \draw[Emerald!60, thin] plot [smooth,tension=.3] coordinates {(v1)(m0)(v2)};
    }
\draw[](v4)to(v1)to(v3)to(v2)to(v4);
    \path (-45:5) coordinate (v2)
          (135:5) coordinate (v1)
          (22.5:1.666) coordinate (v3)
          (202.5:1.666) coordinate (v4);
  \foreach \j in {.13,.26,.39,.87,.74,.61}
    {
      \path (v3)--(v4) coordinate[pos=\j] (m0);
      \draw[Emerald!60, thin] plot [smooth,tension=.3] coordinates {(v1)(m0)(v2)};
    }
\draw[](v4)to(v1)to(v3)to(v2)to(v4);
    \path (45:5) coordinate (v1)
          (-45:5) coordinate (v2)
          (22.5:1.666) coordinate (v3)
          (5,0) coordinate (v4);
  \foreach \j in {.1,.2,.3,.4,.5,.6,.7,.8,.9}
    {
      \path (v3)--(v4) coordinate[pos=\j] (m0);
      \draw[Emerald!60, thin] plot [smooth,tension=.3] coordinates {(v1)(m0)(v2)};
    }
\draw[](v4)to(v1)to(v3)to(v2)to(v4);
    \path (-135:5) coordinate (v2)
          (135:5) coordinate (v1)
          (202.5:1.666) coordinate (v3)
          (-5,0) coordinate (v4);
  \foreach \j in {.1,.2,.3,.4,.5,.6,.7,.8,.9}
    {
      \path (v3)--(v4) coordinate[pos=\j] (m0);
      \draw[Emerald!60, thin] plot [smooth,tension=.3] coordinates {(v1)(m0)(v2)};
    }
\draw[](v4)to(v1)to(v3)to(v2)to(v4);
\draw[NavyBlue,thin, dashed]
    (-45:5)node{$\bullet$}(45:5)node{$\bullet$}(135:5)node{$\bullet$}(-135:5)node{$\bullet$}(-45:5)(135:5);
\draw[red,dashed]
  (-5,0)node{$\bullet$}node[white]{\tiny{$\bullet$}}node{\tiny{$\circ$}}to
  (202.5:1.666)node{$\bullet$}node[white]{\tiny{$\bullet$}}node{\tiny{$\circ$}}to
  (22.5:1.666)node{$\bullet$}node[white]{\tiny{$\bullet$}}node{\tiny{$\circ$}}to
  (5,0)node{$\bullet$}node[white]{\tiny{$\bullet$}}node{\tiny{$\circ$}}
  (0,5)node{$\bullet$}node[white]{\tiny{$\bullet$}}node{\tiny{$\circ$}}to(22.5:1.666)
  (0,-5)node{$\bullet$}node[white]{\tiny{$\bullet$}}node{\tiny{$\circ$}}to(202.5:1.666);
\draw(-10,0)node{};
\end{tikzpicture}
\begin{tikzpicture}[scale=0.250, rotate=90]
\draw[->=stealth](7,-6)to(7,6);
\draw[](8,0)node{wall crossing};
    \path (45:5) coordinate (v2)
          (135:5) coordinate (v1)
          (0:1.666) coordinate (v3)(-180:1.666) coordinate (w3)
          (0,5) coordinate (v4);\draw[red](v3)to(w3);
  \foreach \j in {.1, .18, .26, .34, .42, .5,.58, .66, .74}
    {
      \path (v3)--(v4) coordinate[pos=\j] (m0);
      \path (w3)--(v4) coordinate[pos=\j] (m1);
      \draw[Emerald!60, thin] plot [thick, smooth,tension=.3] coordinates {(v1)(m1)(m0)(v2)};
    }
  \foreach \j in { .82, .9}
    {
      \path (v3)--(v4) coordinate[pos=\j] (m0);
      \path (w3)--(v4) coordinate[pos=\j] (m1);
      \path (m1)--(m0) coordinate[pos=.5] (m1);
      \draw[Emerald!60, thin] plot [thick, smooth,tension=.3] coordinates {(v1)(m1)(v2)};
    }
\draw[](v4)to(v1)to(w3)to(v3)to(v2)to(v4);
    \path (-45:5) coordinate (v2)
          (-135:5) coordinate (v1)
          (0:1.666) coordinate (v3)(-180:1.666) coordinate (w3)
          (0,-5) coordinate (v4);
  \foreach \j in {.1, .18, .26, .34, .42, .5,.58, .66, .74}
    {
      \path (v3)--(v4) coordinate[pos=\j] (m0);
      \path (w3)--(v4) coordinate[pos=\j] (m1);
      \draw[Emerald!60, thin] plot [thick, smooth,tension=.3] coordinates {(v1)(m1)(m0)(v2)};
    }
  \foreach \j in { .82, .9}
    {
      \path (v3)--(v4) coordinate[pos=\j] (m0);
      \path (w3)--(v4) coordinate[pos=\j] (m1);
      \path (m1)--(m0) coordinate[pos=.5] (m1);
      \draw[Emerald!60, thin] plot [thick, smooth,tension=.3] coordinates {(v1)(m1)(v2)};
    }
\draw[](v4)to(v1)to(w3)to(v3)to(v2)to(v4);

    \path (45:5) coordinate (v1)
          (-45:5) coordinate (v2)
          (0:1.666) coordinate (v3)
          (5,0) coordinate (v4);
  \foreach \j in {.1,.2,.3,.4,.5,.6,.7,.8,.9}
    {
      \path (v3)--(v4) coordinate[pos=\j] (m0);
      \draw[Emerald!60, thin] plot [smooth,tension=.3] coordinates {(v1)(m0)(v2)};
    }
\draw[](v4)to(v1)to(v3)to(v2)to(v4);
    \path (-135:5) coordinate (v2)
          (135:5) coordinate (v1)
          (-180:1.666) coordinate (v3)
          (-5,0) coordinate (v4);
  \foreach \j in {.1,.2,.3,.4,.5,.6,.7,.8,.9}
    {
      \path (v3)--(v4) coordinate[pos=\j] (m0);
      \draw[Emerald!60, thin] plot [smooth,tension=.3] coordinates {(v1)(m0)(v2)};
    }
\draw[](v4)to(v1)to(v3)to(v2)to(v4);
\draw[NavyBlue,thin, dashed]
    (-45:5)node{$\bullet$}(45:5)node{$\bullet$}(135:5)node{$\bullet$}(-135:5)node{$\bullet$}(-45:5)(135:5);
\draw[red,dashed](-180:1.666)to(0,5)to(0:1.666);
\draw[red,dashed](-180:1.666)to(0,-5)to(0:1.666);
\draw[red,dashed]
  (-5,0)node{$\bullet$}node[white]{\tiny{$\bullet$}}node{\tiny{$\circ$}}to
  (-180:1.666)node{$\bullet$}node[white]{\tiny{$\bullet$}}node{\tiny{$\circ$}}
  (0:1.666)node{$\bullet$}node[white]{\tiny{$\bullet$}}node{\tiny{$\circ$}}to
  (5,0)node{$\bullet$}node[white]{\tiny{$\bullet$}}node{\tiny{$\circ$}}
  (0,5)node{$\bullet$}node[white]{\tiny{$\bullet$}}node{\tiny{$\circ$}}
  (0,-5)node{$\bullet$}node[white]{\tiny{$\bullet$}}node{\tiny{$\circ$}};
\draw[white,ultra thick](-180:1.666)to(0:1.666);
\draw[red,thick](-2,0)to(2,0);
\draw[red,dashed](-180:1.666)node{$\bullet$}node[white]{\tiny{$\bullet$}}node{\tiny{$\circ$}}
  (0:1.666)node{$\bullet$}node[white]{\tiny{$\bullet$}}node{\tiny{$\circ$}};
\end{tikzpicture}
\begin{tikzpicture}[xscale=.25,yscale=-.25,rotate=90]\draw[Emerald!60,thin](-45:5)--(-225:5)(10,0);
    \path (45:5) coordinate (v2)
          (135:5) coordinate (v1)
          (22.5:1.666) coordinate (v3)
          (0,5) coordinate (v4);
  \foreach \j in {.1, .18, .26, .34, .42, .5,.58, .66, .74, .82, .9}
    {
      \path (v3)--(v4) coordinate[pos=\j] (m0);
      \draw[Emerald!60, thin] plot [smooth,tension=.3] coordinates {(v1)(m0)(v2)};
    }
\draw[](v4)to(v1)to(v3)to(v2)to(v4);
    \path (-45:5) coordinate (v1)
          (-135:5) coordinate (v2)
          (202.5:1.666) coordinate (v3)
          (0,-5) coordinate (v4);
  \foreach \j in {.1, .18, .26, .34, .42, .5,.58, .66, .74, .82, .9}
    {
      \path (v3)--(v4) coordinate[pos=\j] (m0);
      \draw[Emerald!60, thin] plot [smooth,tension=.3] coordinates {(v1)(m0)(v2)};
    }
\draw[](v4)to(v1)to(v3)to(v2)to(v4);
    \path (-45:5) coordinate (v2)
          (135:5) coordinate (v1)
          (22.5:1.666) coordinate (v3)
          (202.5:1.666) coordinate (v4);
  \foreach \j in {.13,.26,.39,.87,.74,.61}
    {
      \path (v3)--(v4) coordinate[pos=\j] (m0);
      \draw[Emerald!60, thin] plot [smooth,tension=.3] coordinates {(v1)(m0)(v2)};
    }
\draw[](v4)to(v1)to(v3)to(v2)to(v4);
    \path (45:5) coordinate (v1)
          (-45:5) coordinate (v2)
          (22.5:1.666) coordinate (v3)
          (5,0) coordinate (v4);
  \foreach \j in {.1,.2,.3,.4,.5,.6,.7,.8,.9}
    {
      \path (v3)--(v4) coordinate[pos=\j] (m0);
      \draw[Emerald!60, thin] plot [smooth,tension=.3] coordinates {(v1)(m0)(v2)};
    }
\draw[](v4)to(v1)to(v3)to(v2)to(v4);
    \path (-135:5) coordinate (v2)
          (135:5) coordinate (v1)
          (202.5:1.666) coordinate (v3)
          (-5,0) coordinate (v4);
  \foreach \j in {.1,.2,.3,.4,.5,.6,.7,.8,.9}
    {
      \path (v3)--(v4) coordinate[pos=\j] (m0);
      \draw[Emerald!60, thin] plot [smooth,tension=.3] coordinates {(v1)(m0)(v2)};
    }
\draw[](v4)to(v1)to(v3)to(v2)to(v4);
\draw[NavyBlue,thin, dashed]
    (-45:5)node{$\bullet$}(45:5)node{$\bullet$}(135:5)node{$\bullet$}(-135:5)node{$\bullet$}(-45:5)(135:5);
\draw[red,dashed]
  (-5,0)node{$\bullet$}node[white]{\tiny{$\bullet$}}node{\tiny{$\circ$}}to
  (202.5:1.666)node{$\bullet$}node[white]{\tiny{$\bullet$}}node{\tiny{$\circ$}}to
  (22.5:1.666)node{$\bullet$}node[white]{\tiny{$\bullet$}}node{\tiny{$\circ$}}to
  (5,0)node{$\bullet$}node[white]{\tiny{$\bullet$}}node{\tiny{$\circ$}}
  (0,5)node{$\bullet$}node[white]{\tiny{$\bullet$}}node{\tiny{$\circ$}}to(22.5:1.666)
  (0,-5)node{$\bullet$}node[white]{\tiny{$\bullet$}}node{\tiny{$\circ$}}to(202.5:1.666);
\end{tikzpicture}
\caption{A loop in $\FQuad{}{\surf}$}
\label{fig:logo}
\end{figure}


We consider the top two parts of the stratification of $\FQuad{}{\surf}$
analogous to the stratification of $\Quad(\surf)$ from \cite[\S~5]{BS}:
\begin{align*}
 F_0(\surf) &=\{ [\xx,\phi,\psi] \in \FQuad{}{\surf} \mid \text{$\phi$ has no saddle trajectories} \},\\
 F_2(\surf) &=\{ [\xx,\phi,\psi] \in \FQuad{}{\surf} \mid \text{$\phi$ has exactly one saddle trajectory} \}.
\end{align*}
Then $B_0(\surf) :=F_0(\surf)$ is open and dense, and $F_2(\surf)$ has codimension 1.
Furthermore,  $B_2(\surf) := F_0(\surf)\cup F_2(\surf)$ is also open and dense, and
has complement of codimension 2.

Let $\cub(\RT)$ be the subspace in $\FQuad{}{\surf}$ consisting of
those saddle-free $[\psi]$ whose WKB triangulation is $\RT$.
Then
\begin{gather}\label{eq:Squad*}
    B_0(\surf)=\bigcup_{\RT\in\EG(\surf)} \cub(\RT).
\end{gather}
By the argument of \cite[Prop~4.9]{BS}, we can see that $\cub(\RT)\isom\UHP^{\RT}$,
where
\[
  \UHP=\{z\in\CC\mid \Imgy(z)>0\}\subset\CC
\]
is the (strict) upper half plane.
\danger{
The coordinates $(u_\gamma)_{\gamma\in T}$ give the complex modulus of the horizontal strip with generic trajectory in the isotopy class $\gamma$ (see \cite[\S4.5]{BS} for more detail).
}
Thus the $U(\RT)$ are precisely the connected components of $B_0(\surf)$.

The boundary of $\cub(\RT)$ meets $F_2(\surf)$ in $2\numarc$ connected components,
which we denote $\partial^\sharp_\gamma\cub(\RT)$ and $\partial^\flat_\gamma\cub(\RT)$,
where the coordinate $u_\gamma$ goes to the negative or positive real axis, respectively.
Note that $u_\gamma$ cannot go to zero because that would correspond to two zeroes of $\phi$ coming together.

Similarly, in $\FQuad{}{\surfo}$ there are cells $\cub(\T)$ and a stratification
$\{B_p(\surfo)\}$.
We can import a key lemma from \cite{BS}, where the original statement is actually for $\Quad(\surf)$.

\begin{lemma}\cite[Prop~5.8]{BS}
Any path in $\FQuad{}{\surf}$ is homotopic (relative to its end-points) to a path in $B_2(\surf)$.
Therefore there is a surjective map $\pi_1 B_2(\surf) \to \pi_1\FQuad{}{\surf}$.
The same holds replacing $\surf$ by $\surfo$.
\end{lemma}

Hence, we have the following result, showing that $\EG(\surf)$ is a \danger{skeleton} for $\FQuad{}{\surf}$.

\begin{lemma}\label{lem:iota_surf}
There is a canonical embedding
$\skel_\surf\colon\EG(\surf)\to\FQuad{}{\surf}$ whose image is dual to $B_2(\surf)$.
More precisely,
the embedding is unique up to homotopy, satisfying
\begin{itemize}
\item for each triangulation $\RT\in\EG(\surf)$, the point $\skel_\surf(\RT)$ is in $\cub(\RT)$,
\item for each flip $x\colon\RT\to\tilt{\RT}{\sharp}{\gamma}$, the path $\skel_\surf(x)$ is in
$\cub(\RT) \;\cup\; \partial^\sharp_\gamma\cub(\RT) \;\cup\; \cub(\tilt{\RT}{\sharp}{\gamma})$,
connecting $\skel_\surf(\RT)$ to $\skel_\surf(\tilt{\RT}{\sharp}{\gamma})$ and intersecting
$\partial^\sharp_\gamma\cub(\RT)$ at exactly one point,
\end{itemize}
Moreover, $\skel_{\surf}$ induces a surjective map
\[\pi_1\EG(\surf)\cong\pi_1B_2(\surf)\to\pi_1\FQuad{}{\surf}.\]
\end{lemma}

Similarly, there is a canonical embedding
$\skel_{\surfo}\colon\EG(\surfo)\to\FQuad{}{\surfo}$.
As a corollary, connected components of $\EG(\surfo)$ are in one-to-one correspondence with
connected components of $\FQuad{}{\surfo}$.
For the rest of the paper, we will fix a triangulation $\T$ in $\EG(\surfo)$
and study the connected component $\EGT(\surfo)$
and the corresponding connected component $\FQuad{\T}{\surfo}$.
For later use, we record the corresponding lemma,
showing that $\EGT(\surfo)$ is a \danger{skeleton} for $\FQuad{\T}{\surfo}$.

\begin{lemma}\label{lem:f.d.1}
There is a (unique up to homotopy) canonical embedding
\begin{equation}\label{eq:Sembed}
\skel_{\surfo}\colon\EGT(\surfo)\to\FQuad{\T}{\surfo}
\end{equation}
whose image is dual to $B_2(\surfo)$
and which induces a surjective map
\[\skel_*\colon \pi_1\EGT(\surfo)\to\pi_1\FQuad{\T}{\surfo}.\]
\end{lemma}

\subsection{Stability conditions on triangulated categories}\label{sec:stab}
A \emph{stability function} (also called \emph{central charge}) on an abelian category $\hua{C}$
is a group homomorphism $Z\colon \Grot(\hua{C})\to\CC$ such that for any object $0\neq M\in\hua{C}$,
we have
\[
 Z(M) = m_Z(M) \cdot \exp (\mai \pi \pha_Z(M))
\]
for some \emph{mass} $m_Z(M) \in \RR_{>0}$ and \emph{phase} $\pha_Z(M)\in (0,1]$,
i.e. $Z(M)$ lies in the half-closed upper half-plane $\UHP\cup\RR_{<0}$.

We say that a non-zero object $M\in\hua{C}$ is \emph{semistable} (resp. \emph{stable})
with respect to $Z$ if every
non-zero proper subobject $L \subseteq M$ satisfies $\pha_Z(L) \leq \pha_Z(M)$
(resp. $\pha_Z(L)<\pha_Z(M)$).
Further, we say that a stability function $Z$ on $\hua{C}$ satisfies the \emph{HN-property},
if every non-zero object $M$ in $\hua{C}$ has a finite filtration
\[
 0=M_0 \subseteq M_1 \subseteq \cdots \subseteq M_k=M,
\]
whose factors $L_i = M_i/M_{i-1}$ are $Z$-semistable with
phases $\pha_Z(L_1)>\cdots>\pha_Z(L_k)$.

\begin{definition}\cite[Prop.~5.3]{B1}\label{pp:ss}
A \emph{stability condition} $\sigma$ on a triangulated category $\D$
consists of a heart $\h$ and
a stability function $Z$ on $\h$ with the HN-property.
We will say $\sigma$ is supported on $\h$.
Equivalently, a stability condition $\sigma=(Z,\h)$ consists of
a central charge $Z\colon\Grot(\D)\to\CC$
and a slicing $\slicing=\{\slicing(\pha) \mid \pha \in\RR\}$, such that
\begin{itemize}
\item $\Hom(M_1,M_2)=0$ if $M_i\in\slicing(\pha_i$) with $\pha_1>\pha_2$.
\item $\slicing(\pha+1)=\slicing(\pha)[1]$.
\item $Z(M)=m_M \cdot \exp (\mai \pi \pha)$ for some $m_M\in\RR_{>0}$, if $M\in\slicing(\pha)$.
\item any object admits an HN-filtration (see \cite[Def~5.1]{B1} for details).
\end{itemize}
\end{definition}

\AKedit{ The important fact proved in \cite{B1} is that,
for a triangulated category $\D$ whose Grothendieck group $K(\D)$ has finite rank,
the stability conditions (that satisfy the so-called support property) form a complex manifold,
denoted by $\Stab\D$, with dimension $\rank K(\D)$.}

There is also a canonical $\CC$-action on $\Stab\D$, which is given by (in terms of center charge and slicing)
\begin{equation}\label{eq:C-act}
  \theta \cdot (Z,\slicing)=(e^{-\mai\pi\theta}\cdot Z,\slicing_{-\operatorname{Re}(\theta)}),
  \quad \theta\in\CC
\end{equation}
where $\slicing_{x}(\pha)=\slicing(x+\pha)$, for some $x,\pha\in\RR$.

When $\Gamma$ is the Ginzburg dg algebra associated to some quiver with potential,
we denote by $\Stap(\Gamma)$ the principal component of $\Stab\D_{fd}(\Gamma)$,
that is, the connected component that includes those stability conditions supported on the canonical heart $\h_\Gamma$.

Recall (e.g. from \cite{B1,Q2,QW}) the following well-known cell structure of $\Stap(\Gamma)$.
Let $\h$ be a \emph{finite} heart in $\D_{fd}(\Gamma)$,
that is, a length category with finitely many simples.
Then the cell $\cub(\h)$ is the subspace in $\Stab\D_{fd}(\Gamma)$ consisting of
those stability conditions $\sigma=(\h,Z)$ whose
the central charge $Z$ takes values in $\UHP$.
The coordinates $\{Z(S) \mid S\in \Sim\h\}$ give an isomorphism $\cub(\h)\isom\UHP^{\Sim\h}$.
\Note{Recall from \S\ref{sec:heart} that $\Sim\h$ denotes the set of simples in $\h$.}
For each tilting $\h\to\h'=\tilt{\h}{\sharp}{S}$,
there is co-dimensional one wall where $Z(S)\in \RR_{>0}$,
\[
 \partial^\sharp_S\cub(\h) = \partial^\flat_{S[1]}\cub(\h')
 = \overline{\cub(\h)}\cap\overline{\cub(\h')}.
\]
Let
\[
    B_0(\Gamma)=\bigcup_{\h\in\EGp(\Gamma)} \cub(\h)
\]
and
\[
    B_2(\Gamma)=B_0(\Gamma)\cup\bigcup_{\h\in\EGp(\Gamma),\, S\in\Sim\h} \partial^\sharp_S\cub(\h).
\]
Note that $B_2(\Gamma)$ is connected as a space,
because $\EGp(\Gamma)$ is connected as a graph (by definition).
Thus $B_2(\Gamma)\subseteq \Stap(\Gamma)$ and, in particular,
$\cub(\h)\subseteq\Stap(\Gamma)$, for all $\h$ in $\EGp(\Gamma)$.

As in Lemma~\ref{lem:iota_surf} for framed quadratic differentials,
the cell structure determines an embedding of the exchange graph as a \danger{skeleton}
for the space of stability conditions, uniquely up to homotopy,
\begin{gather}\label{eq:embed}
    \skel_\Gamma:\EGp(\Gamma) \to \Stap(\Gamma),
\end{gather}
so that the image is dual to $B_2(\Gamma)$.
\note{
A more detailed discussion of this embedding can be found in \cite[\S3]{Q2}.
}

\begin{remark}
In a `finite type' component of spaces of stability conditions,
the gluing structure has been studied in more details
in \cite{QW}, where it is shown that such a component is always contractible.
Our case is rather `tame type', similar to \cite{HKK}.
In such a case, there is a surjection
$\skel_*:\pi_1\EGp(\Gamma) \to \pi_1\Stap(\Gamma)$, as proved in Lemma~\ref{lem:f.d.1}.
However, it is not clear whether this still holds in the general case.
\end{remark}

\subsection{Bridgeland-Smith theory}
In this section, we follow the theory developed in \cite{BS} that relates quadratic differentials and stability conditions and adapt it to our case.
Recall that they proved the following.
\begin{theorem}\cite[Thm.~11.2 and Thm.~9.9]{BS}\label{thm:BS}
There is an isomorphism of complex manifolds
\begin{equation}\label{eq:BS}
\Stap(\Gamma_\RT)/\Autp(\Gamma_\RT)\cong\Quad(\surf).
\end{equation}
In particular, we have the induced short exact sequence
\begin{equation}\label{eq:BS SES}
1\to\pi_1\Stap(\Gamma_\RT)\to\pi_1\Quad(\surf)\xrightarrow{\rho_*}\Autp(\Gamma_\RT)\to1.
\end{equation}
Furthermore, there is another short exact sequence
\begin{equation}\label{eq:BS 9.9}
1\to\ST(\Gamma_\RT)\to\Autp(\Gamma_\RT)\to\MCG(\surf)\to1.
\end{equation}
\end{theorem}

Here, $\Autp(\Gamma_\RT)=\Aut^*\D_{fd}(\Gamma_\RT)/\Aut_0$,
where $\Aut^*\D_{fd}(\Gamma_\RT)$ is the subgroup of the auto-equivalence group of $\D_{fd}(\Gamma_\RT)$ that preserves the principal component $\Stap(\Gamma_\RT)$ and $\Aut_0$ is its subgroup that acts trivially on $\Stap(\Gamma_\RT)$.
\AKedit{
Note that, in \eqref{eq:BS 9.9}, the map $\ST(\Gamma_\RT)\to \Autp(\Gamma_\RT)$
is injective, because $\ST(\Gamma_\RT)$ acts faithfully on $\EGp(\Gamma_\RT)$
and hence on $\Stap(\Gamma_\RT)$ (see Remark~\ref{rem:ST0triv}).
Thus, we also have
$\Autp(\Gamma_\RT)=\Aut^*\D_{fd}(\Gamma_\RT)$.
}

We prove the following upgraded version of Theorem~\ref{thm:BS} using Theorem~\ref{cor:QQ},
which relates the exchange graphs which are the \danger{skeleta} of the respective spaces.
Note that $\Quad(\surf)$ here is denoted $\Quad_\heartsuit(\surf,\M)$ in \cite{BS} and is really an orbifold.
It is realised there as an orbifold quotient of a certain space of framed quadratic differentials,
but with a different sort of framing to ours.
Here $\Quad(\surf)$ is naturally the orbifold quotient of $\FQuad{}{\surf}$ by $\MCG(\surf)$,
or of $\FQuad{}{\surfo}$ by $\MCG(\surfo)$; cf. \eqref{eq:quad0}.

\begin{theorem}\label{thm:Stab=Quad}
Let $\T$ be a triangulation of a decorated marked surface $\surfo$, with underlying marked surface $\surf$ and triangulation $\RT$ and let $\Gamma_\RT$ be the associated Ginzburg dg algebra.
As complex manifolds, there are isomorphisms
\begin{gather}\label{eq:stab/st=quad}
\kappa_\RT\colon\Stap(\Gamma_\RT)/\ST(\Gamma_\RT)\cong\FQuad{}{\surf},\\
     \label{eq:stab=quad}
     \kappa_\T\colon\Stap(\Gamma_\RT)\cong\FQuad{\T}{\surfo}.
\end{gather}
Furthermore the second isomorphism \eqref{eq:stab=quad} is compatible with the embeddings \eqref{eq:embed} and \eqref{eq:Sembed}, given the isomorphism \eqref{eq:EGp=EGT} \note{between the exchange graphs of triangulations and exchange graphs of hearts}.
\end{theorem}

\begin{proof}
We sketch the construction/proof of these isomorphisms as
\danger{a straightforward generalisation of the arguments of Bridgeland-Smith \cite[\S11]{BS}.}
\AKedit{The argument for \eqref{eq:stab/st=quad} is almost identical to that for \eqref{eq:BS}
and indeed the two are directly related via \eqref{eq:BS 9.9}.}
Hence we only need to discuss the case \eqref{eq:stab=quad}.

\AKedit{We start from
the isomorphism of exchange graphs $\EGT(\surfo)\isom\EGp(\Gamma_\RT)$
as in Theorem~\ref{cor:QQ}.}
At each vertex of these graphs,
i.e. a triangulation $\T_0$ in $\EGT(\surfo)$ and a heart $\h_0$ in $\EGp(\Gamma_\RT)$,
the open arcs $\gamma_i$ in $\T_0$
are in one-to-one correspondence to the simples $S_i$ of $\h_0$
\AKedit{(see \cite[Thm~6.6]{QQ}).}
Then we can identify any $\surfo$-framed quadratic differential
$(\xx,\phi,\psi) \in \cub(\T_0)$ with a unique stability condition $\sigma\in\cub(\h_0)$.
More precisely,
\note{
consider the spectral cover $\Sigma_\xx$ of $\xx$,
that is, the double cover associated to the quadratic differential $\phi$,
branched at the zeroes and odd order poles,
and the corresponding abelian differential $\sqrt{\phi}$ on $\Sigma_\xx$.
}
Then the central charge of $Z$ of $\sigma$ is given by the formula
\AKedit{(cf. \cite[\S1.1]{BS})}
\begin{equation}\label{eq:period}
     Z(S_i)=\int_{\widetilde{\eta_i}} \sqrt{\phi},
\end{equation}
where $\widetilde{\eta_i}$ is the simple closed curve on $\Sigma_\xx$ that
covers to the closed arc $\eta_i=\psi(\gamma_i^*)$ in the dual graph $\T_0^*$.
Thus, we have an isomorphism $\kappa_\T\colon \cub(\T_0)\to \cub(\h_0)$
and hence $\kappa_\T\colon B_0(\surfo)\to B_0(\Gamma_\RT)$.
Following the idea in the proof of \cite[Prop.~11.3]{BS},
\note{i.e. using the $\CC$-action on both sides to perturb things a little,}
$\kappa_\T$ can be extended to \AKedit{the whole of} $\FQuad{\T}{\surfo}$ and we eventually obtain the isomorphism as required.
\end{proof}

We summarize the results in this section in Figure~\ref{fig:quad-stab},
where the red/blue middle horizontal maps are isomorphisms between spaces/groups.

\begin{figure}[ht]\centering
\begin{tikzpicture}[xscale=1.8,yscale=1.2]
\draw (-1,1.5) node (s) {$\FQuad{}{\surfo}$}
   (0,3) node (s0) {$\FQuad{\T}{\surfo}$}
   (1,1.5) node (s1) {$\FQuad{}{\surf}$}
   (0,0) node (s2) {$\Quad(\surf)$};
\draw [->, font=\scriptsize,>=stealth]
   (s0) edge node [right] (bt) {$\;\BT(\surfo)$} (s1) 
   (s1) edge node [right] (m1) {$\;\MCG(\surf)$} (s2); \draw (4,3) node (t0) {$\Stap(\Gamma_\RT)$}
   (3,1.5) node (t1) {$\Stap(\Gamma_\RT)/\ST(\Gamma_\RT)$}
   (4,0) node (t2) {$\Stap(\Gamma_\RT)/\Autp(\Gamma_\RT)$};
\draw[font=\scriptsize,->,>=stealth] (t0)
     edge node[left] (st) {$\ST(\Gamma_\RT)\;$} (t1)
      (t1) edge node [left] (m2) {$\;\Autp/\ST\;$} (t2); \foreach \j in {1,2,0}{\draw[->,,>=stealth,red!80,thick](s\j)edge(t\j);}
\draw[->,font=\scriptsize,>=stealth](bt)edge[blue] node[above,white] {$\iota_\T$} node[below]{}(st);
\draw[font=\scriptsize]
     (m1)edge[blue,->,>=stealth](m2)
     (t0)edge[bend left,->,>=stealth] node[right]
{$\Autp(\Gamma_\RT)$}(t2)
      (s)edge[->,>=stealth] node[above] {$\SBr(\surfo)$} (s1)
      (s)edge[->,>=stealth] node[below left] {$\MCG(\surfo)$} (s2); \draw(s0)edge[right hook-stealth](s); \end{tikzpicture}
\caption{Comparing Quad and Stab.}
\label{fig:quad-stab}
\end{figure}
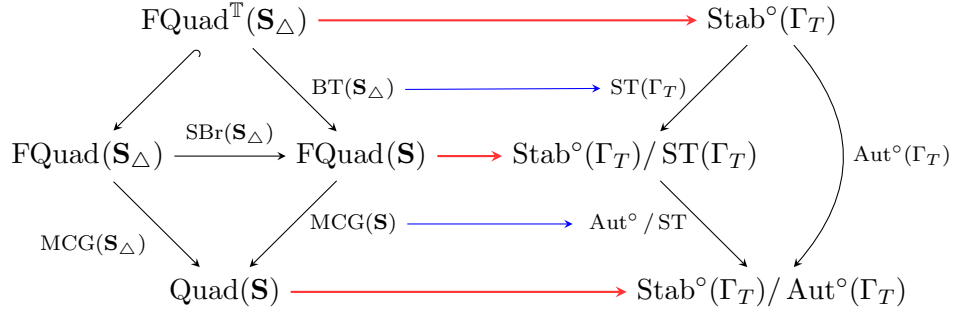

\subsection{Simply connectedness}\label{sec:simply}
We finish the paper by achieving the main aim of the series, namely proving that $\Stap(\Gamma_\RT)$ is simply connected.
To do this we need to upgrade \cite[Lemma~4.6]{Q2}, i.e. that squares and pentagons can be contracted in $\Stap(Q)$ for a Dynkin quiver $Q$, to our context and include the fact that hexagons can be contracted.


\begin{proposition}\label{pp:hex}
Let $\Gamma=\Gamma(Q,W)$ be the Ginzburg algebra of a non-degenerate quiver with potential.
Let $\h$ be a heart in $\EGp(\Gamma)$ with simples $S_i$ and $S_j$ satisfying
$\Ext^1(S_i,S_j)=0$, so that there is a hexagon $H$ in $\EGp$, as in Lemma~\ref{lem:456}.
Then the image, in $\Stap(\Gamma)$, of $H$ under $\skel_\Gamma$, as in \eqref{eq:embed},
is trivial in $\pi_1\Stap(\Gamma)$.
Similarly, the images of pentagons and squares, as in Remark~\ref{rem:456},
are also trivial loops.
\end{proposition}

\begin{proof}
The hexagon $H$ consists of the following six hearts, linked by tilts in the simples labelling the edges,
\begin{gather}\label{eq:hexagon}
\begin{tikzpicture}[arrow/.style={->,>=stealth},rotate=-90,xscale=1,yscale=2]]
\draw (0,0) node (t1) {$(\h;S_i,S_j)$}
    (0,2) node (t2) {$(\h_j;S_i,S_j[1])$}
    (2,3) node (t3) {$(\h_{ji};S_i[1],S_j[1])$}
    (4,2) node (t4) {$(\h_{jii};S_i[2],T_j[1])$}
    (4,0) node (t5) {$(\h_{ii};S_i[2],T_j)$}
    (2,-1) node (t6) {$(\h_i;S_i[1],T_j)$};
\draw [->, font=\scriptsize]
    (t1)edge node[above]{$S_j$} (t2)
    (t2)edge node[right]{$S_i$} (t3)
    (t5)edge node[below]{$T_j$} (t4)
    (t6)edge node[left]{$S_i[1]$} (t5)
    (t1)edge node[left]{$S_i$} (t6)
    (t3)edge node[right]{$S_i[1]$} (t4);
\end{tikzpicture}
\end{gather}
where $\h_i=\tilt{\h}{\sharp}{S_i}$, $\h_j=\tilt{\h}{\sharp}{S_j}$, $\h_{ji}=\tilt{(\h_j)}{\sharp}{S_i}$,
$\h_{ii}=\twi_{S_i}^{-1}(\h)$, $\h_{jii}=\twi_{S_i}^{-1}(\h_j)$ and $T_j=\twi_{S_i}^{-1}(S_j)$.
Note that $\h_{ji}$ is also the forward tilt of $\h$
with respect to the torsion pair whose torsion free part is $\< S_i,S_j\>$, that is,
the extension-closed subcategory generated by $S_i$ and $S_j$.

Recall from \cite[Prop.~5.4]{KQ}, that the simples in $\tilt{\h}{\sharp}{S}$
are $S[1]$ together with $\tilt{\psi}{\sharp}{S}(X)$, for $X\in\Sim\h \setminus \{S\}$,
and the simples in $\tilt{\h}{\flat}{S}$ are $S[-1]$ together with $\tilt{\psi}{\flat}{S}(X)$,
for $X\in\Sim\h \setminus \{S\}$, where $\tilt{\psi}{\sharp}{S}(X)$ and $\tilt{\psi}{\flat}{S}(X)$ appear in the following exact triangles
\begin{gather}
  X[-1] \to S \otimes\Ext^1(X, S)^* \to \tilt{\psi}{\sharp}{S}(X) \to X \label{eq:triang-for-sharp} \\
  X \to \tilt{\psi}{\flat}{S}(X) \to S \otimes\Ext^1(S, X) \to X[1] \label{eq:triang-for-flat}
\end{gather}
or, by \cite[Rem.~7.1]{KQ}, are given by the following formulae
\begin{gather}
\tilt{\psi}{\sharp}{S}(X)= \begin{cases} \twi_{S}^{-1}(X) & \text{if $\Ext^1(X, S)\neq 0$,}\\
      X & \text{if $\Ext^1(X, S)= 0$,} \end{cases}  \label{eq:formula-for-sharp}\\
      \tilt{\psi}{\flat}{S}(X)= \begin{cases} \twi_{S}(X) & \text{if $\Ext^1(S, X)\neq 0$,}\\
      X & \text{if $\Ext^1(S, X)= 0$.} \end{cases} \label{eq:formula-for-flat}
\end{gather}
For example, $\tilt{\psi}{\sharp}{S_j}(S_i)=S_i$ since $\Ext^1(S_i, S_j)=0$,
while $\tilt{\psi}{\sharp}{S_i}(S_j)=T_j$.
Similar calculations show that
\begin{itemize}
\item $S_i$ and $S_j[1]$ are simples in $\h_j$, while $S_i[1]$ and $T_j$ are simples in $\h_i$,
\item $S_i[1]$ and $S_j[1]$ are simples in $\h_{ji}$,
\item $S_i[2]$ and $T_j$ are simples in $\h_{ii}$, while $S_i[2]$ and $T_j[1]$ are simples in $\h_{jii}$.
\end{itemize}
These, along with $S_i$ and $S_j$ in $\h$, are the key simples in the arguments that follow.

Recall from \S\ref{sec:stab} that any heart $\h_0$ gives a cell $\cub(\h_0)$ in $\Stap(\Gamma)$
that is isomorphic to $\UHP^{\numarc}$ using the coordinates $\{Z(S) \mid S\in\Sim\h_0 \}$.
Furthermore, each simple tilting gives a wall between the cells of the corresponding hearts.

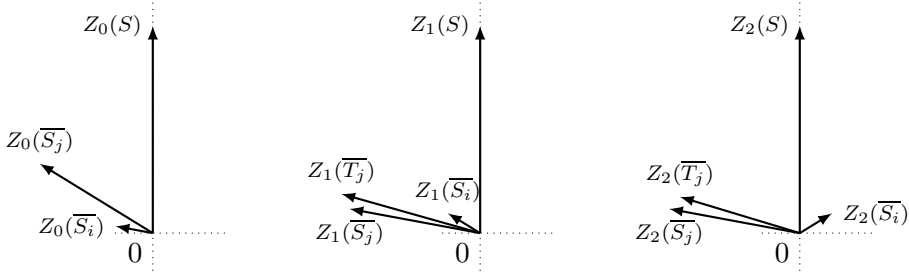
\begin{figure}[ht]\centering
\newcommand{\sh}[1]{\overline{#1}}
\tikzset{axes/.style={dotted}, zarrow/.style={thick,->,>=latex,font=\scriptsize}}
\begin{tikzpicture}[scale=.5]
\draw[axes] (-1,0) -- (2,0) (0,-1) -- (0,6.2);
\draw (0,0) node[below left]{$0$};
\draw[zarrow] (0,0) -- (180/2:5.5)    node[left]{$Z_0(S)$};
\draw[zarrow] (0,0) -- (180-180/17:1)   node[left]{$Z_0(\sh{S_i})$};
\draw[zarrow] (0,0) -- (180-3*180/17:3.5) node[above]{$Z_0(\sh{S_j})$};
\end{tikzpicture}
\qquad
\begin{tikzpicture}[scale=.5]
\draw[axes] (-1,0) -- (2,0) (0,-1) -- (0,6.2);
\draw (0,0) node[below left]{$0$};
\draw[zarrow] (0,0) -- (180/2:5.5)    node[left]{$Z_1(S)$};
\draw[zarrow] (0,0) -- (180-3*180/17:1)   node[above]{$Z_1(\sh{S_i})$};
\draw[zarrow] (0,0) -- (180-180/17:3.5) node[below]{$Z_1(\sh{S_j})$};
\draw[zarrow] (0,0) -- (180-1.5*180/17:3.8) node[above]{$Z_1(\sh{T_j})$};
\end{tikzpicture}
\qquad
\begin{tikzpicture}[scale=.5]
\draw[axes] (-1,0) -- (2,0) (0,-1) -- (0,6.2);
\draw (0,0) node[below left]{$0$};
\draw[zarrow] (0,0) -- (180/2:5.5)    node[left]{$Z_2(S)$};
\draw[zarrow] (0,0) -- (3*180/17:1)   node[right]{$Z_2(\sh{S_i})$};
\draw[zarrow] (0,0) -- (180-180/17:3.5) node[below]{$Z_2(\sh{S_j})$};
\draw[zarrow] (0,0) -- (180-1.6*180/17:3.3) node[above]{$Z_2(\sh{T_j})$};
\end{tikzpicture}
\caption{Stability conditions $\sigma_0,\sigma_1$ and $\sigma_2$, denoting $X[1]$ by $\sh{X}$. }
\label{fig:sc}
\end{figure}

For $k=0,1,2$, define the stability conditions $\sigma_k=(Z_k,\slicing_k) \in\cub(\h_{ji})$ by
\[
 Z_k(S)=M\cdot\exp \textfrac{\mai\pi}{2}, \quad  \text{for $S\in\Sim\h_{ji}\setminus \{S_i[1],S_j[1]\}$,}
\]
and
\[\begin{tabular}{ L L }
Z_0(S_i[1]) = m\cdot\exp \mai\pi(1-\delta), &
Z_0(S_j[1]) = \exp \mai\pi(1-3\delta),\\
Z_1(S_i[1]) = m\cdot\exp \mai\pi(1-3\delta),&
Z_1(S_j[1]) = \exp\mai\pi(1-\delta),\\
Z_2(S_i[1]) = m\cdot\exp \mai\pi\delta,&
Z_2(S_j[1]) = \exp \mai\pi(1-\delta),
\end{tabular}\]
for some small $\delta,m\in\RR_{>0}$ and large $M\in\RR_{>0}$
(see Figure~\ref{fig:sc}).
By \eqref{eq:triang-for-sharp},
\[
    Z_k(T_j[1]) = \dim\Ext^1(S_j,S_i)\cdot Z_k(S_i[1])+Z_k(S_j[1]),
\]
Suppose that $T_j[1]\in\slicing_k(1-\epsilon_k)$, for $k=1,2$.
Then, by choosing $m, \delta$ small enough, we can assume
\begin{gather}\label{eq:2}
    \delta<\epsilon_k<2\delta.
\end{gather}
By also choosing $M$ large enough, the phases of simples in
the hearts $\h$, $\h_i$, $\h_j$ and $\h_{ji}$, other than those of $S_i,S_j,T_j$ and their shifts,
will be very close to $1/2$, for any stability condition in
$[-4\delta,0]\cdot\sigma_k$, when $k=0,1$ or $2$.

\begin{remark}\label{rem:criterion}
Let $\h$ be a finite heart and $\sigma=(Z,\slicing)$ a stability condition in $\cub(\h)$.
By the definition of the $\CC$-action, the heart of $(-\pha)\cdot\sigma$ is
$\h'=\slicing(-\pha,1-\pha]$.
Moreover, $\h[-1]\le\h'\le\h$, in the sense that
the corresponding t-structures satisfy
\[
    \slicing(-1,\infty)  \supset \slicing(-\pha,\infty) \supset \slicing(0,\infty).
\]
Thus (cf. \cite[Rem.~3.3]{KQ}) we have the following criterion.

The heart $\h'=\slicing(-\pha,1-\pha]$ is the backward tilt of
$\h=\slicing(0,1]$ with respect to the torsion pair $\<\torfree,\torsion\>$,
where
\[
 \torfree = \h\cap\h' = \slicing(0,1-\pha]\quad\text{and}\quad
 \torsion = \h\cap\h'[1] = \slicing(1-\pha,1] .
\]
Similarly the heart of $\pha\cdot\sigma$ is $\slicing(\pha,\pha+1]$, which is
the forward tilt of $\h$ with respect to the torsion pair $\<\slicing(0,\pha], \slicing(\pha,1] \>$.
\end{remark}

We can use Remark~\ref{rem:criterion}
to find the hearts of stability conditions obtained from $\sigma_k$ by the $\CC$-action.
For example, the heart $\h'$ of $(-\pha)\cdot\sigma_0$
is the backward tilt of $\h_{ji}$
with respect to the torsion pair whose torsion part $\torsion$ is $\hua{P}_0(1-\pha,1]$.
If $0\leq \pha<\delta$, then $\torsion$ is trivial and $\h'=\h_{ji}$.
If $\delta\leq \pha < 3\delta$, then $\torsion$ is generated by $S_i[1]$ and $\h'=\h_{j}$.
If $3\delta\leq \pha < 1/2$, then $\torsion$ is generated by $S_i[1]$ and $S_j[1]$ and $\h'=\h$.
In other words, we have the following.
\begin{itemize}
\item $(-\delta,0]\cdot\sigma_0$ is in $\cub(\h_{ji})$,
\item $(-\delta)\cdot\sigma_0$ is in the wall
    $\partial^\flat_{S_i[1]}\cub(\h_{ji})=\partial^\sharp_{S_i}\cub(\h_j)$,
\item $(-3\delta,-\delta)\cdot\sigma_0$ is in $\cub(\h_j)$,
\item $(-3\delta)\cdot\sigma_0$ is in the wall $\partial^\flat_{S_j[1]}\cub(\h_j)=\partial^\sharp_{S_j}\cub(\h)$,
\item $[-4\delta,-3\delta)\cdot\sigma_0$ is in $\cub(\h_{})$,
\end{itemize}
Thus
\begin{itemize}
\item[(1)]
    $[-4\delta,0]\cdot\sigma_0$ is homotopic to $\skel_\Gamma(\h\to\h_j\to\h_{ji})$.
\end{itemize}
Similarly, the heart of $(-4\delta)\cdot\sigma_1$ is $\h$
and we may apply the $\CC$-action in a positive direction from there to deduce that
\begin{itemize}
\item $[-4\delta,-3\delta)\cdot\sigma_1$ is in $\cub(\h)$,
\item $(-3\delta)\cdot\sigma_1$ is in the wall $\partial^\flat_{S_i[1]}\cub(\h_i)=\partial^\sharp_{S_i}\cub(\h)$,
\item $(-3\delta,-\epsilon_1)\cdot\sigma_1$ is in $\cub(\h_{i})$.
\end{itemize}
Thus, as $\epsilon_1< 2\delta$,
\begin{itemize}
\item[(2)]
$[-4\delta,-2\delta]\cdot\sigma_1$ is homotopic to $\skel_\Gamma(\h\to\h_i)$.
\end{itemize}
Indeed, since we can observe that $\h_{ji}$ is the forward tilt of $\h_i$ with respect to the torsion theory with
torsion-free part $\torfree=\<T_j,S_j\>$,
we can apply Remark~\ref{rem:criterion} to see that
\[
  \sigma_2'\ := (-2\delta)\cdot\sigma_1\quad\text{and}\quad\sigma_1' :=(-2\delta)\cdot\sigma_2
\]
are both in $\cub(\h_i)$.
Furthermore, we can choose $\sigma_0'\in\cub(\h_{i})$ given by
\begin{align*}
    Z_0'(S) &= M \cdot\exp \textfrac{\mai\pi}{2},\quad  \text{for $S\in\Sim\h_{i}\setminus \{S_i[1],T_j\}$,}\\
    Z_0'(S_i[1]) &= m \cdot\exp \mai\pi\delta ,\\
    Z_0'(T_j) &= \exp 3\mai\pi\delta .
\end{align*}

The scheme of how the proof is completed is depicted in Figure~\ref{fig:strip}.
The goal is to produce a curve that is homotopic to the hexagon $H$ and
that is also the boundary of the union of three contiguous strips in $\Stap(\Gamma)$,
showing that this curve is contractible.

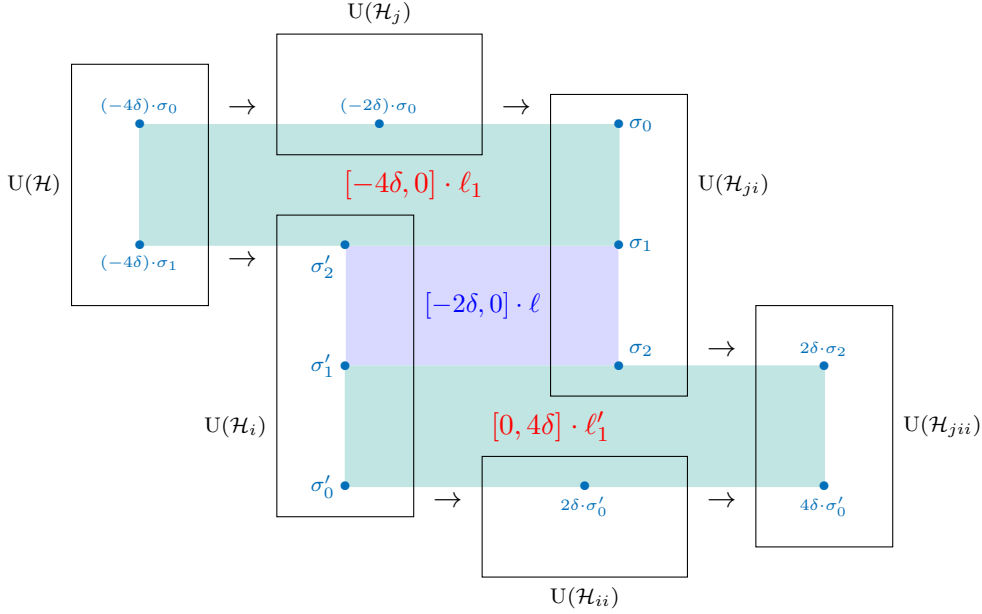
\begin{figure}[h]
\begin{tikzpicture}[yscale=.8,xscale=.9]
\draw[white,fill=blue!15,font=\scriptsize]
    (5,4)rectangle(9,6);
\draw[Emerald!23,fill=Emerald!23,font=\scriptsize]
    (2,8)rectangle(9,6);
\draw[Emerald!23,fill=Emerald!23,font=\scriptsize]
    (5,4)rectangle(12,2);
\draw[NavyBlue,font=\scriptsize](9,8)node{$\bullet$}node[right]{$\sigma_0$}
    (9,6)node{$\bullet$}node[right]{$\sigma_1$}
    (9,4)node{$\bullet$}node[above right]{$\sigma_2$}
    (5,6)node{$\bullet$}node[below left]{$\sigma_2'$}
    (5,4)node{$\bullet$}node[left]{$\sigma_1'$}
    (5,2)node{$\bullet$}node[left]{$\sigma_0'$}
    (5.5,8)node{$\bullet$}node[above]{$_{(-2\delta)\cdot\sigma_0}$}
    (2,8)node{$\bullet$}node[above]{$_{(-4\delta)\cdot\sigma_0}$}
    (2,6)node{$\bullet$}node[below]{$_{(-4\delta)\cdot\sigma_1}$}
    (12,4)node{$\bullet$}node[above]{$_{2\delta\cdot\sigma_2}$}

    (8.5,2)node{$\bullet$}node[below]{$_{2\delta\cdot\sigma_0'}$}
    (12,2)node{$\bullet$}node[below]{$_{4\delta\cdot\sigma_0'}$};

\draw[font=\scriptsize](1,5)rectangle(3,9) (1,7)node[left]{$\cub(\h)$};
\draw[font=\scriptsize](13,5)rectangle(11,1) (13,3)node[right]{$\cub(\h_{jii})$};
\draw[font=\scriptsize](4,1.5)rectangle(6,7-.5) (10,7)node[right]{$\cub(\h_{ji})$};
\draw[font=\scriptsize](10,9-.5)rectangle(8,3.5) (4,3)node[left]{$\cub(\h_{i})$};
\draw[font=\scriptsize](4,7+.5)rectangle(7,9.5) (5.5,9.5)node[above]{$\cub(\h_{j})$};
\draw[font=\scriptsize](10,2.5)rectangle(7,.5) (8.5,.5)node[below]{$\cub(\h_{ii})$};
\draw[](3.5,8)node[above]{$\rightarrow$}(7.5,8)node[above]{$\rightarrow$}
    (3.5,6)node[below]{$\rightarrow$}
    (10.5,4)node[above]{$\rightarrow$};
\draw[](10.5,2)node[below]{$\rightarrow$}(6.5,2)node[below]{$\rightarrow$};
\draw[red] (6,7) node{$[-4\delta,0] \cdot \ell_1$};
\draw[red] (8,3) node{$[0,4\delta] \cdot \ell_1'$};
\draw[blue,font=\small] (7,5) node{$[-2\delta,0] \cdot \ell$};
\end{tikzpicture}
\caption{Filling in the hexagon $H$ with three strips in $\Stap(\Gamma)$}\label{fig:strip}
\end{figure}

The vertical edges of the violet strip in the middle of Figure~\ref{fig:strip} arise as follows:
\begin{itemize}
\item[(3)] there is a line (segment) $\ell$ in $\cub(\h_{ji})$ that connects $\sigma_1$ to $\sigma_2$
by varying just the coordinate $Z(S_i[1])$,
\item[(4)] the line $(-2\delta)\cdot \ell$ 
is in $\cub(\h_i)$, for the same reason that $\sigma_2'$ and $\sigma_1'$ are.
\end{itemize}
Similarly, we can produce the vertical edges of the outer green strips in Figure~\ref{fig:strip}:
\begin{itemize}
\item[(5)]
there is a line $\ell_1$ in $\cub(\h_{ji})$ connecting $\sigma_0$ to $\sigma_1$
by varying just the coordinates $Z(S_i[1])$ and $Z(S_j[1])$,
such that the line $(-4\delta)\cdot \ell_1$ is in $\cub(\h)$;
\item[(6)]
there is a line $\ell_1'$ in $\cub(\h_{i})$ connecting $\sigma_1'$ to $\sigma_0'$
by varying just the coordinates $Z(S_i[1])$ and $Z(T_j)$,
such that the line $4\delta\cdot \ell_1'$ is in $\cub(\h_{jii})$.
\end{itemize}
By construction, the three strips $[-4\delta,0] \cdot \ell_1$, $[-2\delta,0] \cdot \ell$ and
$[0,4\delta] \cdot \ell_1'$ are contiguous in the way indicated in Figure~\ref{fig:strip}.
Finally, similar to (1) and (2) above,
noting that $\h_{jii}$ is the forward tilt of $\h_i$
with respect to the torsion pair with torsion-free part $\torfree=\<T_j,S_i[1]\>$,
we have:
\begin{itemize}
\item[(7)] $[0, 4\delta]\cdot\sigma_0'$ is homotopic to $\skel_\Gamma(\h_i\to\h_{ii}\to\h_{jii})$.
\item[(8)] $[2\delta,4\delta]\cdot\sigma_1'$ is homotopic to $\skel_\Gamma(\h_{ji}\to\h_{jii})$.
\end{itemize}
Combining all the facts (1) to (8) above,
we deduce that the boundary loop of the union of the three strips
is homotopic to the image $\skel_\Gamma(H)$ of the hexagon $H$ in \eqref{eq:hexagon}
and hence this is contractible in $\Stap(\Gamma)$.

The cases of squares and pentagons work in the same way as \cite[Lemma~4.6]{Q2},
which uses a simpler analogue of the above argument.
\end{proof}

From this and earlier results we can now conclude.

\begin{theorem}\label{thm:s.c.}
Let $\RT$ be a triangulation of an (unpunctured) marked surface $\surf$ and
$\Gamma_\RT$ be the associated Ginzburg dg algebra.
Then the principal component $\Stap(\Gamma_\RT)$ is simply connected.
\end{theorem}

\begin{proof}
The immediate corollary of Proposition~\ref{pp:hex} is that the map
\begin{gather}\label{eq:EGmap}
    \pi_1\EGp(\Gamma_T) \to \pi_1\Stap(\Gamma_T),
\end{gather}
induced by the embedding \eqref{eq:embed}, factors through a map
\begin{gather}\label{eq:egmap}
    \pi_1\egp(\Gamma_T) \to \pi_1\Stap(\Gamma_T).
\end{gather}

On the other hand, by Theorem~\ref{thm:Stab=Quad},
we know $\Stap(\Gamma_\RT)\cong\FQuad{\T}{\surfo}$,
where $\surfo$ is some decoration of $\surf$ and $\T$ is some triangulation of $\surfo$ lifting $\RT$.
Furthermore, the map \eqref{eq:EGmap} is identified with
\begin{gather}
    \pi_1\EGT(\surfo) \to \pi_1\FQuad{\T}{\surfo}
\end{gather}
which is surjective, by Lemma~\ref{lem:f.d.1}, and so the map \eqref{eq:egmap} is surjective.

Finally, $\egp(\Gamma_T)\cong \egt(\surfo)$, by Theorem~\ref{cor:QQ},
which is simply-connected, by Theorem~\ref{thm:45}.
Hence $\pi_1\Stap(\Gamma_T)$ is trivial, as required.
\end{proof}

\note{
A direct corollary about the space of quadratic differentials is the following.
\begin{corollary}\label{cor:s.c.}
The moduli space $\FQuad{}{\surfo}$ of $\surfo$-framed quadratic differentials
consists of $\SBr(\surfo)/\BT(\surfo)$ many connected components,
each of which is isomorphic to $\FQuad{\T}{\surfo}$ and they are all simply connected.
\end{corollary}
\begin{proof}
The simply connectedness is a direct corollary of Theorem~\ref{thm:s.c.} and \eqref{eq:stab=quad}.
On the other hand, as summarised in Figure~\ref{fig:quad-stab},
the space $\FQuad{}{\surfo}$ is a $\SBr(\surfo)$ covering of $\FQuad{}{\surf}$,
while any component $\FQuad{\T}{\surfo}$ is a $\BT(\surfo)$ covering of $\FQuad{}{\surf}$.
Hence the action of $\SBr(\surfo)$ provides isomorphisms between components and the set of connected components is acted on simply transitively by $\SBr(\surfo)/\BT(\surfo)$.
\end{proof}
}


\end{document}